\documentclass[11pt]{amsart}
\usepackage[letterpaper,margin=1in]{geometry} 
\usepackage[colorlinks,
             linkcolor=blue,
             citecolor=black!75!red,
             pdfproducer={pdfLaTeX},
             pdfpagemode=None,
             bookmarksopen=true
             bookmarksnumbered=true]{hyperref}
\usepackage[matrix,arrow,ps,color,line,curve,frame,all]{xy}

\usepackage{tikz}
\usetikzlibrary{arrows}

\title{Exotic Elliptic Algebras of dimension 4}

\author[Alex Chirvasitu]{Alex Chirvasitu}
\author[S. Paul Smith]{S. Paul Smith
\\
\mbox{ } 
\\
with an appendix by Derek Tomlin}

\address{Department of Mathematics, Box 354350, University of  Washington, Seattle, WA 98195,USA.}

\email{chirva@math.washington.edu, smith@math.washington.edu}



\usepackage{amsmath,amsthm,stmaryrd}
\usepackage{cleveref}

\newtheorem{lemma}{Lemma}[section]
\newtheorem{theorem}[lemma]{Theorem}
\newtheorem{proposition}[lemma]{Proposition}
\newtheorem{corollary}[lemma]{Corollary}

\theoremstyle{definition} 

\newtheorem{definitionnodiamond}[lemma]{Definition}
\newtheorem{examplenodiamond}[lemma]{Example}
\newtheorem{remarknodiamond}[lemma]{Remark}

\newenvironment{definition}{\begin{definitionnodiamond}}{\hfill\ensuremath\blacklozenge\end{definitionnodiamond}}

\numberwithin{equation}{section}

\newcounter{stepofproof}

\crefname{section}{Section}{Sections}
\crefformat{section}{#2Section~#1#3} 
\Crefformat{section}{#2Section~#1#3} 

\crefname{subsection}{}{Subsections}
\crefformat{subsection}{\S#2#1#3} 
\Crefformat{subsection}{\S#2#1#3}

\crefname{definition}{Definition}{Definitions}
\crefformat{definition}{#2Definition~#1#3} 
\Crefformat{definition}{#2Definition~#1#3} 

\crefname{example}{Example}{Examples}
\crefformat{example}{#2Example~#1#3} 
\Crefformat{example}{#2Example~#1#3} 

\crefname{examplenodiamond}{Example}{Examples}
\crefformat{examplenodiamond}{#2Example~#1#3} 
\Crefformat{examplenodiamond}{#2Example~#1#3} 

\crefname{remark}{Remark}{Remarks}
\crefformat{remark}{#2Remark~#1#3} 
\Crefformat{remark}{#2Remark~#1#3} 

\crefname{remarknodiamond}{Remark}{Remarks}
\crefformat{remarknodiamond}{#2Remark~#1#3} 
\Crefformat{remarknodiamond}{#2Remark~#1#3} 

\crefname{convention}{Convention}{Conventions}
\crefformat{convention}{#2Convention~#1#3} 
\Crefformat{convention}{#2Convention~#1#3} 

\crefname{lemma}{Lemma}{Lemmas}
\crefformat{lemma}{#2Lemma~#1#3} 
\Crefformat{lemma}{#2Lemma~#1#3} 

\crefname{proposition}{Proposition}{Propositions}
\crefformat{proposition}{#2Proposition~#1#3} 
\Crefformat{proposition}{#2Proposition~#1#3} 

\crefname{corollary}{Corollary}{Corollaries}
\crefformat{corollary}{#2Corollary~#1#3} 
\Crefformat{corollary}{#2Corollary~#1#3} 

\crefname{theorem}{Theorem}{Theorems}
\crefformat{theorem}{#2Theorem~#1#3} 
\Crefformat{theorem}{#2Theorem~#1#3} 

\crefname{assumption}{Assumption}{Assumptions}
\crefformat{assumption}{#2Assumption~#1#3} 
\Crefformat{assumption}{#2Assumption~#1#3} 

\crefname{equation}{}{}
\crefformat{equation}{(#2#1#3)} 
\Crefformat{equation}{(#2#1#3)}

\crefname{align}{}{}
\crefformat{align}{(#2#1#3)} 
\Crefformat{align}{(#2#1#3)}

\crefname{proofstep}{Step}{Steps}
\crefformat{proofstep}{#2Step~#1#3} 
\Crefformat{proofstep}{#2Step~#1#3}

\newcommand\cat[1]{\textsc{#1}}
\newcommand\define[1]{\emph{#1}}

\newcommand\arXiv[1]{\href{http://arxiv.org/abs/#1}{\nolinkurl{arXiv:#1}}}
\newcommand\MRnumber[1]{\href{http://www.ams.org/mathscinet-getitem?mr=#1}{\nolinkurl{MR#1}}}
\newcommand\DOI[1]{\href{http://dx.doi.org/#1}{\nolinkurl{DOI:#1}}}
\newcommand\MAILTO[1]{\href{mailto:#1}{\nolinkurl{#1}}}

\usepackage{amsfonts,amssymb}

\newcommand\bP{\mathbb P}

\newcommand\bZ{\mathbb Z}

\newcommand\cA{\mathcal A}

\newcommand\cC{\mathcal C}

\newcommand\cL{\mathcal L}
\newcommand\cM{\mathcal M}

\newcommand\cO{\mathcal O}

\newcommand\cV{\mathcal V}

\DeclareMathOperator\End{\cat{End}}
\newcommand\Vect{\cat{Vect}}

\DeclareMathOperator\QCoh{\cat{Qcoh}}

\newcommand\id{\mathrm{id}}

\renewcommand\lim{\varprojlim}

\def\CC{{\mathbb C}}

\def\GG{{\mathbb G}}

\def\LL{{\mathbb L}}

\def\PP{{\mathbb P}}

\def\RR{{\mathbb R}}

\def\ZZ{{\mathbb Z}}
 
\def\fol{{\bar f}}

\def\im{\operatorname {im}}

\def\Ann{\operatorname{Ann}}
\def\Aut{\operatorname{Aut}}

\def\Fdim{{\sf Fdim}}

\newcommand{\GKdim}{\mathrm{GKdim}}

\def\gr{{\sf gr}}

\def\Gr{{\sf Gr}}

\def\Hom{\operatorname{Hom}}

\def\Mod{{\sf Mod}}

\def\Projnc{\operatorname{Proj}_{nc}}

\def\QGr{\operatorname{\sf QGr}}

\def\a{\alpha}
\def\b{\beta}
\def\c{\gamma}
\def\d{\delta}

\def\g{\gamma}
\def\l{\lambda}

\def\s{\sigma}

\def\ve{\varepsilon}

\def\D{\Delta}
\def\G{\Gamma}
\def\L{\Lambda}

\def\fP{{\mathfrak P}}

\def\fsl{{\mathfrak s}{\mathfrak l}}

\def\wtB{{\widetilde{ B}}}

\def\wtQ{{A}}  
\def\wtV{{\widetilde{ V}}}

\def\QCoh{{\sf Qcoh}}
\def\qcoh{{\sf Qcoh}}

\begin{document}

\maketitle

\begin{abstract} 
Let $E$ be an elliptic curve defined over an algebraically closed field $k$ whose characteristic is not 2 or 3. 
Let $\tau$ be a translation automorphism of $E$ that is not of order 2.
 In a previous paper we studied an algebra $\wtQ=\wtQ(E,\tau)$ that depends on this data:
 $\wtQ(E,\tau)=\big(S(E,\tau) \otimes M_2(k)\big)^\G$ where $S(E,\tau)$ is the 4-dimensional Sklyanin algebra associated to $(E,\tau)$, $M_2(k)$ is the ring of $2 \times 2$ matrices over $k$,
 and $\G$ is $(\ZZ/2)\times (\ZZ/2)$ acting in a particular way as automorphisms of $S$ and $M_2(k)$. The action of $\G$ on $S$ is compatible with the translation action of the 2-torsion subgroup $E[2]$ on $E$. 
 Following the ideas and results in papers of Artin-Tate-Van den Bergh, Smith-Stafford, and Levasseur-Smith,  this paper examines the line modules,  point modules, and fat point modules, over $\wtQ$, and their incidence relations. The right context for the results is non-commutative 
 algebraic geometry: we view $\wtQ$ as a homogeneous coordinate ring of a non-commutative analogue of $\PP^3$ that we denote by 
 $\Projnc(\wtQ)$. Point modules and fat point modules determine ``points'' in  $\Projnc(\wtQ)$. Line modules determine ``lines'' in  $\Projnc(\wtQ)$. 
 Line modules for $A$ are in bijection with certain lines in $\PP(\wtQ_1^*) \cong \PP^3$ and 
therefore correspond to the closed points of a certain subscheme $\LL$ of the Grassmannian $\GG(1,3)$. 
 Shelton-Vancliff call $\LL$ the line scheme for $\wtQ$. We show that $\LL$ is  
 the union of 7 reduced and irreducible components, 3 quartic elliptic space curves and 4 plane conics in the ambient Pl\"ucker $\PP^5$, and that $\deg(\LL)=20$. The union of the lines corresponding to the points on each elliptic curve
is an elliptic scroll in $\PP(\wtQ_1^*)$. Thus, the lines on that elliptic scroll are in natural bijection with a corresponding family of line modules for $\wtQ$.
\end{abstract}


\tableofcontents

\section{Introduction}

This paper is a continuation of  \cite{CS15}.

\subsection{}
\label{sect.1.1}
The 3- and 4-dimensional Sklyanin algebras are certain non-commutative graded algebras $S=\CC+S_1+S_2+\cdots$ that have many of the same ring-theoretic and homological properties as the polynomial rings in 3 and 4 variables.  They are among the most interesting algebras that have 
appeared in non-commutative algebraic geometry; their study has stimulated and justified many of the developments in that subject. 
Such a Sklyanin algebra determines, and is determined by, an elliptic curve, $E$,
a translation automorphism, $\tau:E \to E$, and an invertible $\cO_E$-module $\cL$ of degree 3 or 4, respectively. We will denote the algebra by 
$S(E,\tau,\cL)$.\footnote{The isomorphism class of $S(E,\tau,\cL)$ depends only on the degree of $\cL$. Since the case $\deg(\cL)=3$ will only 
appear in \S\ref{sect.1.1}, and only the case $\deg(\cL)=4$ appears from \S\ref{sect.1.3} onward we will simply write $S(E,\tau)$ from \S\ref{sect.1.3} onward.}

The 3-dimensional Sklyanin algebras are, arguably, the most interesting 3-dimensional  Artin-Schelter regular algebras \cite{AS87}. 
In  \cite{ATV1} and \cite{ATV2},  Artin-Tate-Van den Bergh completed the classification of the 3-dimensional  Artin-Schelter algebras having Hilbert series 
$(1-t)^{-3}$. Such algebras are now viewed as homogeneous coordinate rings of non-commutative deformations of the projective plane $\PP^2$.
If $S$ is such an algebra we write $\Projnc(S)$ for the corresponding non-commutative deformation of $\PP^2$.
Van den Bergh \cite{VdB11-NCQ} showed that as $S$ ranges over all 3-dimensional  Artin-Schelter algebras, 
the spaces  $\Projnc(S)$ give {\it all} non-commutative deformations of $\PP^2$.
The precise statement of Van den Bergh's Theorem involves the deformation theory of the category of $\qcoh(\PP^2)$ of quasi-coherent sheaves on $\PP^2$;
the deformation theory of abelian categories was formulated and developed by Lowen and Van den Bergh in  \cite{LVdB05} and \cite{LVdB06}.
In  \cite{VdB11-NCQ}, Van den Bergh also classifies the non-commutative deformations of $\PP^1 \times \PP^1$.

Moving beyond surfaces, it is natural to try to classify 4-dimensional  Artin-Schelter regular algebras and non-commutative deformations of $\PP^3$.
A bewildering variety of such algebras is known. A classification remains a long way off. The 4-dimensional Sklyanin algebras are, again, the most interesting
such algebras (so far) and they have a very rich structure. The representation theory of $S(E,\tau,\cL)$ and, what is almost the same thing, the geometric structure of  $\Projnc\big(S(E,\tau,\cL)\big)$, is governed by the geometry of 
$E$ and $\tau$ when $E$ is embedded as a quartic curve in $\PP\big(H^0(E,\cL)^*\big)$.

We refer the reader to \cite{A90} and  \cite{S94} for overviews of the 3- and 4-dimensional Sklyanin algebras and the spaces $\Projnc(S(E,\tau,\cL))$. 

\subsection{}
\label{sect.1.2}

The significance of the 4-dimensional Sklyanin algebras is amplified by the fact that they arise naturally  in mathematical physics:
Sklyanin discovered them in the context of quantum statistical mechanics and Baxter's elliptic solutions to the Yang-Baxter equation (\cite{Skl82}, \cite{Skl83}) . 
He also observed that ``degenerate'' versions of the 4-dimensional Sklyanin algebras lead to the quantized enveloping algebras $U_q(\fsl_2)$. 
A detailed examination of this degeneration procedure is carried out in \cite{CSW16}. 

Real forms of 4-dimensional Sklyanin algebras  arise from a natural problem  in non-commutative differential geometry. In \cite{CDV02}, \cite{CDV03},
and \cite{CDV08}, Connes and Dubois-Violette define a non-commutative 3-sphere in terms of the cyclic-homology Chern character 
${\rm ch}_*:K_*(\cA) \to {\rm HC}(\cA)$. They show that  non-commutative 3-spheres are ``unit spheres'' in an ambient non-commutative 
analogue of $\RR^4$ that is defined in terms of a C$^*$-algebra whose defining relations are  necessarily 
those of a 4-dimensional Sklyanin algebra.
The finer structure of non-commutative 3-spheres is very rich and involves, as might be expected, elliptic curves, theta functions, and other algebro-geometric data.
The results of Connes and Dubois-Violette in these three papers provide one of the 
few interactions between non-commutative algebraic geometry and non-commutative differential 
geometry. We expect further exploration of this interaction will prove fruitful. 

 \subsection{Notation and conventions}
    \label{sect.1.3}
In \S\S\ref{sect.1.1} and \ref{sect.1.2}, all algebraic objects were assumed to be vector spaces over the field of complex numbers.
From now on we will work over  an algebraically closed field $k$ of characteristic $\ne 2,3$. Always, $E$ will denote an elliptic curve defined over
 $k$, $\tau:E \to E$ will denote a translation automorphism,  $S(E,\tau)$ will denote the 4-dimensional Sklyanin algebra associated to $(E,\tau)$,
 $\G$ will denote the group $(\ZZ/2)\times (\ZZ/2)$, and $M_2(k)$ will denote the ring of  $2 \times 2$ matrices over $k$. The group 
 $\G$  acts in a particular way as automorphisms of $S(E,\tau)$ and $M_2(k)$ (see  \S\S\ref{sect.Gamma.action} and \ref{ssect.quat.basis}).

As explained in \cite{SS92}, there is a natural copy of $E$ embedded as a quartic curve in $\PP(S_1^*) \cong \PP^3$. 
We fix an identity $o$ and a group law on $E$ such that four points on $E$ are coplanar if and only if their sum is $o$. 
We will identify $\tau$ with $\tau(o)$. We write $E[2]$ and $E[4]$ for the 2-torsion and 4-torsion subgroups of $E$.
The action of $\G$ as automorphisms of $S$ induces an action of $\G$ on $\PP(S_1^*)$; that action sends $E$ to itself, 
and its action on $E$ is the same as the translation action of $E[2]$. We often identify $\G$ and $E[2]$.

 \subsection{}
  The algebras in the title of this paper are the invariant subalgebras 
  $$
  A(E,\tau):=\big(S(E,\tau) \otimes M_2(k)\big)^\G.
  $$
A similar operation can be applied to $A(E,\tau)$. By  \cite[Prop. 6.2]{CS15},  $\big(A(E,\tau) \otimes M_2(k)\big)^\G \cong S(E,\tau)$. 
 
These algebras were studied in some detail in the authors' previous paper \cite{CS15}, where they were denoted $\widetilde{Q}(E,\tau)$, and in the Ph.D. thesis and subsequent paper of Andrew Davies, \cite{Davies16}  and  \cite{Davies-arXiv}. Those papers show that $A(E,\tau)$ has the same excellent homological properties that $S(E,\tau)$ has. For example, $A(E,\tau)$ and $S(E,\tau)$ are, like the polynomial ring in 4 variables, noetherian domains with Hilbert series $(1-t)^{-4}$, Koszul algebras of global homological dimension 4, and so on.

\subsection{}
In this paper we view $A(E,\tau)$ as a homogeneous coordinate ring of a non-commutative analogue of $\PP^3$ that we denote by  $\Projnc(A(E,\tau))$, and   
examine the geometric features of $\Projnc(A(E,\tau))$, particularly the ways in which it is and is not like $\Projnc(S(E,\tau))$. 
Following the ideas and results in papers of Artin-Tate-Van den Bergh, Smith-Stafford, and Levasseur-Smith, 
we examine the line modules,  point modules, and fat point modules, over $A$, and their incidence relations.  
Point modules and fat point modules determine ``points'' in  $\Projnc(A)$. Line modules determine ``lines'' in  $\Projnc(A)$. 

For a moment, let $A$ denote a 3-dimensional quadratic Artin-Schelter regular algebra such that $\Projnc(A) \not\cong \PP^2$, i.e., $\Projnc(A)$ is a genuine non-commutative deformation of $\PP^2$. The ``points'' in $\Projnc(A)$ form a closed subspace of $\Projnc(A)$ that is a genuine commutative curve of arithmetic genus 
one. That curve, which is called the {\it point scheme} of $A$ or $\Projnc(A)$, is a fine moduli space for a moduli problem concerning the parametrization of ``points'' 
in $\Projnc(A)$. There is a similar moduli problem for ``lines'' in $\Projnc(A)$ and the {\it line scheme}, the solution to the moduli problem, is isomorphic to $\PP^2$.
There is also a theorem saying that every ``line'' in $\Projnc(A)$ meets the point scheme with multiplicity three. All these results, and the attendant definitions, 
appear in \cite{ATV2}.  The point scheme for a 3-dimensional Sklyanin algebra  $A(E,\tau)$  is the image of a closed immersion $E \to \Projnc(A)$. 

Non-commutative analogues of $\PP^3$ exhibit a much greater variety of behaviors and we still don't understand the typical, or 
generic, behaviour of points, lines, and their incidence relations, in that setting. Various papers of Michaela Vancliff and her co-authors illustrate some of this variety: \cite{V95}, \cite{VVR97},  \cite{SV99}, \cite{VVW98},   \cite{ShV99}, \cite{ShV02}, \cite{ShV02_bis}, \cite{SV06}, \cite{SV07}, \cite{ChV}.
More examples can be found in \cite{JTS94} and \cite{LBSvdB}.

\subsection{}
It was shown in \cite{SS92} that the point scheme for $\Projnc(S(E,\tau))$ consists of a copy of $E$ and 4 additional points. The point scheme is, in a natural way, a closed subscheme of $\PP(S_1^*) \cong \PP^3$. That closed subscheme is the natural copy of $E$ (embedded as a quartic curve) in $\PP(S_1^*) \cong 
\PP^3$ and the 4 additional points are the vertices of the 4 singular quadrics that contain $E$. In \cite{LS93}, it was shown that the ``lines'' in $\Projnc(S(E,\tau))$
are in natural bijection with the lines in $\PP(S_1^*)$  that are secant to $E$, and these are the same lines as those that lie on the pencil of quadrics in
$\PP(S_1^*)$ that contain $E$. The incidence relations between the lines and points in $\Projnc(S(E,\tau)$ are exactly the same as the incidence relations 
between the secant lines and the points of $E \sqcup\{\text{4 vertices}\}$. 

In addition to the points in $\Projnc(S(E,\tau))$, there are {\it fat points} (see \S\ref{ssect.fat.pts}).  
When $\tau$ has infinite order, there are 4 fat points of multiplicity $n$ for each integer $n\ge 2$ (the 4 ``isolated'' point modules should be considered  fat points of multiplicity 1)  \cite{SSJ93}. When $S(E,\tau)$ degenerates to a homogenization of the quantized enveloping algebra $U_q(\fsl_2)$ the fat points 
degenerate to the finite dimensional simple $U_q(\fsl_2)$-modules \cite[\S\S5,6]{CSW16}.

\subsection{}
In contrast to what happens for $\Projnc(S(E,\tau))$, Van den Bergh has shown there are non-commutative analogues of $\PP^3$ that have only 20 points \cite{VdB88}. Although he showed that such examples exist, and that this behavior is in some sense typical, he did not produce explicit examples. Since then,
Vancliff and her collaborators have produced a wealth of such examples (loc. cit.) including one for which the point scheme is a single point of multiplicity 20. 
Remarkably, there are examples with exactly $n<\infty $ points if and only if $n \in \{ 1,\ldots,20\}-\{2,15,17,19\}$ \cite{SV07}. 

By \cite[Thm. 9.3]{CS15},  $\Projnc(A(E,\tau))$ has exactly 20 points.
 
\subsection{}
 Lines in $\Projnc(A(E,\tau))$ and, what are essentially the same things, line modules for $A(E,\tau)$ are in bijection with certain lines in 
 $\PP(A_1^*) \cong \PP^3$ and  therefore correspond to the closed points of a certain subscheme $\LL$ of the Grassmannian $\GG(1,3)$. 
 Shelton-Vancliff call $\LL$ the {\it line scheme} for $A$. 
 
 In this paper we complete the classification of lines in $\Projnc(A(E,\tau))$.  We show that $\LL$ is  
 the union of 7 reduced and irreducible components, 3 quartic elliptic space curves and 4 plane conics in the ambient Pl\"ucker $\PP^5$, and that $\deg(\LL)=20$. The union of the lines corresponding to the points on each elliptic curve
is an elliptic scroll in $\PP(\wtQ_1^*)$. Thus, the lines on that elliptic scroll are in natural bijection with a corresponding family of line modules for $\wtQ$.

The fact that the degree of $\LL$ is 20 is a special case of a general phenomenon: if $A=TV/(R)$ is an algebra generated by a 4-dimensional vector space $V$
subject to a 6-dimensional subspace of relations $R \subseteq V^{\otimes 2}$, then every irreducible component of the line scheme $\LL$ has dimension $\ge 1$ and 
if $\dim(\LL)=1$, then $\deg(\LL)=20$ \cite[Thm. 1.1]{CSV15}.
 
We also determine the incidence relations between lines, points, and fat points of multiplicity 2,   in $\Projnc(A(E,\tau))$.
  
\subsection{}
It is pointed out in \cite[Remark 3.2]{ShV02} that if a point in a non-commutative analogue of $\PP^3$ lies on only finitely many lines, then it lies on exactly 6 
lines counted with multiplicity.  It is also remarked there that there is a paucity of such examples. We show that every point in $\Projnc(A(E,\tau))$  lies on 
exactly 6 lines (\Cref{main.thm.pts.lines} below).

\subsection{}
The remainder of this introduction will make some of what we have just said more precise.

  \subsection{The line scheme $\LL$ for $\wtQ$}
 If $A$ is any connected graded $k$-algebra generated by $A_1$ we call a module $L \in \Gr(A)$ a {\sf line module} if its Hilbert series is $(1-t)^{-2}$ and $L=A L_0$. 
  
In \cite{ShV02}, Shelton and Vancliff formulate a moduli problem, ``classification of line modules'',
for certain graded algebras $A$ and show there is a 
fine moduli space for it, the closed points of which are in natural bijection with the isomorphism classes 
of line modules for $A$. This moduli space is called the {\sf line scheme} for $A$. We denote it by $\LL$. 

For $\wtQ=\wtQ(E,\tau)$, $\LL$ is a closed subscheme of the Grassmannian $\GG(1,3)$ that parametrizes the lines in $\PP(\wtQ_1^*) \cong \PP^3$.  Often, we identify $A_1$ with $S_1$ and $\PP(\wtQ_1^*)$ with $\PP(S_1^*)$.

For each 2-torsion point $\xi$ there is a natural bijection between the secant lines $\overline{p,p+\xi} \subseteq 
\PP(S_1^*)$, $p \in E$, and isomorphism classes of certain line modules for $A$ \cite[\S10]{CS15}.
 
For each 2-torsion point $\xi$ the union of the secant lines $\overline{p,p+\xi}$ is an elliptic scroll in $\PP(S_1^*)$
and there is a corresponding closed immersion $E/\langle \xi\rangle \to \GG(1,3)$. We frequently identify  $E/\langle \xi\rangle$ with its image in $\GG(1,3)$. 

The degree of a closed subscheme of $\GG(1,3)$ is the degree of its image under the Pl\"ucker embedding $\GG(1,3) \to \PP^5$.

\begin{theorem}
[\Cref{thm.main1}]
\label{main.thm1}
Let  $\xi_1,\xi_2,\xi_3$ be the 2-torsion points of $E$.
The line scheme for $\wtQ$ is a reduced irreducible curve of degree 20. It is  the union of 3 disjoint quartic 
elliptic curves, $E_j \cong E/\langle \xi_j\rangle$, and 4 disjoint plane conics, 
$$
\LL \; = \;  \big(E_1 \; \sqcup \; E_2 \; \sqcup \;E_3 \big)  \; \bigcup \; 
\big(C_0  \; \sqcup \; C_1 \; \sqcup \; C_2  \; \sqcup \;   C_3\big),
$$
having the property that $\big\vert C_i \cap E_j \big\vert =2$ for all $(i,j) \in \{0,1,2,3\} \times \{1,2,3\}$.
\end{theorem}

We describe the line modules parametrized by 
the conics $C_j$ in \Cref{se.commuting_conic,se.other_conics}. There we show that $\LL$ contains the components in \Cref{main.thm1}. To show their union is $\LL$ we use a result by Derek Tomlin showing that $\dim(\LL) \le 1$. 
We are grateful to  Derek Tomlin for writing an appendix to this paper in which he provides a 
computer calculation of the reduced scheme $\LL_{\rm red}$ that we use  to show that 
$\LL$ is contained in the union of the components in \Cref{main.thm1}.
His work will appear, with more detail, in his Ph.D. thesis \cite{DT}. The key point, for us, is that Tomlin shows that 
$\dim(\LL) \le 1$. This inequality, when ${\rm char}(k)=0$ and $\tau$ is ``generic'', appeared earlier in the thesis of
Andrew Davies  \cite[Thm. 6.1.27]{Davies16}.

\subsection{}
Points on $C_0$ correspond to lines $y=y'=0$ in $\PP(\wtQ_1^*)$ 
where $y$ and $y'$ are linearly independent elements in $\wtQ_1$ such that
$[y,y']=0$. The conics $C_1,C_2,C_3$, are obtained from $C_0$ by using auto-equivalences of the category of 
graded left $\wtQ$-modules, $\Gr(\wtQ)$, that are induced by a coaction $\wtQ \to \wtQ \otimes H$ of a finite dimensional Hopf algebra
$H$. We think of $H$ as the coordinate ring of a finite quantum group acting on $\wtQ$. As a coalgebra, $H=k(H_4)$, the ring of 
$k$-valued functions on the Heisenberg group $H_4$ of order $4^3$. As an algebra, though, $H$ is not commutative.

The  isomorphism classes of the  line modules parametrized by $C_j$ are in natural bijection with 
the lines in a ruling on a smooth quadric in $\PP(A_1^*)$, the equations of which appear in \Cref{table.conics.quadrics}.

\subsection{}
In \cite{CS15}, we found various $\wtQ$-modules that become irreducible in the quotient category
$$
\QGr(\wtQ) \; := \; \frac{\Gr(\wtQ)}{\Fdim(\wtQ)}
$$
where $\Fdim(\wtQ)$ is the full subcategory of $\Gr(\wtQ)$
consisting of those graded modules that are the sum of their finite dimensional submodules. 
We write $\pi^*$ for the quotient functor $\Gr(\wtQ) \to \QGr(\wtQ)$ and $\pi_*$ for its right adjoint.

The irreducible objects in $\QGr(\wtQ)$ that we found in \cite{CS15} were of the form $\pi^*N$ where $N$
was either a point module for $\wtQ$ or a fat point module for $\wtQ$ of multiplicity two. 

We think of $\QGr(A)$ as if it is the category of quasi-coherent sheaves on the implicitly defined non-commutative variety $\Projnc(A)$. We think of $\Projnc(A)$ as a non-commutative analogue of $\PP^3$. 
The geometric properties of $\Projnc(S)$ and $\Projnc(A)$ are non-commutative analogues of the very 
beautiful geometric properties of quartic elliptic curves in $\PP^3$ and our main results are best understood in that context.

\subsection{}
\label{ssect.base.locus}
For both $S$ and $A$ there is a non-commutative analogue of the fact that $E$ is the base locus of a pencil of quadrics in 
$\PP(S_1^*)$. 

By \cite{SS92}, $S$ has a central subalgebra $k[\Omega,\Omega']$ that is a polynomial ring in two variables 
$\Omega,\Omega' \in S_2$ and ``the zero locus of $(\Omega,\Omega')$  in $\Projnc(S)$ is isomorphic to 
$E$'' in the sense that  $S/(\Omega,\Omega')$ is isomorphic to a twisted homogeneous coordinate ring 
$B(E,\tau,\cL)$ and $\QGr\!\big(B(E,\tau,\cL)\big)$ is equivalent to $\qcoh(E)$ by \cite{AV90}. 
We say that $\Projnc(S)$ contains a copy of $E$ because the homomorphism $S \to B(E,\tau,\cL)$ 
induces a fully faithful functor $i_*:\qcoh(E) \to \QGr(S)$ that has a left and a right adjoint and the essential image of $i_*$ is closed under subquotients.
In other words, $i_*$ behaves like a direct image functor for a closed immersion $E \to \Projnc(S)$. 

By \cite{CS15}, $A$ also has a central subalgebra $k[\Theta,\Theta']$ that is a polynomial ring in two variables 
$\Theta,\Theta' \in A_2$ and ``the zero locus of $(\Theta,\Theta')$  in $\Projnc(A)$ is isomorphic to 
$E/E[2]$'' in the sense that $\QGr\!\big(A/(\Theta,\Theta')\big)$ is equivalent to $\qcoh(E/E[2])$.
The ring $A/(\Theta,\Theta')$  is not a twisted homogeneous coordinate ring in the sense of \cite{AV90}. 
The natural functor $\Gr\big(A/(\Theta,\Theta') \big) \to \Gr(A)$ induces a functor $\qcoh(E/E[2]) \to \QGr(A)$
that behaves like the direct image functor for a ``closed immersion''  $E/E[2] \to \Projnc(A)$. 
Although $E/E[2]$ is isomorphic to $E$ it is better to think of it as $E/E[2]$ in this context.

\subsection{Points, lines, fat points, and quadrics}
\label{ssect.fat.pts}

A module $N \in \Gr(\wtQ)$ is a {\sf point module} if its Hilbert series is $(1-t)^{-1}$ and $N=\wtQ N_0$. By \cite[Thm. 9.3]{CS15}, the set $\fP$ of 
isomorphism classes of point modules has cardinality 20, and is a disjoint union of five 4-element subsets that we labelled $\fP_\infty$, $\fP_0, \fP_1,\fP_2,\fP_3$. We call point modules in $\fP_\infty$ {\sf special} and those in the other $\fP_j$ {\sf ordinary}.

A module $F \in \Gr(\wtQ)$ is a {\sf  fat point module of multiplicity two} if its Hilbert series is either
$2(1-t)^{-1}$ or $(1+t)(1-t)^{-1}$, $F=\wtQ F_0$, and $\pi^*F$ is a simple object in $\QGr(A)$ (the last condition is equivalent to the condition that  every proper quotient of $F$ has finite dimension). 
The image in $\QGr(A)$ under the ``direct image functor'' of 
the skyscraper sheaf $\cO_{p+E[2]}$ at a point $p+E[2] \in E/E[2]$ is a simple object in $\QGr(A)$ and 
the truncation $\big(\pi_*\cO_{p+E[2]}\big)_{\ge 0}$ is a fat point of multiplicity two with Hilbert series $2(1-t)^{-2}$.
If $v$ is any non-zero element in $\big(\pi_*\cO_{p+E[2]}\big)_{0}$, then $Av$ is a fat point of multiplicity two with 
Hilbert series $(1+t)(1-t)^{-2}$. See \cite[Thm. 8.1 and Cor. 8.7]{CS15}. 

If $N \in \Gr(\wtQ)$ is a point module (resp., a fat point module) we will often refer to  the isomorphism class of $\pi^*N$ in $\QGr(\wtQ)$ as a {\it point} (resp., a {\it fat point of multiplicity two}) because it is the analogue of the structure sheaf of a point (resp., a point with multiplicity two).  

If $L$ is a line module for $\wtQ$ we call  the isomorphism class of $\pi^*L$ a {\it line} because it is the analogue of the structure sheaf of a line. 

Since isomorphism classes of line modules and point modules for $A$, or $S$, are in bijection with certain lines and 
points in $\PP(A_1^*)$, or $\PP(S_1^*)$, there is some potential for confusion when one speaks of a line or point. We will
be careful to avoid ambiguity.

\subsection{Incidence relations}
After we classify the line modules for $\wtQ$, we determine which points and fat points of multiplicity two lie on which lines.  
We say that a point $\pi^*N$, or  a fat point $\pi^*F$,   {\sf lies} on the line $\pi^*L$ if there is an epimorphism 
$\pi^*L \to \pi^*N$ or $\pi^*L \to \pi^*F$. We also express this by saying that the point $\pi^*N$, or the fat point $\pi^*F$,
{\sf belongs to} the line $\pi^*L$. We note that $\pi^*N$, or $\pi^*F$, lies on $\pi^*L$ if and only if there is a surjective
homomorphism $L_{\ge n} \to N_{\ge n}$, or  $L_{\ge n} \to F_{\ge n}$ for $n \gg 0$. 

\begin{theorem}
\label{main.thm.pts.lines}
$\phantom{x}$
\begin{enumerate}
  \item 
 Every point in $\fP_\infty$ lies on exactly two lines in each $E/\langle \xi \rangle$ and lies on no lines in $C_0 \cup \ldots \cup C_3$.
  \item 
Every point in $\fP_j$, $j \ne \infty$, lies on exactly one line in each $E/\langle \xi \rangle$, 
on exactly one line in $C_i$ if $i \ne j$, and on no lines in $C_j$.
  \item 
Every fat point of multiplicity two corresponding to a point in $E/E[2]$  lies on exactly two lines in each family 
$E/\langle \xi \rangle$.
\item
For $i=0,1,2,3$,  there is a unique fat point of multiplicity two in $E/E[2]$ through which every line in $C_i$  passes, and ... see
below.
\end{enumerate}
\end{theorem}

To state (4) more precisely we must introduce the  point
$$
\tau':=(abc,a,b,c) \in E
$$
which satisfies $2\tau'=-\tau$. In the previous sentence, $a$, $b$, $c$, are related to the 
structure constants for $S$ and $A$ (see \S\ref{sect.defn.S}). If $k=\CC$, then $a$, $b$, $c$, are the values at $\tau$ of certain meromorphic functions on $E$ and, as a consequence of Jacobi's theta function identity $\theta_{00}(\tau)^4+\theta_{11}(\tau)^4=\theta_{01}(\tau)^4+\theta_{10}(\tau)^4$, they satisfy the identity $a^2b^2c^2+a^2+b^2+c^2=0$. Now we can state (4) more precisely:
\begin{enumerate}
  \item[(4)] 
{\it 
there are 4-torsion points $o=\ve_0,\ve_1,\ve_2,\ve_3$, such that every line  in $C_i$ passes through the 
fat point $\tau'+\ve_i$ of multiplicity two  and through no other fat point of multiplicity two in $E/E[2]$. 
In coordinates, $\tau'+\ve_1=(a,-ia,i,1)$, $\tau'+\ve_2=(b,1,-ib,i)$, $\tau'+\ve_3=(c,i,1,-ic)$. 
}
\end{enumerate}

Our understanding of \Cref{main.thm.pts.lines}(4) is informed by the fact that there are exactly four singular quadrics in 
$\PP(S_1^*)$ that contain $E$ and the vertex of each singular quadric lies on every line on that quadric.
One can embed $E/E[2]$ as a quartic curve in $\PP^3$  in such a way that 4 points of $E/E[2]$ are coplanar if and only if their sum is $o+E[2]$. There is a pencil of quadrics that contain $E/E[2]$. The surfaces in this pencil may be 
labeled $Q_z$, $z \in E/E[2]$, in such a way that $Q_z=Q_{-z}$ and the secant line $\overline{pq}$ through $p,q \in 
E/E[2]$ lies on $Q_z$ if and only if $p+q=\pm z$. 
There are exactly 4 singular quadrics containing $E/E[2]$, namely $\{Q_{\ve+E[2]} \; | \; \ve \in E[4]/E[2]\}$. 
Thus, if $\ve \in E[4]/E[2]$, then all lines in the pencil of secant lines $\{ \overline{pq} \; | \; p+q=\ve+E[2]\}$ pass through the vertex of $Q_{\ve}$.
We think of \Cref{main.thm.pts.lines}(4) as saying that $\tau'+\ve_i+E[2]$ behaves like the vertex of one 
of these singular quadrics in so far as all the lines (in $\Projnc(\wtQ)$) parametrized by $C_i$ pass through $\tau'+\ve_i+E[2]$.

We use the adjectives {\it elliptic} and {\it conic} for lines and/or line modules parametrized by the $E_j$'s and the $C_i$'s,
respectively.

\subsection{New results on the Sklyanin algebra}
Although $A$ is our primary interest in this paper we need some new results on the 4-dimensional Sklyanin algebras.
These results are in \S\ref{sect.4dim.Skl.alg}.

We determine four explicit two-dimensional irreducible representations $V(\tau+\xi)$, $\xi \in E[2]$,
of $S$ whose ``homogenizations'', $\wtV(\tau+\xi)$, are fat points of multiplicity 2; i.e., $\dim_k\big(\wtV(\tau+\xi)_n\big)=2$
when $n \ge 0$ and is zero otherwise, and every proper quotient of $\wtV(\tau+\xi)$ has finite dimension. Previously,
graded modules with these properties were constructed in \cite{SSJ93} only when $k=\CC$ and then in terms of 
theta functions.

We use the modules  $\wtV(\tau+\xi)$ to provide a better, more abstract, description of the point modules for $A$.
As a module over $S \otimes M_2(k)$, each $\wtV(\tau+\xi) \otimes k^2$ has four different $\G$-equivariant structures
and the $\G$-invariant subspaces give four different point modules for $A$. 
This equivariant construction of the point modules for $A$ enables us to determine which elliptic line 
modules map onto them. The elliptic line modules for $A$ are obtained in \cite[\S11]{CS15}  by a similar equivariant 
construction.

In \S\ref{sect.4dim.Skl.alg}, and in later sections, we exploit the action of the Heisenberg group of order $4^3$ as automorphisms of $S$.

\subsection{Non-commutative quadrics in $\Projnc(\wtQ)$}

If $z$ is a non-zero degree-two homogeneous central element in $\wtQ$ it is reasonable to view $\QGr(\wtQ/(z))$ as the category of ``quasi-coherent sheaves'' on a ``non-commutative quadric''  $\Projnc(\wtQ/(z)) \subseteq \Projnc(A)$. 
A justification for this view can be found  in sections 10 and 11 of Stafford and Van den Bergh's survey \cite{StVdB01}, and 
in \cite{SVdB-NCQ}. 

The best understood non-commutative surfaces
are those that are analogous to $\PP^2$ and $\PP^1 \times \PP^1$.  Non-commutative quadrics, meaning non-commutative analogues of $\PP^1 \times \PP^1$,
are also treated in some detail in \cite{SVdB-NCQ} and \cite{VdB11-NCQ}.   
We don't understand how $\Projnc(\wtQ/(z))$ fits into the taxonomy of non-commutative surfaces. 
For example, we do not know if  $\Projnc(\wtQ/(z))$ is birationally isomorphic to any previously known non-commutative 
surfaces.

 \subsection*{Acknowledgement}
We are grateful to Michaela Vancliff for several useful conversations, for pointing out some errors in earlier
drafts, and for sharing some calculations with us. We are also indebted to Derek Tomlin for sharing some calculations
with us and for writing an appendix to this paper which contains the result of those calculations. His computer calculations confirm some of our pencil and paper calculations and some abstract results such as \Cref{cor.ell_line=conic_line}. We are particularly grateful to the referee for reading the paper so carefully and
making specific corrections and suggestions for improvement.
 
 \section{The 4-dimensional Sklyanin algebras $S(E,\tau)$}
 \label{sect.4dim.Skl.alg}

\subsection{Definition of the Sklyanin algebra}
\label{sect.defn.S}
Always, $\a_1,\a_2,\a_3$ are fixed elements in $k- \{0,\pm 1\}$ such that $\a_1+\a_2+\a_3+\a_1\a_2\a_3=0$. 
We often write $\a=\a_1$, $\b=\a_2$, and $\c = \a_3$. 
As in \cite[\S6]{CS15}, we fix $a,b,c, i \in k$ such that $a^2=\a$, $b^2=\b$, $c^2=\c$, and $i^2=-1$.

The 4-dimensional Sklyanin algebra, $S=S(\a_1,\a_2,\a_3)$, is the quotient of the free algebra
$k \langle x_0,x_1,x_2,x_3\rangle$ by the six relations
\begin{equation}
\label{S-relns}
x_0x_i-x_ix_0 = \a_i(x_jx_k+x_kx_j), \qquad \quad x_0x_i+x_ix_0 =  x_jx_k-x_kx_j, 
\end{equation}
where $(i,j,k)$ runs over the cyclic permutations of $(1,2,3)$. 

\subsubsection{Notation}
If $r$ and $s$ are elements in a ring $R$ we write $[r,s]$ for their commutator and $\{r,s\}$ for their anti-commutator. With this notation, the relations for $S$ are
$$
[x_0,x_i]-\a_i\{x_j,x_k\} \; = \; \{x_0,x_i\}-[x_j,x_k]  \;=\; 0.
$$

\subsection{The curve $E \subseteq \PP^3=\PP(S_1^*)$}
\label{ssect.E.eqns}
Sklyanin originally defined $S$ over $\CC$ and did so by defining 
$\a_1,\a_2,\a_3$ to be values of certain meromorphic functions on an elliptic curve, $E$, evaluated at a particular point $\tau \in E$. Later, in \cite[\S2.4]{SS92}, it was shown that there is a ``natural'' copy of 
$E$ in $\PP(S_1^*)$ embedded as the quartic curve that is 
the intersection of any two of the quadrics
 \begin{equation}
 \label{eqns.for.E}
 \begin{cases}
  x_0^2+x_1^2+x_2^2+x_3^2 \; = \; 0,     
  \\
  x_0^2 - \b\c x_1^2 - \c x_2^2+\b x_3^2  \; = \; 0,
\\
x_0^2+\c x_1^2-\a\c x_2^2-\a x_3^2  \; = \; 0,
\\
x_0^2-\b x_1^2+\a x_2^2-\a\b x_3^2  \; = \; 0.
  \end{cases}
  \end{equation}
These formulas are illuminated by the calculation in  \Cref{quadrics.Q.tau} and by \Cref{prop.Q.tau}.
  
Thus, $E$ is the base locus of a pencil of quadrics. Exactly four of these quadrics are singular. 
We label the vertices of the singular quadrics $e_0,\ldots,e_3$ in such a way that $e_j$ is the point where $\{x_0,\ldots,x_3\}-\{x_j\}$ vanish. For example, $e_1=(0,1,0,0)$.

Let $p,q \in E$. We write $\overline{pq}$ for the line in $\PP(S_1^*)$ whose scheme-theoretic intersection with 
$E$ is the divisor $(p)+(q)$.

\begin{lemma}
The quadrics containing $E$ may be labelled $Q(z)$, $z \in E$, in such a way that 
\begin{enumerate}
  \item 
  $Q(z)=Q(-z)$;
  \item 
  $Q(z)$ is singular if and only if $z \in E[2]$;
  \item 
  the lines $\overline{pq}$, $p+q=z$, provide a ruling on $Q(z)$. 
\end{enumerate} 
\end{lemma}
\begin{proof}
Given $z \in E$, define $Q(z)$ to be the union of the secant lines $\overline{p,z-p}$ as $p$ ranges over $E$. 

As remarked in \S\ref{sect.1.3}, the group law on $E$ is chosen so it has the following property: 
if $p,q,r,s \in E$, then $p+q+r+s=o$ if and only if there is a plane $H \subseteq \PP^3$
such that the scheme-theoretic intersection $H \cap E$ is the divisor $(p)+(q)+(r)+(s)$. 

Let $Q$ be a smooth quadric containing $E$. Let $\ell$, $\ell'$, and $\ell''$ be
lines on $Q$ such that $\ell'$ and $\ell''$ belong to the same ruling and $\ell$ belongs to the other ruling. Let $p,p',p'',q,q',q'' \in E$ be such that
the scheme-theoretic intersections are
$\ell \cap E=(p)+(q)$, $\ell' \cap E=(p')+(q')$, and $\ell'' \cap E=(p'')+(q'')$. Since $\ell \cap \ell' \ne \varnothing$, there is a plane $H$ containing $\ell \cup \ell'$;
since $H \cap E=(p)+(q)+(p')+(q')$, $p+q+p'+q'=o$. Similarly, $p+q+p''+q''=o$. Thus, if $z=p+q$, then $Q=Q(z)=Q(-z)$. Since $Q$ has two different rulings on it, $z \ne -z$, i.e., $z \notin E[2]$. 

Now let $Q$ be a singular quadric containing $E$. Let $\ell=\overline{pq}$ and $\ell'=\overline{p'q'}$ be different lines on $Q$. Since  $\ell$ and $\ell'$ meet at 
the vertex of $Q$,  $\ell$ and $\ell'$ lie on a common plane that meets $E$ at $(p)+(q)+(p')+(q')$. Hence $p+q+p'+q'=o$.  Since $Q$ is singular, there is a plane 
that meets $Q$ in $\ell$ with multiplicity two. Hence  $p+q+p+q=o$. Thus, if $z=p+q$, then $z \in E[2]$ and $p'+q'=z=-z$. Thus, $Q=Q(z)=Q(-z)$. 

The result follows.
\end{proof}

We will label the elements in $E[2]=\{o=\xi_0,\xi_1,\xi_2,\xi_3\}$ are labelled so that $e_j$ is the vertex of the quadric $Q(\xi_j)$. 

In \Cref{prop.Q.tau}, we will show that the four quadrics cut out by the equations in (\ref{eqns.for.E}) are $Q(\tau+\xi)$, 
$\xi \in E[2]$.

 \subsection{Central elements in $S$}
\label{sect.ann.pt.mods}
There is a non-commutative analogue of the fact that the image of $E$ in $\PP(S_1^*)$ is the 
intersection of a pencil of quadrics: there is a closed immersion 
$E \to \Projnc(S)$ such that the image of $E$ is the ``intersection of 
a pencil of non-commutative quadrics''.  (See \S\ref{ssect.base.locus}.)

\begin{proposition}
\label{prop.center.S}
{\rm (cf.,  \cite[Cor. 3.9]{SS92}, \cite[p.39]{LS93})}
The center of $S$ contains $\Omega= -x_0^2 +  x_1^2 + x_2^2 + x_3^2$ and  the elements 
\begin{equation}
\label{central.elts.S}
\begin{cases}
\Omega_0  \; := \;  (1+\c)x_1^2\;\,\;+\;(1+\a\c)x_2^2\,\;+\; (1-\a)x_3^2,  
\\
\Omega_1  \; := \; (1+\b\c)x_0^2\;-\;(\c+\b\c)x_2^2 \;+\;   (\b-\b\c)x_3^2,
\\
\Omega_2  \; := \; (1+\a\c)x_0^2\;+\; (\c-\a\c)x_1^2 \;-\; (\a+\a\c)x_3^2,
\\
\Omega_3  \; := \; (1+\a\b)x_0^2\;-\;(\b+\a\b)x_1^2\;+\;(\a-\a\b)x_2^2.
\end{cases}
\end{equation} 
These elements satisfy the relations  $2\a\b\c \Omega_0+\a \Omega_1 + \b\Omega_2+\c\Omega_3=0$ and
$$
\a_j(1+\a_i)(1+\a_j)\Omega \; = \; (1+\a_i\a_j)\Omega_i-(1+\a_j)\Omega_k
$$
if $(i,j,k)$ is a cyclic permutation of $(1,2,3)$. 
\end{proposition}
 \begin{proof}
By  \cite[Cor. 3.9]{SS92}, $\Omega$ and 
$\Omega':=x_1^2+\big(\frac{1+\a}{1-\b}\big)x_2^2+\big(\frac{1-\a}{1+\c}\big)x_3^2$
belong to the center of $S$.
Error-prone computations show that the elements in (\ref{central.elts.S}) are linear combinations of $\Omega$ and 
$\Omega'$ so belong to the center of $S$. The first step in the proof is to observe that $(1+\c)\Omega' = \Omega_0$
because $(1+\a)(1+\c)=(1-\b)(1+\a\c)$. 
\end{proof}

\subsection{Line modules for $S$ and their annihilators}
Let $p,q \in E$. We write $(\overline{pq})^\perp$ for the 2-dimensional subspace of $S_1$ that vanishes on 
$\overline{pq}$.

Let $S^2E:=E \times E/\!\sim$ where $\sim$ is the equivalence relation $(p,q) \sim (q,p)$. 
The closed points in $S^2E$ are in bijection with the effective divisors of degree 2:  $(p)+(q)$ is the image  in $S^2E$ of $(p,q)$.

\begin{theorem}
\cite[Thm. 4.5]{LS93}
The function $S^2E \to \Gr(S)$ that sends $(p)+(q)$ to $S/S(\overline{pq})^\perp$ 
 is a bijection from $S^2E$ to the  set of isomorphism classes of line modules for $S$.
\end{theorem}

If $p,q \in E$, we define $M_{p,q}:= S/S(\overline{pq})^\perp$.

 \subsubsection{The central elements $\Omega(z)$}
Let $z \in E$. We will abuse notation and use the symbol $\Omega(z)$ to denote any non-zero element in $k\Omega_0+k\Omega_1$
 that annihilates all line modules $M_{p,q}$ for which $p+q=z$. 
By \cite{LS93}, $\Omega(z)=\Omega(-z-2\tau)$. Since $\Omega(z)$ is only defined up to a non-zero scalar multiple, the previous sentence has the same meaning as the phrase ``a particular line module is annihilated by $\Omega(z)$ if and only if it is annihilated by $\Omega(-z-2\tau)$''. 

\begin{proposition}
  \cite[Lem. 6.2, Cor. 6.6, Prop. 6.4, Prop. 6.8]{LS93}
\label{prop.line.killers}
Let $p,q \in E$. 
\begin{enumerate}
\item{}
The line module $M_{p,q}$ is annihilated by  $\Omega(z)$
if and only if $p+q \, \in \, \{z,-z-2\tau\}$.  
\item{}
The line $\overline{pq}$ lies on one of the singular quadrics that contains $E$
if and only if  $p+q \in E[2]$. 
\end{enumerate}
\end{proposition}

\subsection{Point modules for $S$}
If $p \in \PP(S_1^*)$, we write $p^\perp$ for the subspace of $S_1$ that vanishes at $p$.

\begin{theorem}
 \cite{SS92}
If $e_0,\ldots, e_3$ are the vertices of the singular quadrics that contain $E$, then  
$$
\big\{  \hbox{isomorphism classes of point modules for $S$}\big\}  \; = \;  \big\{S/Sp^\perp \; | \;   p \in E\sqcup\{e_0,e_1,e_2,e_3\}\big\}.
$$
\end{theorem}

If $p \in E\sqcup\{e_0,e_1,e_2,e_3\}$, we define $M_p:=S/Sp^\perp$.

 \begin{theorem}
  \cite[Thms. 5.5, 5.7]{LS93}
  Let $p,q \in E$. 
  \begin{enumerate}
  \item 
  There  is an exact sequence $0 \to M_{p+\tau,q-\tau}(-1) \to M_{p,q} \to M_{p} \to 0$. 
  \item 
If $p+q=\xi_j$, then there is an exact sequence $0 \to M_{p-\tau,q-\tau}(-1) \to M_{p,q} \to M_{e_j} \to 0$. 
\item{}
$M_{e_j}$ is annihilated by the central element $\Omega_j=\Omega(\xi_j)$ in (\ref{central.elts.S}).
\end{enumerate}
\end{theorem}
 
We call the point modules corresponding to the $e_i$'s {\sf special}.

One reason point modules  are important is that they are irreducible/simple as objects in the quotient category $\QGr(S)$. Not all simple objects in $\QGr(S)$ arise from point modules (see \S\ref{ssect.2-fat-pts}).

 \subsection{The action of $\G$ and $E[2]$ as automorphisms of $S$}
 \label{sect.Gamma.action}
In \cite[\S6]{CS15}, we defined an action of the group $\G=\ZZ_2 \times \ZZ_2=\{1=\c_0,\c_1,\c_2,\c_3\}$ as 
$k$-algebra automorphisms of $S$.
Table \ref{Gamma.action} describes the action: the entry in row $\c_i$ and column $x_j$ is $\c_i(x_j)$.
\begin{table}[htp]
\begin{center}
\begin{tabular}{|c|c|c|c|c|c|c|c|}
\hline
 & $x_0$ &  $x_1$ &  $x_2$ &  $x_3$
\\
\hline
$\c_1 $  & $x_0$ &  $x_1$ &  $-x_2$ &  $-x_3$
\\
\hline
$\c_2 $  & $x_0$ &  $-x_1$ &  $x_2$ &  $-x_3$
\\
\hline
$\c_3$  & $x_0$ &  $-x_1$ &  $-x_2$ &  $x_3$
\\
\hline
\end{tabular}
\end{center}
\vskip .12in
\caption{The action of $\G$ and $E[2]$ as automorphisms of $S$.}
 \label{Gamma.action}
\end{table}
\newline 
The induced action of $\G$ on $\PP(S_1^*)$ restricts to an action of $\G$ as automorphisms of 
$E$.  The results in \cite[\S7]{CS15} showed that if $p \in E \subseteq \PP(S_1^*)$, then $\c_j(p)=p+\xi_j$.
Because $\c_j(p)=p+\xi_j$ we often blur the distinction between $\G$ and $E[2]$
although we usually use multiplicative notation for $\G$ and additive notation for $E[2]$.

\subsubsection{}
Let $A$ be a $\ZZ$-graded $k$-algebra.
Let $\Aut_{\gr}(A)$ denote the group of graded $k$-algebra automorphisms of $A$. If $\l \in k^\times$, let 
$\phi_\l$ be the automorphism of $A$ that is multiplication by $\l^n$ on $A_n$. Then  $\l \mapsto \phi_\l$
is a homomorphism $k^\times \to \Aut_{\gr}(A)$. 
We will often identify $\l \in k^\times$ with $\phi_\l$  and, if  $\psi \in \Aut(A)$, we will write  $\psi^m=\l$ if $\psi^m=\phi_\l$
and $\l\psi$  for $\phi_\l \psi$. 

\subsubsection{}
If $\psi$ is a $k$-algebra automorphism of a $k$-algebra $R$ and $M$ is a left $R$-module we write 
$\psi^*M$ for the left $R$-module that is $M$ as a vector space with the new action given by $x \cdot m:=\psi(x)m$.

 \subsection{The action of the Heisenberg group $H_4$ as automorphisms of $S$}
By  \cite{FO98} and \cite[pp. 64-65]{SSJ93}, for example, the Heisenberg group of order $4^3$ acts as graded 
$k$-algebra automorphisms of $S$ when $k=\CC$. The next result shows this holds without restriction on $k$.

\begin{proposition} 
\label{prop.aut.S}
 If $(i,j,k)$ is a cyclic permutation of $(1,2,3)$, then 
there is a $k$-algebra automorphism $\phi_i$ of $S$  such that 
$\phi_i(x_0)=a_ja_kx_i$, $\phi_i(x_i)=-ix_0$, $\phi_i(x_j)=-ia_jx_k$, and $\phi_i(x_k)=-a_kx_j$,
where $(a_1,a_2,a_3)=(a,b,c)$. 
Thus, $\phi_i(x_j)$ is the entry in row $\phi_i$ and column $x_j$ in \Cref{autom}.
\begin{table}[htp]
\begin{center}
\begin{tabular}{|c|c|c|c|c|c|c|c|}
\hline
 & $x_0$ &  $x_1$ &  $x_2$ &  $x_3$  
\\
\hline
$\phi_1 $  & $bc x_1$ &   $-i x_0$ &  $-ib x_3$ &  $-c x_2$ $\phantom{\Big)}$
\\
\hline
$\phi_2$  & $ac x_2$ &   $-a  x_3$ &  $-ix_0$ &  $ -ic x_1$ $\phantom{\Big)}$
\\
\hline
$\phi_3$  & $ab x_3$ &   $-ia x_2$ &  $-bx_1$ &  $ -i  x_0$ $\phantom{\Big)}$
\\
\hline
\end{tabular}
\end{center} 
\caption{Automorphisms of $S$.}
\label{autom}
\end{table}

\noindent
Fix $\nu_1,\nu_2,\nu_3 \in k^\times$ such that $a\nu_1^2=b\nu_2^2=c\nu_3^2=-iabc$. Let
$\ve_1= \nu_1^{-1}\phi_1$, $\ve_2= \nu_2^{-1}\phi_2$,  $\ve_3 = \nu_3^{-1}\phi_3$, 
and $\d= i$. The subgroup $\langle \ve_1,\ve_2,\ve_3,\d\rangle \subseteq \Aut(S)$
is isomorphic to the Heisenberg group of order $4^3$, 
$$
H_4 \; : = \; \langle \ve_1,\ve_2,\d \; | \; \ve_1^4=\ve_2^4=\d^4=1, \;  \d\ve_1=\ve_1\d, \; \ve_2\d=\d\ve_2, \; \varepsilon_1\varepsilon_2=\delta \varepsilon_2\varepsilon_1\rangle.
$$
Furthermore, $\ve_1^2=\c_1$, $\ve_2^2=\c_2$, and $\ve_3^2=\c_3$
where $\c_1,\c_2,\c_3$ are the automorphisms of $S$ in \Cref{Gamma.action}. 
In particular, we can identify $\G$ with the subgroup $\langle \ve_1^2, \ve_2^2\rangle \subseteq H_4 \subseteq \Aut(S)$.
\end{proposition}
\begin{proof}
Let   $(i,j,k)$ be a cyclic permutation of $(1,2,3)$, let $\l_0,\l_i,\l_j,\l_k \in k^\times$,
and let $\phi:S_1 \to S_1$ be the linear map acting on $x_0,x_i,x_j,x_k$ as
\begin{table}[htp]
\begin{center}
\begin{tabular}{|c|c|c|c|c|c|c|c|}
\hline
 & $x_0$ &  $x_i$ &  $x_j$ &  $x_k$
\\
\hline
$\phi$  & $\l_0x_i$ &   $\l_ix_0$ &  $\l_jx_k$ &  $\l_kx_j$ 
\\
\hline
\end{tabular}
\end{center} 
\end{table}

\noindent
Following \cite[Prop. 4]{Skl83}, it is easy to see that $\phi$ extends to an automorphism  of $S$ if and only if  
$$
\frac{\l_0\l_i}{\l_j\l_k} \; = \; -1, \qquad 
\frac{\l_0\l_j}{\l_k\l_i} \; = \; -\a_j, \quad \hbox{and} \quad
\frac{\l_0\l_k}{\l_i\l_j} \; = \;  \a_k.
$$
The maps $\phi_1$, $\phi_2$, and $\phi_3$, satisfy these conditions so extend to automorphisms of $S$.  

Simple calculations show that  $\phi_1^2=-ibc\c_1$,  $\phi_2^2=-iac\c_2$,  and $\phi_3^2=-iab\c_3$, where $\c_1$, $\c_2$, and $\c_3$, are the automorphisms in \Cref{Gamma.action}. (For the sake of symmetry, let  $\ve_3=\nu_3^{-1}\phi_3$.) It follows that $\ve_1^2=\c_1$, $\ve_2^2=\c_2$, and 
$\ve_3^2=\c_3$.  Hence $\ve_1^4=\ve_2^4=\ve_3^4=1$. 
It is easy to check that  $\phi_1\phi_2=i\phi_2\phi_1$  which implies that $\ve_1\ve_2=i\ve_2\ve_1=\d\ve_2\ve_1$.
 We leave the rest of the proof to the reader. 
\end{proof}

\subsubsection{Remark}
\label{ssect.phi_i.phi_j}
Later we will make use of the fact that $\ve_2\ve_3=\d\ve_3\ve_2$ and $\ve_3\ve_1=\d\ve_1\ve_3$. These equalities follow
from the fact that  $\phi_2\phi_3=\d\phi_3\phi_2$ and $\phi_3\phi_1=\d\phi_1\phi_3$.

\begin{proposition}  
\label{prop.H4.action.on.Z}
$H_4$ acts on  the central element $\Omega=-x_0^2+x_1^2+x_2^2+x_3^2$ as follows:
\begin{equation}
\label{central.elts.S2}
\begin{cases}
-\phi_1(\Omega)  \; = \;  \Omega_1+\b\c \Omega  \;  = \; x_0^2+\b\c x_1^2-\c x_2^2+\b x_3^2,
 \\
-\phi_2(\Omega)  \; = \; \Omega_2+\a\c \Omega \;  = \;  x_0^2+\c x_1^2+\a\c x_2^2- \a x_3^2,
\\
-\phi_3(\Omega)  \; = \; \Omega_3+\a\b \Omega \;  = \;  x_0^2-\b x_1^2+\a x_2^2+\a\b x_3^2.
\end{cases}
\end{equation} 
\end{proposition}

 \subsubsection{}
 The elements of $S$ in  (\ref{central.elts.S2}) look much like the polynomials in (\ref{eqns.for.E}) that cut out $E$.
 
 \subsubsection{}
\Cref{{prop.aut.S}} gives a representation of $H_4$ on $S_1$. We let $H_4$ act on $S_1^*$ via the 
contragredient representation.

The center of $H_4$ is $\langle \d\rangle$ and  $H_4/\langle \d\rangle \cong (\bZ/4)^2$.  
The induced action of $H_4$ as automorphisms of $\PP(S_1^*)$ factors through $H_4/\langle \d\rangle$ 
thereby giving an action of $ (\bZ/4)^2$ as automorphisms of $\PP(S_1^*)$.

\begin{proposition}
\label{prop.transl.by.E4} 
The action of $H_4$ on $\PP(S_1^*)$ restricts to an action on $E$ as translation 
by  $E[4]$. Furthermore, if we identify $\ve_j \in H_4$ with the 4-torsion point $\ve_j(o)$, then $2\ve_j=\xi_j$.
\end{proposition}
\begin{proof}
Each $\ve_j:S_1 \to S_1$ is a scalar multiple of the linear automorphism $\phi_j$ in \Cref{prop.aut.S}.
Thus,  to show  that $E$ is stable under the action of $\ve_j$ it suffices 
to treat $\phi_j$ as an automorphism of the commutative polynomial ring $k[x_0,x_1,x_2,x_3]$
and show that the linear span of the polynomials in  (\ref{eqns.for.E}) is stable under the action of $\phi_j$.
Straightforward calculations show that
\begin{equation}
\label{quadrics.Q.tau}
\phi_j(x_0^2+x_1^2+x_2^2+x_3^2) \; = \; 
\begin{cases}
-x_0^2 +\b\c x_1^2 + \c x_2^2-\b x_3^2 & \text{if $j=1$,} 
\\
-x_0^2-\c x_1^2+\a\c x_2^2+\a x_3^2 & \text{if $j=2$,} 
\\
-x_0^2+\b x_1^2-\a x_2^2+\a\b x_3^2 & \text{if $j=3$.}
\end{cases}  
\end{equation}
These are scalar multiples of the other polynomials in  (\ref{eqns.for.E}). Similarly,
$$
\phi_j(x_0^2-\b\c x_1^2-\c x_2^2+\b x_3^2) \; = \; 
\begin{cases}
\b\c(x_0^2+x_1^2+x_2^2+x_3^2) & \text{if $j=1$,} 
\\
\c(x_0^2-\b x_1^2+\a x_2^2-\a\b x_3^2)  & \text{if $j=2$,} 
\\
-\b(x_0^2+\c x_1^2-\a\c x_2^2-\a x_3^2 )  & \text{if $j=3$.}
\end{cases}  
$$
These are also scalar multiples of the polynomials in  (\ref{eqns.for.E}). Hence $E$ is stable under the actions of 
$\ve_1$, $\ve_2$, and $\ve_3$.

Since every automorphism of $E$ as an algebraic variety is the composition of a translation and an automorphism of the group $(E,+,o)$, there is a point $q \in E$ and an automorphism $\phi$ of the group $(E,+)$ such that 
$\ve_1(p)=\phi(p)+q$ for all  $p \in E$. By 
\Cref{prop.aut.S}, $\ve_1^2=\c_1$. By \cite[\S7]{CS15}, $\c_1(p)=p+\xi_1$. Thus,
$p+\xi_1=\phi^2(p)+\phi(q)+q$ for all $p \in E$. It follows that $p=\phi^2(p)$ and $\xi_1=\phi(q)+q$. 

Because the characteristic of $k$ is not 2 or 3 the group of automorphisms of the group $(E,+)$ is cyclic of 
order 6 if $j(E)=0$,
cyclic of order 4 if $j(E)=12^3$, and cyclic of order 2 in all other cases. Thus, in all cases there is a unique automorphism of 
$(E,+)$ having order 2, namely $p \mapsto -p$. See \cite[Chap. 3]{Huse04}, for example.

If $\phi$ is not the identity map, then $\phi(p)=-p$ for all $p \in E$ 
which implies that $p+\xi_1=-p$ for all $p$. That is absurd, so we conclude that $\phi$ is the identity morphism and $p+\xi_1=p+2q$. Thus, $\ve_1$ is translation by $q$ which is a point of order 4. Similar arguments apply to $\ve_2$ and $\ve_3$. 

Since $\ve_j^2=\c_j$, and $\c_j$ acts on $E$ as translation by $\xi_j$, in $E$ we have $\ve_j+\ve_j=\xi_j$. 
\end{proof}

As in the previous proof we will often identify the image of 
$H_4$ in $\Aut(E)$ with the subgroup $E[4]$ acting by translation and when we do that
we will identify the automorphism $\ve_j$ with the point $\ve_j(o)$ which we will also label $\ve_j$. 
For the sake of symmetry we set $\ve_0=o$.

\begin{corollary}
\label{cor_H4_action_on_center}
The action of $\ve_j$ on $S$ sends $\Omega(z)$ to $\Omega(z+\xi_j)$. In particular, up to scalar multiples,
$\Omega(o)=\Omega_0$, and $\ve_j(\Omega_0)=\Omega_j=\Omega(\xi_j)$. 
\end{corollary}
\begin{proof}
A straightforward calculation using the definitions of $\ve_j$ and $\Omega_0$ shows that $\ve_j(\Omega_0)$ is a scalar
multiple of $\Omega_j$ for $j=1,2,3$. 

If $p,q \in E$, then the line module $M_{p,q}$ is annihilated by $\Omega(z)$ if and only if $p+q+\tau=\pm(z+\tau)$.
The same idea as in the proof of \cite[Prop. 7.7]{CS15}, shows that the auto-equivalence $\ve_j^*$ of $\Gr(S)$ induced by
the automorphism $\ve_j$ has the property that $\ve_j^*(M_{p,q}) \cong M_{p+\ve_j,q+\ve_j}$ where $p+\ve_j$
denotes translation by the 4-torsion point $\ve_j$. 
Since $\Omega(z+2\ve_j)$ annihilates $M_{p+\ve_j,q+\ve_j}$, we conclude that the automorphism
$\ve_j$ sends $\Omega(z)$ to $\Omega(z+\xi_j)$.
\end{proof}

\begin{proposition}
With respect to the coordinate functions $(x_0,x_1,x_2,x_3)$ on $S_1^*$ and $\PP(S_1^*)$, the action of $\ve_j \in H_4$ on a point $(\l_0,\l_1,\l_2,\l_3)\in \PP(S_1^*)$ is
\begin{equation}
\label{E4.action.formulas}
\ve_j(\l_0,\l_1,\l_2,\l_3)  \; =\; 
\begin{cases}
(bc \l_1,-i \l_0, ib \l_3,c\l_2) \phantom{\big)}& \text{if $j=1$,}  
\\
(ac \l_2, a \l_3, -i \l_0,ic\l_1) \phantom{\big)}& \text{if $j=2$,}  
\\
(ab \l_3,i a \l_2, b \l_1,-i\l_0) \phantom{\big)}& \text{if $j=3$.}  
\end{cases}
\end{equation}
\end{proposition}

By \Cref{prop.transl.by.E4}, the formulas in (\ref{E4.action.formulas}) have the following interpretation.

\begin{corollary}
\label{cor.transl.by.E4}
Let $\ve_j \in E[4]$, $j=1,2,3$, be the points $\ve_j(o)$. If $p=(\l_0,\l_1,\l_2,\l_3)  \in E$, then
\begin{align*}
p+\ve_1 & \; = \; (bc \l_1,-i \l_0, ib \l_3,c\l_2),
\\
p+\ve_2 & \; = \; (ac \l_2, a \l_3, -i \l_0,ic\l_1), \phantom{\big)} 
\\
p+\ve_3 & \; = \; (ab \l_3,i a \l_2, b \l_1,-i\l_0).
\end{align*} 
\end{corollary}

One can use the formulas in \Cref{cor.transl.by.E4} to check that  $p+2\ve_j=\c_j(p)$ where $\c_j$ is the automorphism of $\PP(S_1^*)$ induced by the linear automorphism $\c_j$ of $S_1$ in \Cref{Gamma.action}. 

\begin{corollary}
\label{cor.E4}
The 4-torsion subgroup $E[4]$ is the intersection of $E$ with the four ``coordinate'' planes 
$x_0x_1x_2x_3x_4=0$. More precisely,
if $j\in \{0,1,2,3\}$, then
$$
\ve_j + E[2] \; = \; E \cap \{x_j=0\}.
$$
\end{corollary}
\begin{proof}
By \cite[\S7]{CS15}, $E[2]$ is the intersection of $E$ with the plane $x_0=0$. It follows from the formulas in  \Cref{cor.transl.by.E4} that  $\ve_j + E[2] = E \cap \{x_j=0\}$.
\end{proof}

\subsection{The points in $-\frac{1}{2}\tau+E[2]$ and the quadrics $Q(\tau+\xi)$, $\xi \in E[2]$}

The point $$\tau':=(abc,a,b,c) \; \in \; E $$ plays an important role in the rest of \S\ref{sect.4dim.Skl.alg} 
and in \S\ref{subse.fat_conic}.

\begin{proposition}
\label{prop.abc.a.b.c}
Let $\tau'=(abc,a,b,c)$. Then $2\tau'=-\tau$, $-\tau'=(-abc,a,b,c)$, and 
\begin{align*}
\tau'+\ve_1 & \; = \; (a,-i a, i,1),  \phantom{xxxx}  -\tau'+\ve_1  \; = \; (a,i a, i,1)
\\
\tau'+\ve_2 & \; = \; (b, 1, -i b,i),  \phantom{xxxx}  -\tau'+\ve_2  \; = \; (b, 1, i b,i),  
\\
\tau'+\ve_3 & \; = \; (c,i , 1,-ic),  \phantom{xxxx}  -\tau'+\ve_3  \; = \; (c,i , 1,ic).
\end{align*} 
\end{proposition}
\begin{proof}
An explicit formula for the translation automorphism $q \mapsto q+\tau$ is given in \cite[Cor. 2.8]{SS92}.  
Applying it to $\tau'$ gives $\tau'+\tau=(-abc,a,b,c)$. However, by \cite[Eq. (7-2) and Prop. 7.4(3)]{CS15}, if 
$(\l_0,\l_1,\l_2,\l_3) \in E$, then $-(\l_0,\l_1,\l_2,\l_3)=(-\l_0,\l_1,\l_2,\l_3)$. Hence $(-abc,a,b,c)=-\tau'$ and $2\tau'=-\tau$. 
The formulas for $\pm \tau'+\ve_j$ come from \Cref{cor.transl.by.E4}.
\end{proof}

The quadrics defined by  the equations in (\ref{eqns.for.E}) can now be labelled according to the system described in \S\ref{ssect.E.eqns}. These equations are similar to the central elements of $S$ in \Cref{prop.H4.action.on.Z}.

\begin{proposition}
\label{prop.Q.tau}
The quadrics $Q(\tau+\xi)$, $\xi \in E[2]$, are 
\begin{align*}
Q(\tau) & \; = \; \{x_0^2+x_1^2+x_2^2+x_3^2=0\},
\\
Q(\tau+\xi_1) & \; = \; \{   x_0^2 - \b\c x_1^2 - \c x_2^2+\b x_3^2 =0  \},
\\
Q(\tau+\xi_2) & \; = \; \{ x_0^2+\c x_1^2-\a\c x_2^2-\a x_3^2=0  \},
\\
Q(\tau+\xi_3) & \; = \; \{ x_0^2-\b x_1^2+\a x_2^2-\a\b x_3^2  =0  \}.
\end{align*}
\end{proposition}
\begin{proof}
The line $x_0+ix_1=x_2-ix_3=0$ lies on the quadric $x_0^2+x_1^2+x_2^2+x_3^2=0$ and passes through
the points $-\tau'+\ve_1=(a,ia,i,1)$ and $-\tau'-\ve_1=(-\tau'+\ve_1)-\xi_1=(a,ia,-i,-1)$. Thus, the secant line through $-\tau'+\ve_1$ and $-\tau'-\ve_1$ lies on the quadric $x_0^2+x_1^2+x_2^2+x_3^2=0$. Hence that quadric is $Q(-\tau'+\ve_1-\tau'-\ve_1)=Q(\tau)$. 

We will do one more case and leave the other two cases to the reader. 

The line $x_0-icx_1=icx_2+x_3=0$ lies on the quadric $x_0^2+\c x_1^2-\a\c x_2^2-\a x_3^2=0$
and passes through the points $(ic,1,-i,-c)$ and $(ic,1,i,c)$. It follows from the various formulas above that
$(ic,1,-i,-c)=(-\tau'+\ve_2)+\ve_1$ and $(ic,1,i,c)=(-\tau'+\ve_2)-\ve_1$. 
Since $\tau+\xi_2$ is the sum of  $(-\tau'+\ve_2)+\ve_1$ and $(-\tau'+\ve_2)-\ve_1$ the line  
through the latter two points lies on $Q(\tau+\xi_2)$.  Therefore $Q(\tau+\xi_2)$ is 
 the quadric $x_0^2+\c x_1^2-\a\c x_2^2-\a x_3^2=0$.  
\end{proof}

\subsection{Some 2-dimensional simple $S$-modules $V(\tau+\xi)$, $\xi \in E[2]$}
\label{ssect.2-dim-irreps}
This section concerns un-graded $S$-modules.

The finite dimensional simple $S$-modules were classified in \cite{Skl83} and \cite{SSJ93} when $k=\CC$ and $\tau$ has infinite order: for each integer $n\ge 0$, there are four 1-parameter families of simple 
$S$-modules of dimension $n+1$ which were labelled $V(n\tau+\xi)^\l$, $\xi \in E[2]$, $\l \in k^\times$, in \cite{SSJ93}.
In  \cite{Skl83}, Sklyanin defined them by having the $x_j$'s
act as difference operators on certain spaces of theta functions. 

\Cref{prop.2-dim-irreps} describes four 2-dimensional simple modules that exist for all $k$ and all 
$\tau$ whose order is not 2 or 4. The four 1-parameter families of 2-dimensional simple $S$-modules can be obtained from these four modules in the following way: if $V$ is a left $S$-module and $\l \in k^\times$, let $V^\l$ be the vector space $V$ with a new action of $S$ in which each $x \in S_1$ acts on  $V^\l$ as $\l x$ acts on $V$.  

\subsubsection{Quaternions}
Define
\begin{equation}
\label{defn.a_i}
q_0=\begin{pmatrix} 1 & 0 \\ 0 & 1 \end{pmatrix}, \qquad  q_1=\begin{pmatrix} i & 0 \\ 0 & -i \end{pmatrix}, \qquad q_2= \begin{pmatrix} 0 & i \\ i & 0 \end{pmatrix}, \qquad q_3=  \begin{pmatrix} 0 & -1 \\ 1 & 0 \end{pmatrix}. 
\end{equation}
Then $q_1^2=q_2^2=q_3^2 = -1$ and if $(i,j,k)$ is a cyclic permutation of $(1,2,3)$, then $q_iq_j=q_k$ and $q_iq_j+q_jq_i=0$.

\begin{proposition} 
\label{prop.2-dim-irreps}
For $j=0,1,2,3$, let $\rho_j:S \to M_2(k)$ be the homomorphism with $\rho_j(x_i)$ equal to the entry in row 
$\rho_j$ and column $x_i$ of \Cref{2-dim-irreps}. The $\rho_j$'s are structure maps for four pairwise non-isomorphic 
2-dimensional simple $S$-modules, $V(\tau+\xi_j)$.  
\begin{table}[htp]
\begin{center}
\begin{tabular}{|c|c|c|c|c|c|c|c|c|}
\hline
 & &$x_0$ &  $x_1$ &  $x_2$ &  $x_3$  
 \\
\hline
$V(\tau)$ & $\rho_0$   & $1$ &   $q_1$ &  $q_2$ &  $q_3$ $\phantom{\Big)}$
\\
\hline
$V(\tau+\xi_1)$   & $\rho_1$ & $bc q_1$ &   $-i $ &  $-ib q_3$ &  $-c q_2$ $\phantom{\Big)}$
\\
\hline
$V(\tau+\xi_2)$  & $\rho_2$  & $ac q_2$ &   $-a  q_3$ &  $-i$ &  $ -ic q_1$ $\phantom{\Big)}$
\\
\hline
$V(\tau+\xi_3)$   & $\rho_3$ & $ab q_3$ &   $-ia q_2$ &  $-bq_1$ &  $ -i  $ $\phantom{\Big)}$
\\
\hline
\end{tabular}
\end{center}
\caption{2-dimensional simple $S$-modules.}
\label{2-dim-irreps}
\end{table}

\noindent
If $1 \le j \le 3$, then $V(\tau+\xi_j) \cong \phi_j^* V(\tau)$ where $\phi_j$ is the automorphism in 
\Cref{prop.aut.S}. 
\end{proposition}
\begin{proof}
First, $V(\tau)$ is an $S$-module because 
the elements in the row labelled $V(\tau)$ satisfy the relations
 \begin{align*}
[x_0,x_1] & \; = \; 0 \; = \; \{x_2,x_3\} \qquad \qquad  \{x_0,x_1\} \; = \; 2x_1 \; = \; [x_2,x_3] 
\\
[x_0,x_2] & \; = \; 0 \; = \; \{x_3,x_1\} \qquad \qquad \{x_0,x_2\} \; = \; 2x_2 \; = \; [x_3,x_1] 
\\
[x_0,x_3] & \; = \; 0 \; = \; \{x_1,x_2\} \qquad \qquad \{x_0,x_3\} \; = \; 2x_3 \; = \; [x_1,x_2]
\end{align*}
and therefore satisfy the six defining relations for $S$.  
Since $M_2(k)$ is generated as an algebra by $\{1,q_1,q_2,q_3\}$, $V(\tau)$ is  a simple $S$-module. 

The other three modules are obtained from $V(\tau)$ by applying the autoequivalences $\phi_j^*$ of $\Mod(S)$
induced by the $k$-algebra automorphisms $\phi_j$ in \Cref{prop.aut.S}. Explicitly, $\phi_j^*(V(\tau))$ is obtained by
having each $x_i$ act on $k^2$ by $\phi_j(x_i)$. Thus, applying $\phi_j$ to the elements in row $V(\tau)$ gives the
elements in the row $V(\tau+\xi_j)$. Of course, $V(\tau+\xi_j)$ is simple because $V(\tau)$ is. 

Let $j \in \{1,2,3\}$. The four modules are pairwise non-isomorphic because $x_j+i$ annihilates $V(\tau+\xi_j)$  but none of the other three modules, and $x_0-1$ annihilates $V(\tau)$ but not $V(\tau+\xi_j)$.
\end{proof}

It is easy to see that $V(\tau+\xi_j)$ is annihilated by the central element 
$$
\begin{cases}
\Omega+4 & \text{if $j=0$,}
\\
 x_0^2+\b\c x_1^2-\c x_2^2+\b x_3^2 +4\b\c & \text{if $j=1$},
 \\
  x_0^2+\c x_1^2+\a\c x_2^2- \a x_3^2 +4\a\c & \text{if $j=2$},
\\
x_0^2-\b x_1^2+\a x_2^2+\a\b x_3^2+4\a\b & \text{if $j=3$}.
\end{cases}
$$

The next proof uses the duality between lines in $\PP(S_1)$ and lines in $\PP(S_1^*)$ that is induced by the 
duality $X \longleftrightarrow X^\perp$ between 2-dimensional subspaces of $S_1$ and 
2-dimensional subspaces of $S_1^*$. 

\begin{proposition}
\label{prop.Mpq.onto.V}
Let $\xi \in E[2]$ and $p,q \in E$. If $p+q=\tau+\xi$, then
\begin{enumerate}
  \item 
there is a surjective homomorphism $M_{p,q} \to V(\tau+\xi)$ in $\Mod(S)$ and
  \item 
$\Hom_S(M_{p,q}, V(\tau+\xi)) \cong k$.
\end{enumerate} 
\end{proposition}
\begin{proof}
(1)
Consider the case $\xi=o$. The element ${1 \choose 0} \in V(\tau)$ is annihilated by $x_0+ix_1$ and $x_2-ix_3$.
In the proof of \Cref{prop.Q.tau}, we observed that the line  $x_0+ix_1=x_2-ix_3=0$ is the secant line through the points
 $-\tau'+\ve_1$ and $-\tau'-\ve_1$. Hence, there is a surjective homomorphism
$$
M_{-\tau'+\ve_1,-\tau'-\ve_1} \; =  \; \frac{S}{S(x_0+ix_1)+S(x_2-ix_3)} \; \twoheadrightarrow \; V(\tau).
$$
 
The action of $S$ on $V(\tau)$ restricts to a linear isomorphism $\rho:S_1 \to \End_kV(\tau)$. We will also write $\rho$ 
for the induced isomorphism $\PP(S_1) \to \PP(\End_kV(\tau))$. 
 Let $Q \subseteq \PP(\End_kV(\tau))$ be the quadric where the determinant vanishes. Let $\bf{L}$ denote the set
of lines on $Q$ that are the images of the simple left ideals in $\End_kV(\tau)$. The lines in $\bf{L}$
provide a ruling on $Q$.
The preimage $\rho^{-1}(Q)$ of $Q$ in $\PP(S_1)$ is also a smooth quadric and 
$\rho^{-1}({\bf L})=\{\rho^{-1}(\ell)  \; | \;  \ell \in \bf{L}\}$ is a ruling on $\rho^{-1}(Q)$. 

If $x=\l_0 x_0+\l_1x_1+\l_2 x_2 +\l_3 x_3$, then $\det\big(\rho(x)\big) = \l_0^2+\l_1^2+\l_2^2+\l_3^2$
so $\rho^{-1}(Q)=\{\l_0 x_0+\l_1x_1+\l_2 x_2 +\l_3 x_3 \; | \; \l_0^2+\l_1^2+\l_2^2+\l_3^2=0\}$. 
It is easy to verify the following fact: if $\ell \in \rho^{-1}(\bf{L})$, then the line $\ell^\perp :=\{\hbox{the points where $\ell$ vanishes}\}  \subseteq \PP(S_1^*)$ lies on the quadric $x_0^2+x_1^2+x_2^2+x_3^2=0$, i.e., on $\Omega(\tau)$.
Therefore $\{\ell^\perp \; | \;  \ell\in \rho^{-1}(\bf{L})\}$ is a ruling on $\Omega(\tau)$. 
Since the secant line through $-\tau'+\ve_1$ and $-\tau'-\ve_1$ belongs to $\{\ell^\perp \; | \;  \ell\in \rho^{-1}(\bf{L})\}$
it follows that 
$$
\{\ell^\perp \; | \;  \rho( \ell)\in {\bf L}\}   \; = \; \{\overline{pq} \; | \; p,q \in E \, \hbox{ and } p+q=\tau\}.
$$
This equality says that if $p+q=\tau$, then $S(\overline{pq})^\perp$ annihilates a non-zero element in $V(\tau)$ so there
is a surjective map
$$
M_{p,q} \; = \; \frac{S}{S(\overline{pq})^\perp} \; \twoheadrightarrow \; V(\tau).
$$

This completes the proof of (1) when $\xi=o$. 

The other cases are obtained by applying the auto-equivalences $\phi_j^*$, $j=1,2,3$, to the case $\xi=o$. 
There is a surjective homomorphism $\phi_j^*M_{p,q} \to \phi_j^*V(\tau)=V(\tau+\xi_j)$.
However, because $\phi_j$ is a scalar multiple of $\ve_j$,  $\phi_j^*M_{p,q} \cong \ve_j^*M_{p,q} \cong
M_{p+\ve_j,q+\ve_j}$.\footnote{This last isomorphism is discussed in the proof of \Cref{cor_H4_action_on_center}.}
In conclusion, if $p,q \in E$ are such that $p+q=\tau+\xi_j$, then there is a surjective homomorphism
$M_{p,q} \to V(\tau+\xi_j)$. 

(2)
Arguing as in the previous paragraph, it suffices to prove (2) when $\xi=o$. Since $M_{p,q}$ is a cyclic module and 
$\dim_k(V(\tau))=2$, the dimension of $\Hom_S(M_{p,q}, V(\tau+\xi))$ is at most two. However, if it were two, then 
$(\overline{pq})^\perp$ would annihilate $V(\tau+\xi)$. If $(p')+(q')$ is another divisor on $E$ such that $p'+q'=\tau$,
then the lines $\overline{pq}$ and $\overline{p'q'}$ belong to the same ruling on the smooth quadric $Q(\tau)$ so 
are disjoint. Hence  $(\overline{pq})^\perp+(\overline{p'q'})^\perp=S_1$. By (1) applied to $p'$ and $q'$ in place of $p$ 
and $q$ there is a non-zero element in $V(\tau)$ that is annihilated by $(\overline{p'q'})^\perp$ and hence by $S_1$.
This can not happen because $V(\tau)$ is a simple module of dimension 2. We conclude that the dimension of $\Hom_S(M_{p,q}, V(\tau+\xi))$ is one.
\end{proof}

\begin{corollary}
\label{cor.ann.V}
If $\xi \in E[2]$, then $V(\tau+\xi)$ is annihilated by $\Omega(\tau+\xi)$.  
\end{corollary}

\subsection{Multiplicity-two fat point modules in $\Gr(S)$}
\label{ssect.2-fat-pts}
Let $k[t]$ be the polynomial ring in one variable. As in \cite[Rmk. 2, p.79]{SSJ93}, for each $\xi \in E[2]$ we make 
$$
\wtV(\tau+\xi) \; := \; V(\tau+\xi) \otimes k[t]
$$
a graded left $S$-module by declaring that $\deg(v \otimes t^m)=m$ and having $a \in A_n$ act on $v\otimes t^m$
as $a(v \otimes t^m)=(av) \otimes t^{m+n}$. 

Clearly, for all $n \ge 0$, the graded $A$-modules $\wtV(\tau+\xi)(n)_{\ge 0}$ and  $\wtV(\tau+\xi)$ are isomorphic. 

When $k=\CC$ and $\tau$ has infinite order, the next result is a special case of results in \cite[\S4]{SSJ93}.

\begin{proposition}
\label{prop.L->F}
Let $\xi \in E[2]$ and let $p,q \in E$. If $p+q=\tau+\xi$, then 
\begin{enumerate}
  \item 
$\Hom_{\Gr(S)}\big(M_{p,q}, \wtV(\tau+\xi)\big) \cong k$;
  \item 
every non-zero $f \in \Hom_{\Gr(S)}\big(M_{p,q}, \wtV(\tau+\xi)\big) \cong k$ is surjective in degrees $\ge 1$;
\item
$\wtV(\tau+\xi)$ is 1-critical   of multiplicity 2;
  \item 
there is an exact sequence $0 \to M_{p-2\tau,q-2\tau}(-2) \to M_{p,q}  \stackrel{f}{\longrightarrow} \im(f) \to 0$.
\end{enumerate}
\end{proposition}
\begin{proof}
(1)
Let $ \fol:M_{p,q} \to V(\tau+\xi)$ be a surjective $S$-module homomorphism as in \Cref{prop.Mpq.onto.V}.  
For $m \in (M_{p,q})_n$ define $f(m):=\fol(m) \otimes t^n$.   It is easy to see that $f$ is a morphism in $\Gr(S)$ 
(cf., \cite[Rmk. 2, p.79]{SSJ93}). The argument used to prove \Cref{prop.Mpq.onto.V}(2) can be adapted to show that 
the dimension of $\Hom_{\Gr(S)}\big(M_{p,q}, \wtV(\tau+\xi)\big)$ is one.

(2) 
Since the structure map $S_1 \to \End_kV(\tau+\xi)$ is surjective  (see \Cref{2-dim-irreps}), $S_1(v \otimes 1) = V(\tau+\xi) \otimes t$ for all non-zero $v$ in $V(\tau+\xi)$. Therefore $f((M_{p,q})_1)=\wtV(\tau+\xi)_1$. 
The result follows.

(3)
Certainly,  $\wtV(\tau+\xi)$ has multiplicity 2. If $v$ is a non-zero element in  $V(\tau+\xi)$, then $S_1v=V(\tau+\xi)$.
Hence $S_1(v \otimes t^n)= \wtV(\tau+\xi)_{n+1}$ for all $n \ge 0$.  It follows that every proper graded quotient module of
 $\wtV(\tau+\xi)$ has finite dimension. Hence  $\wtV(\tau+\xi)$ is 1-critical. 

(4)
This is a special case of \cite[Prop. 4.4]{SSJ93}.  
\end{proof}

\subsubsection{Some functors}
Let $R$ be a graded $k$-algebra. 

We will now define a functor $G:\Mod(R) \to \Gr(R)$. 
Let $k[t]$ be the polynomial ring in one variable. If $V$ is a 
left $R$-module we make $G(V):=V \otimes k[t]$ a graded $R$-module by declaring that $(GV)_n=V \otimes t^n$
and making $r \in R_d$ act on $v \otimes t^n$ by $r(v \otimes t^n)=(rv) \otimes t^{n+d}$. If $h:V \to V'$ is a 
homomorphism of left $R$-modules we define $G(h):G(V) \to G(V')$ by $G(h)(v \otimes t^n)=h(v) \otimes t^n$.

If $\l \in k^\times$, let $\psi_\l\in \Aut_{\rm gr}(R)$ be the graded $k$-algebra automorphism defined by 
$\psi_\l(r)=\l^nr$ for $r \in R_n$. The auto-equivalence $\psi_\l^*:\Gr(R) \to \Gr(R)$ is defined as follows: 
\begin{enumerate}
  \item 
 if $M \in \Gr(R)$ the underlying graded vector space for $\psi_\l^*M$ is $M$; 
  \item 
  if $m \in M$ we write $m_\l$ for $m$ viewed as an element in $\psi_\l^*M$; 
  \item 
  the action of $r \in R_d$ on $m_\l \in \psi_\l^* M$ is $rm_\l=\l^d (rm)_\l$;
  \item
 if $h:M \to N$ is a morphism in $\Gr(R)$, then $\psi_\l^*(h)=h$. 
\end{enumerate}

 \begin{lemma}
\label{old.lemma}
Let $R$ be a graded $k$-algebra, $\l \in k^\times$, $V \in \Mod(R)$, and $M \in \Gr(R)$.
\begin{enumerate}
  \item 
There is an isomorphism $f:M \to \psi_\l^*M$ in $\Gr(R)$ given by $f(m)=\l^n m$ for $m \in M_n$. 
  \item 
There are isomorphisms $  \psi_\l^*GV  \stackrel{f}{\longleftarrow} GV \stackrel{\theta}{\longrightarrow} G\psi_\l^* V$ in
 $\Gr(R)$ given by $f(v \otimes t^n)=\l^n(v \otimes t^n)_\l$ and $\theta(v \otimes t^n)=\l^n v_\l \otimes t^n$.
 \item
If $\phi \in \Aut_{\sf gr}(R)$, then $G(\phi^*V) \cong \phi^*(G(V))$ in $\Gr(R)$.
  \end{enumerate}
\end{lemma}
\begin{proof}
(1) 
Certainly, $f$ is an isomorphism of graded vector spaces. Let $r \in R_j$ and $m \in M_n$. Then $r(f(m)_\l)=r(\l^n m_\l)
= \l^j\l^n (rm)_\l=\l^{n+j}(rm)_\l=f(rm)$. Hence $f$ is an $R$-module homomorphism and therefore an isomorphism of graded $R$-modules.

(2) 
Certainly, $\theta$ is an isomorphism of 
graded vector spaces. If $r \in R_d$, then $\theta(r(v \otimes t^n)) = \theta((rv) \otimes t^{n+d}) = \l^{n+d}(rv)_\l \otimes t^{n+d} $ and  $r\theta(v \otimes t^n)=r(\l^n v_\l \otimes t^n)=\l^n r(v_\l \otimes t^n)=\l^n\l^d (rv)_\l \otimes t^{n+d}$.
Hence $\theta(r(v \otimes t^n)) =r\theta(v \otimes t^n)$; i.e., $\theta$ is an $R$-module homomorphism and hence an isomorphism in $\Gr(R)$. 
The function $f$ in part (2) is the isomorphism $f$ in part (1) applied to $M=GV$.

(3)
By definition, $\phi^*V$ is equal to $V$ as a vector space and we write $v_\phi$ when an element $v \in V$ is considered
as an element of $\phi^*V$. Thus, $rv_\phi=(\phi(r)v)_\phi$ for all $r \in R$ and $v\in V$. 

The function $g:G(\phi^*V) \to \phi^*(G(V))$ given by $g(v_\phi \otimes t^n)=(v \otimes t^n)_\phi$ is an 
isomorphism of graded vector spaces. It is also a homomorphism of $R$-modules because if $r \in R_d$, then
$$
g(r(v_\phi \otimes t^n)) =g(rv_\phi \otimes t^{n+d})) =g((\phi(r)v)_\phi \otimes t^{n+d})= (\phi(r)v \otimes t^{n+d})_\phi
$$ 
and 
$$
rg(v_\phi \otimes t^n) =r  (v \otimes t^n)_\phi = \big( \phi(r)(v \otimes t^n) \big)_\phi  =   (\phi(r)v \otimes t^{n+d})_\phi.
$$ 
Hence $g$ is an isomorphism in $\Gr(R)$.
\end{proof}

\begin{proposition}
\label{prop.V.tau.F.tau}
$\phantom{x}$
\begin{enumerate}
  \item 
  $\c^* V(\tau) \cong V(\tau)$ for all $\c \in \G=\langle\ve_1^2,\ve_2^2\rangle\subseteq H_4$.
  \item 
  $\c^*\wtV(\tau+\xi_j) \cong \wtV(\tau+\xi_j)$ for all  $\c \in \G \times Z=\langle\ve_1^2,\ve_2^2,\d\rangle 
  \subseteq  H_4$ and all $j\in \{0,1,2,3\}$. 
  \item
  $\phi_j^*\wtV(\tau) \cong \wtV(\tau+\xi_j)$ for $j\in \{0,1,2,3\}$.
\end{enumerate}
\end{proposition}
\begin{proof}
(1)
Let $\rho:S_1 \to M_2(k)$ be the restriction to $S_1$ of the structure map for $V(\tau)$; i.e., $\rho(x_i)$ is the 
entry in row $V(\tau)$ and column $x_i$ of \Cref{2-dim-irreps}.   

We write $\c_0=1,\c_1,\c_2,\c_3$ for the elements of $\G$ as in \Cref{Gamma.action}.

We will show that the function $\theta:k^2 \to k^2$, $\theta(v) = q_jv$, is an isomorphism of left $S$-modules 
$V(\tau) \to \c_j^*V(\tau)$ by showing that $\theta\big(\rho(x)v\big) = \rho\c_j(x)\theta(v)$ for all 
$x \in S_1$ and $v \in k^2$.  
To do this it suffices to show that $q_j\rho(x_i) = \rho\c_j(x_i)q_j$ for $i=0,1,2,3$.  
This is true because
$$
\rho\c_j(x_i) q_j \; = \; 
\begin{cases} 
\rho(x_i) q_j \phantom{x} \,= \; \;q_i q_j  \, \,= q_jq_i = q_j \rho(x_i) & \text{if $j \in \{0,i\}$,}
\\
\rho(-x_i)q_j =-q_i q_j=q_jq_i = q_j \rho(x_i) & \text{if $j \notin \{0,i\}$.}
\end{cases}
$$ 
 
(2)
We first prove that $\c^*\wtV(\tau) \cong \wtV(\tau)$.
The elements in $Z=\langle \d \rangle \subseteq H_4$ act on $S_1$ as scalar multiplication. Thus, if $\c \in \G \times Z$,
then $\c^* = \c_j^*\psi_\l^*$ for some $\l \in \{ \pm i, \pm 1\}$ and some $j \in \{0,1,2,3\}$. 
By \Cref{old.lemma}, $\c_j^*\psi_\l^*\wtV(\tau)   \cong G\big(\c_j^*V(\tau)\big) \cong G\big(V(\tau)\big)=\wtV(\tau)$. 

(3) 
By  definition, $V(\tau+\xi_j) = \phi_j^*V(\tau)$ so, by \Cref{old.lemma}(3),  $\phi_j^*\wtV(\tau) = \phi_j^*G\big(V(\tau)\big)  \cong  G\phi_j^*V(\tau) =G \big(V(\tau+\xi_j)\big)  = \wtV(\tau+\xi_j)$.

Now we can finish the proof of (2). Since $\phi_j \c = \psi_\l\c\phi_j$ for some  $\l \in \{ \pm i, \pm 1\}$, 
$$
\c^*\wtV(\tau+\xi_j) \; = \; \c^*\phi_j^*\wtV(\tau) \; = \;  (\phi_j\c)^*\wtV(\tau) \; = \;   \phi_j^*\c^*\psi_\l^*\wtV(\tau)   \;  \cong \;
 \phi_j^*\c^*\wtV(\tau) \;  \cong \;
 \phi_j^*\wtV(\tau) \;  \cong \;
\wtV(\tau+\xi_j) 
$$
where, again, we used \Cref{old.lemma}(3).
\end{proof}

 \section{Preliminaries concerning $A=A(E,\tau)$}

 \subsection{The definition of $\wtQ$}
\label{ssect.quat.basis}
We make $\G$ act as $k$-algebra automorphisms of $M_2(k)$ by  $\c_j(a):=q_jaq_j^{-1}$ and  define 
$$
\wtQ:=(S \otimes M_2(k))^\G
$$
where $\G$ acts diagonally as graded $k$-algebra automorphisms of $S \otimes M_2(k)$. 
The elements  
$$
y_0:=x_0 \otimes 1, \quad y_1:=x_1\otimes q_1, \quad y_2:=x_2\otimes  q_2, \quad y_3:=x_3\otimes  q_3,
$$
are a basis for $\wtQ_1$ and  $A$ is isomorphic to the free algebra $k \langle y_0,y_1,y_2,y_3\rangle$ modulo the relations
\begin{equation}
\label{S-tilde-relns}
y_0y_i-y_iy_0 \;=\; \a_i(y_jy_k-y_ky_j) \qquad \hbox{and} \qquad y_0y_i+y_iy_0  \;=\;  y_jy_k+y_ky_j,
\end{equation}
See \cite[\S6]{CS15} for more information.

  \subsubsection{$\G$-equivariant $S'$-modules}
  We define $S':= S \otimes M_2(k)$ and write $\Gr(S')^\G$ for the category of $\G$-equivariant graded left $S'$-modules. 
By \cite[Prop. 3.9]{CS15}, there is an equivalence  of categories $\Gr(S')^\G \equiv \Gr(\wtQ)$ given by 
sending a $\G$-equivariant $S'$-module $M$ to $M^\G$, its $\G$-invariant subspace. 

We usually view a $\G$-equivariant $S'$-module  as a left $S'$-module $M$ endowed with a left action $\G \times M \to M$, 
$(\c,m) \mapsto m^\c$, such that $(xm)^\c=\c(x)m^\c$ for all $x \in S'$, $m \in M$, and $\c \in \G$. We often describe this by saying that the
actions of $\G$ and $S'$ on $M$ are {\sf compatible}.

  \subsubsection{Twisting $S$ by a 2-cocycle}
  \label{ssect.davies.cocycle}
Let $\widehat{\G}=\{1,\chi_1,\chi_2,\chi_3\}$ denote the character group of $\G$.  
Since the characteristic of $k$ does not divide $|\G|$, there is a $\widehat{\G}$-grading on $S$, 
$$
S \; = \; \bigoplus_{\chi \in \widehat{\G}} S_\chi   \; = \; \bigoplus_{\chi \in \widehat{\G}} \; \{a \in S \; | \; \c(a)=\chi(\c)a \; \hbox{ for all } \, \c \in \G\}.
$$ 

Let $\mu:\widehat{\G}  \times \widehat{\G} \to k^\times$ be the 2-cocycle  $\mu(\chi_1^p \chi_2^q,\chi_1^r\chi_2^s)  := (-1)^{ps}$. In his  Ph.D. thesis \cite[Prop.3.1.16]{Davies-arXiv}, Davies
showed that  $S$ with the new multiplication  $w*w' \; := \; \mu(\chi,\chi') ww'$
if $w \in S_{\chi}$ and $w'\in S_{\chi'}$ is isomorphic to $\wtQ$.

In \S\ref{ssect.construction.of.S-tilde}, we give a slightly different interpretation of this construction of $A$ by starting with 
the homomorphism $H_4 \to \Aut(S)$ and viewing $S$ as an algebra object in the category of  comodules over the  ring $k(H_4)$ of $k$-valued functions on $H_4$.

\subsection{Our convention for identifying $A_1$ and $S_1$}
\label{conv.identification}
 From the perspective in \S\ref{ssect.davies.cocycle}, 
 $S$ and $\wtQ$ have the same underlying $\ZZ$-graded vector space but different algebra structures. 
 
 In terms of the generators $x_i$ and $y_i$ we  adopt the convention that $S_1$ and $\wtQ_1$ are identified via
  \begin{equation*}
  x_0 \, =\, y_0, \qquad   x_1 \, =\, iy_1, \qquad  x_2 \, =\, iy_2, \qquad  x_3 \, =\, y_3. \qquad 
  \end{equation*}
We use this convention whenever we refer to $S$ and $\wtQ$ as being supported on the same graded vector space. 
Later on, we will use the fact that this is the identification used in \cite[Prop. 10.10]{CS15}.

\subsubsection{}
By \cite[Prop. 6.2]{CS15}, $\G$ also acts as automorphisms of $\wtQ$. In fact, \cite[Prop. 6.2]{CS15} shows that 
if one identifies $S_1$ and $\wtQ_1$ according to the above convention, then the action of 
$\G$ on $S_1=\wtQ_1$ extends to an action as $k$-algebra automorphisms of both $S$ and $\wtQ$.  
We will sometimes use this convention to obtain an action of $H_4$ on $S_1=\wtQ_1$;  this action of $H_4$
extends to an action of $H_4$ as $k$-algebra automorphisms of $S$ but not of $\wtQ$.

\subsection{The centers of $S$ and $A$}
\label{ssect.2centers}
Let $\Omega_j$, $j=0,1,2,3$,  be the central elements of $S$ defined in \Cref{prop.center.S}.
The only central elements in $S$ that play a role in this paper are those in the  subalgebra $k[\Omega_0,\Omega_1]=k[\Omega_0,\Omega_1,\Omega_2,\Omega_3]$.
Each $\Omega_j$ is fixed by $\G$ so
$k[\Omega_0,\Omega_1] \otimes 1 \subseteq (S \otimes M_2(k))^\G=A$. 
We will identify $k[\Omega_0,\Omega_1]$ with this subalgebra of $A$. Since $k[\Omega_0,\Omega_1] \otimes 1$ is in the center of $S \otimes M_2(k)$ this copy of 
$k[\Omega_0,\Omega_1]$ belongs to the center of $A$.  

We will take advantage of this fact by using the notation $\Omega(z)$, $z \in E$, for the 
central element $\Omega(z) \otimes 1$ in $A$. Each $\Omega(z)$ is a linear combination of $x_0^2$, $x_1^2$, $x_2^2$,
and $x_3^2$. As an element in $A$, $\Omega(z)$, which is really $\Omega(z) \otimes 1$, is a linear combination of 
$y_0^2$, $y_1^2$, $y_2^2$, and $y_3^2$. Since $y_j=x_j \otimes q_j$, 
$$
y_j^2=
\begin{cases} 
x_0^2 \otimes 1 & \text{if $j=0$,}
\\
-x_j^2 \otimes 1 & \text{if $j \ne 0$.}
\end{cases} 
$$
Thus, when considered as elements of $A$, the central elements in (\ref{central.elts.S2})   are
\begin{equation}
\label{central.A.elts}
\begin{cases}
\Omega \phantom{x} \,\; = \; -y_0^2- y_1^2- y_2^2-y_3^2,
\\
\Theta_1 \; := \; y_0^2-\b\c y_1^2+\c y_2^2-\b y_3^2,
\\
\Theta_2 \; := \; y_0^2-\c y_1^2-\a\c y_2^2+\a y_3^2,
\\
\Theta_3 \; := \;  y_0^2+\b y_1^2-\a y_2^2-\a\b y_3^2.
\end{cases}
\end{equation}

\subsection{The ring $\wtB$}
\label{ssect.B_tilde}
Let $\wtB$ be the quotient of $A$ by the ideal generated by all $\Omega(z)$. By \cite[\S8]{CS15}, there is an equivalence 
of categories $\QGr(\wtB) \equiv \qcoh(E/E[2])$. Under this equivalence the skyscraper sheaf $\cO_{p+E[2]}$ at 
the point $p+E[2]$ corresponds to the image of $M_p \otimes k^2$ where $M_p$ is the point module in $\Gr(S)$ corresponding to $p$ and $A$ acts on $M_p\otimes k^2$ by virtue of the fact that $S$ is a subalgebra of 
$S \otimes M_2(k)$. If $\omega \in E[2]$, then $M_p \otimes k^2$ and $M_{p+\omega} \otimes k^2$ are isomorphic in $\Gr(A)$.

We call the $A$-module $M_p \otimes k^2$ a fat point of multiplicity two; ``fat'' because its Hilbert series is $2(1-t)^{-1}$
and ``point'' because it is irreducible as an object in $\QGr(A)$.

\subsection{Elliptic line modules}
Let $p \in E$ and write $x=p+E[2] \in E/E[2]$. If $\xi$, $\xi'$, and $\xi''$ are the three 2-torsion points, we define the $S'$-module
\begin{equation}
\label{defn.M.x.xi}
M_{x,\xi} \; := \; (M_{p,p+\xi} \oplus M_{p+\xi',p+\xi'+\xi}) \otimes k^2.
\end{equation}

\begin{proposition}
\cite[\S10]{CS15}
If $\xi$ and $\omega$ are 2-torsion points and $x,y \in E/E[2]$, then  
 $M_{x,\xi} \cong M_{y,\omega}$ as $S'$-modules if and only if $(x,\xi)=(y,\omega)$. 
\end{proposition}
 
In \cite[\S10]{CS15}, we showed there are exactly two $\G$-equivariant structures on each $M_{x,\xi}$ and then, by taking the $\G$-invariant subspace of the equivariant structures on the $M_{x,\xi}$'s, we constructed  line modules for $\wtQ$ parametrized by
$(E/\langle \xi\rangle) \sqcup (E/\langle \xi'\rangle) \sqcup (E/\langle \xi''\rangle)$. 
We  call these  {\sf elliptic line modules}.

\subsection{Other remarks, conventions, and notation}

\subsubsection{Symmetries in $\wtQ$}
\label{ssect.symmetry}
Calculations in $\wtQ$ can often be carried out more efficiently by exploiting the symmetries involving the ordered triples 
$(\a_1,\a_2,\a_3)$ and $(y_1,y_2,y_3)$. To formalize this idea let $\a_1$, $\a_2$, and $\a_3$, denote central indeterminates and
define the algebra $\widehat{S}:=k[\a_1,\a_2,\a_3][y_0,y_1,y_2,y_3]$ with relations  (\ref{S-tilde-relns}). There is a $k$-algebra automorphism
$\phi:\widehat{S} \to \widehat{S}$ given by $\phi(y_0)=y_0$, and $\phi(\a_j)=\a_{j+1}$ and $\phi(y_j)=y_{j+1}$ for $j \in \{1,2,3\}=\ZZ/(3)$.
Most explicit calculations in $\wtQ$ can be interpreted as calculations in $\widehat{S}$ followed by 
the obvious specialization $\widehat{S} \to \wtQ$. 

For example, if one wants to show that the elements $\D_i:=y_0^2-\a_{j}\a_k y_i^2 +\a_k y_j^2 -\a_j y_k^2$, where $(i,j,k)$ is a cyclic permutation 
of $(1,2,3)$, are central in $\wtQ$ it suffices to show that the corresponding element $\widehat{\D}_1 \in \widehat{S}$ is central in 
$\widehat{S}$, then apply $\phi$ and $\phi^2$ to $\widehat{\D}_1$ to see that $\widehat{\D}_2$ and $\widehat{\D}_3$ are central in 
$\widehat{S}$, then specialize to $\wtQ$ to see that $\D_1$, $\D_2$, and $\D_3$, are central in $\wtQ$.

\subsubsection{Comodules}
If $H$ is a $k$-coalgebra we write $\cM^H$ for the category of right $H$-comodules. If $\dim_k(H)<\infty$, then 
$\cM^H$ is equivalent to the category $\Mod(H^*)$ of left $H^*$-modules.

\subsubsection{Invariant modules}
Let $\G$ be a group acting as automorphisms of a ring $R$. If $\c \in \G$ we denote by $\c^*$ the following 
automorphism of $\Mod(R)$:  if $f:M \to N$ is a morphism in $\Mod(R)$, 
$\c^*M$ is $M$ as an abelian group with a new action of $R$, $r*m=\c(r)m$, and $\c^*(f)=f$. If $\c^*M \cong M$ for all 
$\c \in \G$ we say that $M$ is {\sf invariant} under the action of $\G$. 

\subsubsection{Hilbert series, multiplicity, and criticality}
The  {\sf Hilbert series} of a finitely generated graded module $M$ over $S$, or $A$, or $A \otimes M_2(k)$,  is the
formal series $H(M;t)=\sum_{n \in \ZZ} \dim_k(M_n)t^n$. Because these rings have finite global dimension and Hilbert series
$(1-t)^{-4}$ for $S$ and $A$, and $4(1-t)^{-4}$ for $A \otimes M_2(k)$, $H(M;t)=f(t)(1-t)^{-d}$ for some integer $d$
and some $f(t) \in \ZZ[t^{\pm 1}]$ such that $f(1) \ne 0$. The integer $d$ is the Gelfand-Kirillov dimension, or {\sf GK-dimension} of $M$. We call $f(1)$ the {\sf multiplicity} of $M$. If $\GKdim(M)=d$ and $\GKdim(M/N)<d$ for all non-zero $N \subseteq M$ we say that $M$ is $d$-critical. If $M$ is 1-critical, then $M$ becomes a simple object in the quotient 
category $\QGr$.


\section{Point modules for $\wtQ$}

The point modules for $\wtQ$ were classified in \cite[\S9]{CS15} by a rather unilluminating calculation.
In this section we show they can be obtained in a more systematic and meaningful way. 
Let $j\in \{0,1,2,3\}$ and write $F=\wtV(\tau+\xi_j)$. 
We will show there are four $\G$-equivariant structures on the $S'$-module $F \otimes k^2$,  
and that the four point modules $(F \otimes k^2)^\G$  for $A$ are the four point modules in $\fP_j$ below.

\subsection{Classification of point modules for $\wtQ$}
If $M$ is a point module for $\wtQ$, then $M \cong \wtQ/\wtQ p^\perp$ for a unique 
$p \in \PP(\wtQ_1^*)$. By \cite[\S9]{CS15}, $\wtQ$ has 20 point modules up to isomorphism and, with respect 
to the coordinate functions $(y_0,y_1,y_2,y_3)$ on $\PP(\wtQ_1^*)$, the corresponding points are those  in \Cref{table.20.pts}.

\begin{table}[htp]
\begin{center}
\begin{tabular}{|l||l|l|l|l|l|}
\hline
$\quad \fP_\infty$ & $\qquad \fP_0$ & $\qquad \fP_1$ & $\qquad \fP_2$ & $\qquad \fP_3$ & $\G$
\\
\hline
\hline
$(1,0,0,0)$ & $(1,1,1,1)$ & $(bc,-i,-ib,-c)$ & $(ac,-a,-i,-ic)$&   $(ab, -ia,-b,-i)$  & 
\\
\hline
$(0,1,0,0)$ &$(1,1,-1,-1)$   & $(bc,-i,ib,c)$  &  $(ac,-a,i,ic)$&  $(ab, -ia,b,i)$  & $\c_1$
\\
\hline
 $(0,0,1,0)$ &$(1,-1,1,-1)$   & $(bc,i,-ib,c)$  &  $(ac,a,-i,ic)$&  $(ab, ia,-b,i)$  & $\c_2$
  \\
\hline
 $(0,0,0,1)$ & $(1,-1,-1,1)$ &$(bc,i,ib,-c)$  & $(ac,a,i,-ic)$  & $(ab, ia,b,-i)$  & $\c_3$
  \\
\hline
\end{tabular}
\end{center}
\vskip .12in
\caption{The points in $\fP$.}
\label{table.20.pts}
\end{table}

\noindent
For example, if $p=(bc,-i,-ib,-c)$ the corresponding point module is
$$
M_p \; =\; \frac{A}{Ap^\perp} \; = \; \frac{A}{A(ix_0+bcx_1)+A(ix_0+cx_2)+A(x_0+bx_3)}\,.
$$ 

Each point in $\fP_\infty$ is fixed by $\G$. If $j \ne \infty$, then $\fP_j$ is a $\G$-orbit 
and the points in the column labelled $\fP_j$ are, in descending order, $p$, $\c_1(p)$, $\c_2(p)$, and $\c_3(p)$.

The involution $\theta:\fP \to \fP$ is defined by 
\begin{equation}
\label{eq.theta}
\theta(p):= 
\begin{cases} 
p & \text{if $p \in \fP_\infty \cup \fP_0$}
\\
\c_i(p) & \text{if $p \in \fP_i$, $i=1,2,3$.}
\end{cases}
\end{equation} 
If $M_{p}$ is the point module corresponding to $p \in \fP$, then $M_{p}(1)_{\ge 0} \cong M_{\theta(p)}$. 
In particular, if $p \in \fP_j$ and $j \ne \infty$, then  $M_{p}(1)_{\ge 0} \cong M_{\c_j(p)} \cong \c_j^*M_p$.

We call point modules in $\fP_\infty$ {\sf special} and those in the other $\fP_j$'s {\sf ordinary}.

\begin{proposition}
\label{prop.ann.pts}
Let $p \in \fP-\fP_\infty$. Then $M_p$ is annihilated by
$$
\begin{cases}
4\Theta_1+(1-\b)(1+\c)\Omega   &   \text{if $p \in \fP_0$},
\\
4\b\c\Omega +(1-\b)(1+\c)\Theta_1  & \text{if $p \in \fP_1$} ,
\\
4\a\c\Omega +(1-\c)(1+\a)\Theta_2   &   \text{if $p \in \fP_2$} ,
\\
4\a\b\Omega +(1-\a)(1+\b)\Theta_3   &   \text{if $p \in \fP_3$},
\end{cases}
$$
where $\Theta_1,\Theta_2,\Theta_3$ are the central elements in (\ref{central.A.elts}).  
\end{proposition}
\begin{proof}
Since the points in $\fP_j$ form a single $\G$-orbit and the elements $y_0^2,\ldots,y_3^2$ are fixed by $\G$, 
a linear combination of $y_0^2,\ldots,y_3^2$ annihilates $M_p$ for some $p \in \fP_j$ if and only if it annihilates $M_p$ for all $p \in \fP_j$. Furthermore, if that linear combination is in the center of $A$ it annihilates $M_p$ if and only if it annihilates 
$(M_p)_0$. 

We prove the result for $j\in \{0,1\}$ and leave the reader to check the other two cases.  

Let $M \in \Gr(A)$ be the point module associated to $p=(1,1,1,1) \in \fP_0$. Let $e_0$ be a basis for $M_0$. There is an element $e_1 \in M_1$ such that
$y_ie_0 = e_1$ for $i=0,1,2,3$. Since $\theta(1,1,1,1) =(1,1,1,1)$, there is an element 
$e_2 \in M_2$ such that $y_ie_1 = e_2$ for $i=0,1,2,3$. It follows that $y_i^2 e_0 = e_2$ for $i=0,1,2,3$.
Hence
$ (y_0^2 + y_1^2+ y_2^2 +  y_3^2)e_0 = 4 e_2$. On the other hand, $\Theta_1 e_0 = (y_0^2 -\b\c y_1^2+\c y_2^2 -\b  y_3^2)e_0 = (1-\b)(1+\c) e_2$. Therefore $e_0$ is annihilated by $4\Theta_1- (1-\b)(1+\c) (y_0^2 + y_1^2+ y_2^2 +  y_3^2)$. 

Let $M \in \Gr(A)$ be the point module associated to  $p=(bc,-i,-ib,-c) \in \fP_1$. Let $e_0$ be a basis for $M_0$. There is an element $e_1 \in M_1$ such that $y_0e_0 =bc e_1$, $y_1e_0=-ie_1$, $y_2e_0=-ibe_1$, and $y_3e_0=-ce_1$. 
Since $\theta(p)=\c_1(p)=(bc,-i,ib,c)$, there is an element 
$e_2 \in M_2$ such that $y_0e_1 = bce_2$,  $y_1e_1 =-ie_2$, $y_2e_1 = ibe_2$, and $y_3e_1 = ce_2$.
Therefore $\Theta_1e_0= (y_0^2 -\b\c y_1^2+\c y_2^2 -\b  y_3^2)e_0 = 4\b\c e_2$. On the other hand, $ (y_0^2 + y_1^2+ y_2^2 +  y_3^2)e_0 = (\c+1)(\b-1)e_2$. Therefore $e_0$ is annihilated by $4\b\c\Omega +(1-\b)(1+\c)\Theta_1$.
\end{proof}

 \subsection{The actions of $H_4$ on $A_1$ and $\fP$}
 We identify $\wtQ_1$ with $S_1$ according to the convention in \S\ref{conv.identification}.
  There are corresponding identifications of $\wtQ_1^*$ and $S_1^*$, and 
of  $\PP(\wtQ_1^*)$ and  $\PP(S_1^*)$. 
If $p=(\l_0,\l_1,\l_2,\l_3)$ is a point in $S_1^*$ written with respect to the coordinate functions $(x_0,x_1,x_2,x_3)$, then $x_j(p)=\l_j$. Therefore $y_j(p)=\l_j$ if $j\in \{0,3\}$ and $y_j(p)=-i\l_j$ if $j \in \{1,2\}$. Thus, 
as a point in $\wtQ_1^*$, $p$ has coordinates $(\l_0,-i\l_1,-i\l_2,\l_3)$ with respect to the coordinate functions
$(y_0,y_1,y_2,y_3)$.

These identifications lead to actions of $H_4$ on $\wtQ_1^*$ and $\PP(\wtQ_1^*)$. 
When expressed in terms of the coordinate functions $(y_0,y_1,y_2,y_3)$, the action of $H_4$ on $\wtQ_1^*$
is also given by the  formulas in (\ref{E4.action.formulas}). 

\begin{proposition}
\label{prop.H4.moves.points.for.A}
The action of $H_4$ on $\PP(\wtQ_1^*)$ is such that:
\begin{enumerate}
  \item 
$\fP_\infty$ is a single $H_4$-orbit;  
  \item 
  if $\{i,j,k\}=\{1,2,3\}$, then $\ve_i(\fP_0) =\fP_i$, $\ve_i(\fP_i) = \fP_0$,  $\ve_i(\fP_j) = \fP_k$,  and $\ve_i(\fP_k)=\fP_j$.
\end{enumerate}
  \end{proposition}
 \begin{proof}
Case-by-case calculations using the formulas in (\ref{E4.action.formulas}) prove the result.
 For example,
\begin{align*}
 \ve_1(1,1,1,1) & \; = \;  (bc,-i,ib,c) \in \fP_1,
\\
 \ve_2(1,1,1,1) & \; = \;   (ac,a,-i,ic)  \in \fP_2,
 \\
 \ve_3(1,1,1,1) & \; = \;   (ab,ia,b,-i)\in \fP_3.
 \end{align*}
Further details are unenlightening.
 \end{proof}

\subsection{Equivariant realization of the point modules for $\wtQ$} 
In the rest of this section, the letter $F$ denotes one of the modules $\wtV(\tau+\xi)$
defined in \S\ref{ssect.2-fat-pts}.
\Cref{pr.equi_fat} shows there are four $\G$-equivariant structures on the $S'$-module $F\otimes k^2$.
\Cref{pr.equi_four} shows that the 16 ordinary point modules for $\wtQ$ are isomorphic to $(F \otimes k^2)^\G$ as $F$ varies and each $F\otimes k^2$
takes on its four equivariant structures. 
Moreover, if $F=\wtV(\tau+\xi_j)$ the four point modules $(F \otimes k^2)^\G$ belong to  $\fP_j$. 

The four special point modules for $\wtQ$ arise from the four special point modules for $S$ in a similar way. 
If $P$ is one of the four special point modules for $S$, then 
$P^{\oplus 2}\otimes k^2$ has a unique $\G$-equivariant $S'$-module structure and $(P^{\oplus 2} \otimes k^2)^\G$ is one of the 4 special point modules for $\wtQ$.

\begin{proposition}\label{pr.equi_fat}
  Up to isomorphism, there are four $\G$-equivariant structures on the $S'$-module $F\otimes k^2$. 
\end{proposition}
\begin{proof}
First, we must show that $\c^*(F\otimes k^2) \cong F\otimes k^2$ as $S'$-modules for all  $\gamma\in \G$.

By \Cref{prop.V.tau.F.tau}(3), $\c^*F\cong F$ in $\Gr(S)$.  
Since the identification of the isomorphism classes of $S$-modules with those of $S'$-modules via the Morita equivalence $-\otimes k^2$, intertwines the actions of $\c^*$  on $\Gr(S)$ and $\Gr(S')$,
$F$ is a $\G$-invariant $S$-module if and only if $F\otimes k^2$  is a $\G$-invariant $S'$-module.
Since $\c^*F\cong F$, i.e., since $F$ is $\G$-invariant, 
we can construct actions of $\G$ on
$   \Aut_{S'}(F\otimes k^2)\cong k^\times\cong \Aut_S(F)$
as in the proof of \cite[Prop. 10.2]{CS15}; both will be trivial.

{\bf Step 1: Existence of an equivariant structure.} 
By Step 1  in the proof of \cite[Prop. 10.2]{CS15}, 
the existence of an equivariant structure is controlled by an obstruction 
$c \in H^2(\G,k^\times)$ where $H^2(\G,k^\times)$ is computed 
with respect to the trivial $\G$-action on $k^\times$ from above. 

This cohomology group is isomorphic to $\bZ/2$ so it is not immediately clear that $c=0$.

In the same way, there are obstructions $c_1,c_2\in H^2(\G,k^\times)$ to the existence of a $\G$-equivariant structure on the $S$-module $F$ and the $M_2(k)$-module $k^2$ respectively. 

Since $c$ is the obstruction for $F\otimes k^2$, we get $c=c_1+c_2$ by
construction (e.g. \cite[Equation (25)]{CS15}). It is easy to see that $c_2$ is non-zero, so it remains to show that 
$c_1$ is also non-zero. In other
words, the goal is to prove that the only $\G$-equivariant structure on the $S$-module $F$ is the trivial one. This is
what \Cref{le.aux} below does.

{\bf Step 2: Classification of equivariant structures.} As in \cite[Remark 10.3]{CS15}, once we know an equivariant structure exists we also know that the set of isomorphism classes of such structures is acted upon simply transitively by $H^1(\G,k^\times)\cong \G$. 
\end{proof}

\begin{lemma}\label{le.aux}
There are no $\G$-equivariant structures on the $S$-module $F$.  
\end{lemma}
\begin{proof}
Suppose there were one. Then the space of equivariant structures would
be a homogeneous space over $H^1=H^1(\G,k^\times)\cong \G$ . 

If $n\ge 0$, then $\dim_k(F_n)=2$ so $F_n$ is a direct sum of two $\G$-eigenspaces. We denote the
constituent characters of the $\G$-action on $F_n$ by $\chi_n$ and $\eta_n$. 

Since $F$ is 1-critical, \cite[Lemma 2.10]{LS93} tells us that a central element of $S$ either annihilates 
$F$ or acts faithfully on it. By the remark after  \Cref{prop.2-dim-irreps}, some $\Omega(z)$ acts faithfully on $F$. 
Alternatively, if all $\Omega(z)$ annihilated $F$, then $F$ would be a module over 
$B=S/(\Omega_0,\Omega_1)$;
however, by the main result in \cite{AV90}, the 1-critical $B$-modules are shifts of point modules. 
Since all $\Omega(z)$ are fixed by $\G$ and have degree two, the multi-sets
$\{\chi_n,\eta_n\}$ and $\{\chi_{n+2},\eta_{n+2}\}$ are equal. 

On the one hand, equivariant structures can be twisted by
characters of $\G$ (this is the free transitive action of $H^1$ on the
set of equivariant structures). On the other hand though, $F(n)_{\ge 0} \cong F$ for all $n \ge 0$
so shifting the equivariant structure must produce the old equivariant structure twisted by some 
character. In other words, there is a character $\chi$ such that for all $n$ we have 
\begin{equation*}
  \{\chi_{n+1},\eta_{n+1}\} = \{\chi_n+\chi,\eta_n+\chi\}.
\end{equation*}
The upshot of this is that there are two characters $\eta$, $\eta'$ such that 
\begin{equation*}
  \{\chi_{n+1},\eta_{n+1}\}\cap\{\chi_n+\eta,\eta_n+\eta\}=\varnothing 
\end{equation*}
for all $n$ and similarly for $\eta'$ (take $\eta$, $\eta'$ to be distinct from either $\chi$ or $\eta_n-\chi_n+\chi$). In this case the two generators among $x_i$, $0\le i\le 3$,  that are $\eta$- and respectively $\eta'$-eigenvectors for the action of $\G$ must annihilate $F$. But this does not happen because all $x_i$ act faithfully on all $V(\tau+\xi)$. 
\end{proof}

For each $\xi_j \in E[2]$, the four equivariant structures on $\wtV(\tau+\xi_j) \otimes k^2$,  
give rise, by descent, to the four point modules for  $\wtQ$ that belong to $\fP_j$. 

\begin{corollary}
\label{cor.equi_fat}
Fix $j \in \{0,1,2,3\}$ and let $F=\wtV(\tau+\xi_j)$. If $p \in \fP_j$, there is a unique equivariant structure on 
the $S'$-module $F \otimes k^2$ such that $M_p \cong (F \otimes k^2)^\G$. 
\end{corollary}
\begin{proof}
By \Cref{pr.equi_fat}, there are four pairwise non-isomorphic $\G$-equivariant $S'$-module structures on $F\otimes k^2$. 
Hence, by \cite[Proposition 3.2]{CS15}, the four $\wtQ$-modules $(F\otimes k^2)^\G$ are pairwise non-isomorphic. 
Since  $\Omega(\tau+\xi_j)$ annihilates $F$ and $F \otimes k^2$ it also annihilates each  
$(F \otimes k^2)^\G$. 
\end{proof}

\begin{proposition}\label{pr.equi_four}
If $P$ is a special point module for $S$, then there is a unique $\G$-equivariant structure on the $S'$-module $(P\oplus P)\otimes k^2$. 
\end{proposition}
\begin{proof}
Recall that $P \cong S/I$ where $I$ is an  ideal generated by three of the $x_i$'s. 

Fix elements $\gamma$ and $\delta$ such that  $\G=\langle \c\rangle \times \langle \d \rangle$. 
Fix an ordered basis $\{v,v'\}$ for $k^2$ such that the actions of $\gamma$ and $\d$ on $M_2(k)$ are conjugation by $\big({{1 \,\,\phantom{-}0}\atop{0 \,\, -1}}\big)$ and $\big({{0 \,\ 1}\atop{1 \,\,\, 0}}\big)$ respectively.

{\bf Existence.} The simple structure of $P$ will make it easy to check directly that the following construction really is an equivariant structure. 

Make $\gamma$ act on $(P_0\oplus P_0)\otimes k^2 = (P_0\oplus P_0)\otimes v \oplus (P_0\oplus P_0)
\otimes v'$ as $1$ on the two outermost summands and as $-1$ on the two innermost ones. 
This ensures that the action of $\gamma$ on $(P_0\oplus P_0)\otimes k^2$ anti-commutes with the action of
$1 \otimes \big({{0 \,\ 1}\atop{1 \,\,\, 0}}\big) \in S \otimes M_2(k)$, as it should. 

Next, let $\delta$ interchange the two outermost $P_0$'s and also the two innermost ones, thus making the action of $\delta$ on $(P_0\oplus P_0)\otimes k^2$ commute with the actions of both $\gamma$ and $1 \otimes \big({{0 \,\ 1}\atop{1 \,\,\, 0}}\big)$. This action of $\G$ 
on the degree-zero  component of $(P\oplus P)\otimes k^2$ can be extended uniquely (the uniqueness follows from \cite[Lemma 10.3]{CS15}) 
to $(P\oplus P)\otimes k^2$ so as to make it compatible with the $S'$-action. 

{\bf Uniqueness.} Since $\gamma$ commutes with the idempotents $\big({{1\, \,\,\phantom{}0}\atop{0 \,\,\, 0}}\big)$ and
$\big({{0\, \,\,\phantom{}0}\atop{0 \,\,\, 1}}\big)$ , it must implement $\langle\gamma\rangle$-equivariant structures on the two copies $M$ and $N$ of $P\oplus P$ in $(P\oplus P)\otimes k^2$. Similarly, $\delta$ must interchange these two subspaces of $(P\oplus P)\otimes k^2$.

Because $\big({{0 \,\ 1}\atop{1 \,\,\, 0}}\big)$ anti-commutes with $\gamma$, it interchanges its $1$ and $-1$-eigenspaces. There are now two possibilities: Either $M_0$ and $N_0$ are eigenspaces for $\gamma$ with opposite eigenvalues, or both break up as direct sums of one-dimensional $1$ and $-1$-eigenspaces. The former case is impossible because then $\delta$, which commutes with $\gamma$, would not leave its eigenspaces invariant. This means that we are in the latter case, and we can choose decompositions $M\cong P\oplus P$ and $N\cong P\oplus P$ that make the present equivariant structure agree with the one constructed explicitly above. 
\end{proof}

\section{Commuting lines in $\Gr(\wtQ)$ and commuting subspaces of $\wtQ_1$}
\label{se.commuting_conic}

This section concerns line modules that correspond to pairs of commuting elements in $\wtQ_1$.

\subsection{}
A 2-dimensional subspace $ky+ky' \subseteq \wtQ_1$ is called a {\sf commuting subspace} of $\wtQ_1$ if $[y,y']=0$.
Such subspaces exist. For example, $[iy_0+bcy_1,cy_2+iby_3]=  [iy_0-bcy_1,cy_2-iby_3] =0$.\footnote{The automorphism $\c_2:\wtQ\to \wtQ$ sends the commuting pair $\{iy_0+bcy_1,cy_2+iby_3\}$ to $\{iy_0-bcy_1,cy_2-iby_3\}$.}

If $ky+ky'$ is a commuting subspace of $\wtQ_1$ we call the line $y=y'=0$ in $\PP(\wtQ_1^*)$ a {\sf commuting line}.

\Cref{prop.commute} shows that the set of commuting lines in $\PP(\wtQ_1^*)$, and therefore the set of commuting 
subspaces, is parametrized by a smooth conic in the Grassmannian $\GG(1,3)$. 

\begin{proposition}
If $ky+ky'$ is a commuting subspace of $\wtQ_1$, then $\wtQ/\wtQ y+\wtQ y'$ is a line module.
\end{proposition}
\begin{proof} 
By \cite[Prop.2.8]{LS93},  if 
$s,t,u,v \in \wtQ_1$ are non-zero elements such that  $st=uv$  and $\{t,v\}$ is linearly independent, 
then $\wtQ/\wtQ t+\wtQ v$ is a line module.\footnote{The result in  \cite[Prop.2.8]{LS93} is stated for a class of algebras that does not include $\wtQ$ but its proof depends only on the 
good homological properties of the ring, its Hilbert series, and the fact that the ring in question is generated in degree one, 
so the result applies to $\wtQ$.} In particular, if $ky+ky'$ is a commuting subspace of $\wtQ_1$, then $yy'=y'y$ so  
$\wtQ/\wtQ y + \wtQ y'$ is a line module. 
\end{proof} 

We call line modules of this form {\sf commuting line modules}. 

\subsection{The commuting conic}
\label{ssect.comm.conic}  
The set of lines in $\PP(\wtQ_1^*)$ is a projective variety, the Grassmanian $\GG(1,3)$, which we will view as a closed subvariety of 
$\PP^5=\PP(\wedge^2 \wtQ_1^*)$.
The Pl\"ucker coordinates $z_{01},z_{02},z_{03},z_{12},z_{13},z_{23},$ on this $\PP^5$ are  defined as follows.

The Pl\"ucker coordinates of  the line $\sum_{j=0}^3 \l_jy_j=\sum_{j=0}^3\l_j' y_j=0$ in $\PP(\wtQ_1^*)$ or, more precisely, of the point in $\GG(1,3)  \subseteq \PP^5$ that corresponds to it, are $(M_{01},M_{02},M_{03},M_{12},M_{13},M_{23})$ where $M_{ij}$ is the determinant of the $2 \times 2$ sub-matrix of
\begin{equation}
\label{line.matrix}
M:=\begin{pmatrix}
 \l_0 & \l_1 & \l_2 & \l_3   \\
 \l_0' & \l_1' & \l_2' & \l_3'
\end{pmatrix}
\end{equation}
consisting of columns $i$ and $j$ where the columns are labelled 0, 1, 2, 3, starting 
from the left. For example,  $M_{01}=\l_0\l_1'-\l_0'\l_1$. It is also convenient to define $M_{ji}=-M_{ij}$ and $z_{ji}=-z_{ij}$.  

The Pl\"ucker coordinates satisfy the Pl\"ucker relation $z_{01}z_{23} + z_{02}z_{31}+z_{03}z_{12} = 0$. Indeed, $\GG(1,3)$
is the quadric hypersurface in $\PP^5$ cut out by this equation.

 In anticipation of the next result we call the subvariety of $\PP^5$  given by (\ref{eq.commuting.conic}) the {\sf commuting conic}.  We denote it by $C_0$.  It parametrizes the isomorphism classes of commuting line modules.

\begin{proposition}
\label{prop.commute}
A line $y=y'=0$ in $\PP(\wtQ_1^*)$ is a commuting line if and only if the corresponding point in $\GG(1,3)$ 
 lies on the smooth conic
\begin{equation}
\label{eq.commuting.conic}
z_{01}z_{23} + z_{02}z_{31}+z_{03}z_{12}   \; = \; z_{23} +\a z_{01} \; = \; z_{31} +\b z_{02} \; = \;  z_{12}+\c z_{03} \; = \; 0.
\end{equation}
There is an isomorphism $\psi: \PP^1 \to C_0$ given by the formulas
$$
\psi(s,t) \; = \; \big(ia^{-1}(s^2+t^2), \, 2b^{-1}st, \, c^{-1}(s^2-t^2),  \, -c(s^2-t^2),   \,    2bst, \,   -ia(s^2+t^2)            \big)
$$
and  $\psi^{-1}(z_{01}, \ldots, z_{23})= (cz_{03}-iaz_{01},bz_{02})$.
\end{proposition}
\begin{proof}
Let $y=\l_0 y_0+\cdots +\l_3y_3$, $y'=\l_0' y_0+\cdots +\l_3' y_3$, and let $M$ be the matrix in (\ref{line.matrix}).
A calculation shows that
$$
[y,y']  \; = \;   \sum_{i,j=0}^3 \l_i\l_j'[y_iy_j]   \; = \;   \sum_{0\le i<j \le 3} (\l_i\l_j'-\l_i'\l_j)[y_i,y_j] \; = \;   \sum_{0\le i<j \le 3} M_{ij}[y_i,y_j].
$$
Since $[y_0,y_j]=\a_j[y_k,y_i]$ whenever $(i,j,k)$ is a cyclic permutation of $(1,2,3)$ this sum equals
$$
(\a M_{01}+M_{23})[y_2,y_3]   \, + \,  (\b M_{02}+M_{31})[y_3,y_1]  \, + \,  (\c M_{03}+M_{12})[y_1,y_2] .
$$
Thus, $[y,y']=0$ if and only if 
\begin{equation}
\label{P2.in.G13}
 \a M_{01}+M_{23} \; = \; \b M_{02}+M_{31} \; = \; \c M_{03}+M_{12} \; = \; 0.
 \end{equation}
The minors satisfy the Pl\"ucker relation  so $[y,y']=0$ if and only if the corresponding point in $\GG(1,3)$ lies on the subvariety of 
$\PP^5=\PP(\wedge^2\wtQ_1^*)$ given by (\ref{eq.commuting.conic}). We complete the proof by showing this subvariety is a smooth conic.

The equations $ z_{23} +\a z_{01} =z_{31}+\b z_{02}=z_{12}+\c z_{03} =0$ cut out a $\PP^2$ in $\PP^5$ so the commuting conic
is the zero locus of the quadratic form $z_{01}z_{23} + z_{02}z_{31}+z_{03}z_{12}$ on this $\PP^2$. By using the coordinate functions $z_{01},z_{02},z_{03}$, for this $\PP^2$ and making the substitutions  $z_{23} =-\a z_{01}$,
 $z_{13}=\b z_{02}$,  and $z_{12}=-\c z_{03}$, this quadratic form becomes $\a z_{01}^2+ \b z_{02}^2 + \c z_{03}^2$
 which is non-degenerate because $\a\b\c \ne 0$. Thus, the commuting conic is smooth as claimed.
 
We leave the reader to check that $\psi$ is an isomorphism with the claimed inverse. 
\end{proof}

\begin{proposition}
The commuting conic is cut out by the equations $ z_{23} +\a z_{01} =z_{31}+\b z_{02}=z_{12}+\c z_{03} =0$
and any one of the following four equations: 
\begin{align*}
\a z_{01}^2+ \b z_{02}^2 + \c z_{03}^2 & \; = \; 0,  \quad \hbox{or}
\\
\a\b\c z_{01}^2+ \b z_{12}^2 + \c z_{13}^2 & \; = \; 0,  \quad \hbox{or}
\\
\a\b\c z_{02}^2+ \a z_{12}^2 + \c z_{23}^2 & \; = \; 0,  \quad \hbox{or}
\\
\a\b\c z_{01}^2+ \a z_{13}^2 + \b z_{23}^2 & \; = \; 0. 
\end{align*}
\end{proposition}
\begin{proof}
One obtains these equations by using  the equations $z_{23} +\a z_{01} =z_{31}+\b z_{02}= z_{12}+\c z_{03} =0$
and making the obvious substitutions into the equation  $z_{01}z_{23} + z_{02}z_{31}+z_{03}z_{12}=0$. 
\end{proof}

\begin{proposition}
\label{prop.commuting.quadric}
The commuting lines $y=y'=0$ in $\PP(\wtQ_1^*)$ are precisely the lines in one of the rulings on the quadric
$y_0^2+\b\c y_1^2 + \c\a y_2^2+ \a\b y_3^2=0$. 
\end{proposition}
\begin{proof}
The equation  for the quadric is $(y_0+ibcy_1)(y_0-ibcy_1)=\a(icy_2+by_3)(icy_2-by_3)$.
For all $t \in \PP^1$, the line
\begin{equation}
\label{eq.comm.lines}
(y_0+ibcy_1)-ta(icy_2-by_3) =  t(y_0-ibcy_1)-a(icy_2+by_3) =0
\end{equation}
lies on the quadric. As $t$ varies over $\PP^1$ one obtains all the lines in one of the rulings.
The Pl\"ucker coordinates of the line (\ref{eq.comm.lines}) are given by the $2 \times 2$ minors of the matrix
$$
\begin{pmatrix}
  1    & ibc & -ia ct & a b t   \\
   t   & -ibct & -iac & -ab 
\end{pmatrix}.
$$
The corresponding point in $\GG(1,3)$, namely
$$
\big( -2ibct, \, iac(t^2-1), \, -ab(t^2+1), \, ab\c( t^2+1), \, iac\b ( t^2-1), \, 2i\a bct  \big),
$$
lies on the plane $ z_{23} +\a z_{01} = z_{31}+\b z_{02} = z_{12}+\c z_{03}  =  0$ and hence on the commuting conic.
\end{proof}

\begin{proposition}
\label{prop.commuting.quadric.2}
\label{thm.pts.on.comm.lines}
Let $p \in \fP$. 
\begin{enumerate}
  \item 
The point $p$ lies on the quadric in \Cref{prop.commuting.quadric} if and only if $p \in \fP_1 \cup \fP_2  \cup \fP_3$.
  \item 
  If $p \in \fP_1 \cup \fP_2 \cup \fP_3$, there is a unique commuting line passing through $p$.
  \item{}
  None of the points in $\fP_\infty \cup \fP_0$ lies on a commuting line. 
\end{enumerate}
\end{proposition}
\begin{proof}
Part (1)  is a simple calculation. Parts (2) and (3) follow from the fact that a smooth quadric in $\PP^3$ is the
disjoint union of the lines in one of the rulings on it.
\end{proof}

\subsection{Commuting subspaces of $\wtQ_1$}
Just as the commuting lines provide a ruling on the quadric $y_0^2+\b\c y_1^2 + \c\a y_2^2+ \a\b y_3^2=0$ in $\PP(\wtQ_1^*)$, the commuting subspaces of $\wtQ_1$ provide a ruling on a quadric in $\PP(\wtQ_1)$. 
We determine that quadric after the following routine lemma. 

Let $V$ be a finite dimensional vector space and $V^*$ its dual. If $L$ is a linear subspace of $\PP(V)$ we write $L^\perp$
for the linear subspace of $\PP(V^*)$ that vanishes on $L$.

\begin{lemma}
\label{lem.orthog.rulings}
Let $\l_0,\l_1,\l_2,\l_3 \in k^\times$. 
Let $V$ be a 4-dimensional vector space with basis $y_0,\ldots,y_3$. Let $w_0,\ldots,w_3$ be the dual basis for $V^*$. 
Let $S \subseteq \PP(V^*)$ and $S' \subseteq \PP(V)$
be the smooth quadrics defined by  
$\l_0^2y_0^2+\l_1^2y_1^2+\l_2^2y_2^2+\l_3^2y_3^2=0$  and
$\l_0^{-2}w_0^2+\l_1^{-2}w_1^2+\l_2^{-2}w_2^2+\l_3^{-2}w_3^2=0$.
\begin{enumerate}
  \item 
   For each $t \in \PP^1$, the lines 
  \begin{align*}
\phantom{xxxxxx}   L_t:  \; \; &
  (\l_0y_0 +i \l_1 y_1)-t (i\l_2y_2 + \l_3 y_3)\;=\; t(\l_0y_0 -i \l_1 y_1) -(i\l_2y_2 - \l_3 y_3)=0 \quad \hbox{and}
  \\
 \phantom{xxxxxx}   L'_t:  \; \;&
   (\l_0^{-1}w_0 +i \l_1^{-1} w_1)-t  (i\l_2^{-1}w_2 + \l_3^{-1} w_3)\;=\; t(\l_0^{-1}w_0 -i \l_1^{-1} w_1) - (i\l_2^{-1}w_2 - \l_3^{-1} w_3) =0
  \end{align*}
  lie on $S$ and $S'$ respectively.
   \item 
  $L_t'=L_t^\perp$.
\end{enumerate}
\end{lemma}
\begin{proof}
It becomes obvious that $L_t$ lies on $S$ when we write the equation for $S$ as
$$
(\l_0y_0 +i \l_1 y_1)(\l_0y_0 -i \l_1 y_1) = (i\l_2y_2 + \l_3 y_3)(i\l_2y_2 - \l_3 y_3).
$$
The points $(\l_0^{-1},i\l_1^{-1},-it\l_2^{-1},-t\l_3^{-1})$ and $(-t\l_0^{-1},it\l_1^{-1},i\l_2^{-1},-\l_3^{-1})$
 lie on $L_t$.
 
 Similarly, the line $L'_t$ lies on $S'$ and passes through  the two points $(\l_0,i\l_1,-it\l_2,-t\l_3)$ and 
 $(-t\l_0,it\l_1,i\l_2,-\l_3)$, both of which belong to  
$$
(\l_0^{-1},i\l_1^{-1},-it\l_2^{-1},-t\l_3^{-1})^\perp 
\cap 
(-t\l_0^{-1},it\l_1^{-1},i\l_2^{-1},-\l_3^{-1})^\perp \; = \; L_t^\perp.
$$
Thus $L'_t \subseteq L_t^\perp$ and, for dimension reasons, $L'_t = L_t^\perp$.
\end{proof}
 
\begin{proposition}
\label{prop.comm.subspace.ruling}
Let $w_0,w_1,w_2,w_3$ be the basis for $\wtQ_1^*$ dual to the basis $y_0,y_1,y_2,y_3$. 
The commuting subspaces of $\wtQ_1$ form a ruling on the quadric 
$\a\b\c w_0^2 + \a w_1^2 + \b w_2^2 +\c w_3^2=0$ in $\PP(\wtQ_1)$.  The commuting subspace vanishing on
 the commuting line in (\ref{eq.comm.lines})  is
\begin{equation}
\label{eq.comm.subspace}
(abcw_0+iaw_1)-t(ibw_2-cw_3) =  t(abcw_0-iaw_1)-(ibw_2+cw_3) =0.
\end{equation}
\end{proposition}
\begin{proof} 
In order that the line in (\ref{eq.comm.lines}) be the line $L_t$ in   \Cref{lem.orthog.rulings}, 
we must take $(\l_0,\l_1,\l_2,\l_3)=(1,bc,ac,-ab)$. Hence  
 $ (\l_0^{-1},\l_1^{-1},\l_2^{-1},\l_3^{-1})=(abc)^{-1}(abc,a,b,-c)$.
The line $L_t^\perp$, i.e., the commuting subspace vanishing on $L_t$, is the line $L_t'$ in \Cref{lem.orthog.rulings}
so is given by (\ref{eq.comm.subspace}).  
\end{proof}

\begin{lemma}
There are no commuting subspaces in $S_1$.  
\end{lemma}
\begin{proof}
Let $x=\ve x_0+\l x_1+\mu x_2+\nu x_3$ and $x'=\ve' x_0+\l' x_1+\mu' x_2+\nu' x_3$.  
We use the same notation as in \Cref{prop.commute} for the $2 \times 2$ minors of
$$
M=
\begin{pmatrix}
\ve & \l & \mu & \nu  \\
\ve' & \l' & \mu' & \nu' 
\end{pmatrix}.
$$

A computation gives
\begin{align*}
[x,x'] & \;=\; M_{01}[x_0,x_1] \, + \,  M_{02}[x_0,x_2] \, + \,  M_{03}[x_0,x_3] \, + \, M_{12}\{x_0,x_3\} \, - \, M_{13}\{x_0,x_2\} \, + \,  M_{23}\{x_0,x_1\}
\\
&  \;=\;  (M_{01}+M_{23})x_0x_1 + (M_{23} - M_{01})x_1x_0 + (M_{02}-M_{13})x_0x_2 - (M_{13} + M_{02})x_2x_0 
\\
&\;  \phantom{(M_{01}+M_{23})x_0x_1} \, + \, (M_{03}+M_{12})x_0x_3 \,+\, (M_{12} - M_{03})x_3x_0 
\end{align*}
The relations for $S$ can be written as 
$$
x_jx_k= \hbox{$\frac{1}{2}$}\Big(\alpha_i^{-1}[x_0,x_i]+\{x_0,x_i\}\Big) \qquad x_kx_j= \hbox{$\frac{1}{2}$}\Big(\alpha_i^{-1}[x_0,x_i]-\{x_0,x_i\}\Big) 
$$
as $(i,j,k)$ runs over all cyclic permutations of $(1,2,3)$ so $\{x_j^2, x_0x_i, x_ix_0 \; | \; 1 \le i \le 3, \, 0 \le j \le 3\}$
is a basis for $S_2$. Hence if $[x,x']=0$, then all the $2 \times 2$ minors of $M$ vanish. This implies that $x$ and $x'$ are linearly dependent. Therefore $S_1$ does not contain a commuting subspace. 
\end{proof}

In retrospect it is clear why $\wtQ_1$ but not $S_1$ contains commuting subspaces. 
The relations for $\wtQ$ are such that the
subspace of $\wtQ_2$ spanned by the commutators $\{[y,y'] \; | \; y,y' \in \wtQ_1\}$ has dimension 3 whereas 
the subspace of $S_2$ spanned by   $\{[x,x'] \; | \; x,x' \in S_1\}$ has dimension 6. This is due to the fact that $\wtQ$ has 
relations $[y_0,y_i]=\a_i[y_j,y_k]$ but $S$ has  relations $[x_0,x_i]=\a_i\{x_j,x_k\}$.

\section{More conic line modules via quantum symmetries}\label{se.other_conics} 
\label{sect.quantum.symm}

\subsection{}
The commuting conic, $C_0$,  parametrizing  the commuting line modules  is ``biased'' as far as point-line incidence is concerned: by \Cref{prop.commuting.quadric.2}, only points in the families 
$\fP_1 \cup \fP_2 \cup \fP_3$ lie on a commuting line. 
This section cures this bias by finding three more conics $C_1,C_2,C_3 \subseteq \GG(1,3)$ that parametrize 
line modules with the property that every  point in $\fP-(\fP_\infty \cup \fP_j)$ lies on exactly one line in $C_j$ and 
no points in $ \fP_\infty \cup \fP_j$ lie on a line in $C_j$.

The new conics will be obtained by ``moving $C_0$ around''. One possible meaning for this phrase would be to take an automorphism $\psi \in \Aut(\wtQ)$ and send a line module $L$  to $\psi^*L$. This  does not do what we want  because $\psi$ sends  commuting  subspaces to  commuting subspaces whence $\psi^*L$ belongs to $C_0$ if $L$ does.  In other words, $\wtQ$ does not have enough symmetries to ``move'' $C_0$. Nevertheless, we will show that $\wtQ$ has enough ``quantum symmetries'' for this purpose, in a sense that will be made precise below. The quantum symmetries are, in effect, auto-equivalences of $\Gr(\wtQ)$.

\subsection{} 
\label{ssect.construction.of.S-tilde}
When $G$ is a finite group we write $k(G)$ for the  algebra of $k$-valued functions 
on $G$ and $kG$ for its group algebra;  $k(G)$ and $kG$ are mutually dual Hopf algebras. 

The construction of $\wtQ$ from $S$ can be described in the following way. 
First regard $S$ as a $k(H_4)$-comodule algebra via the homomorphism  $H_4 \to \Aut(S)$ defined in
\Cref{prop.aut.S}. In other words, $S$ is an algebra object in the category $\cM^{k(H_4)}$ of right $k(H_4)$-comodules.
Then apply to $S$ the monoidal functor $\cM^{k(H_4)}\to \Vect$ corresponding to the 2-cocycle 
  $\sigma:k(H_4)^{\otimes 2}\to k$ that is the composition of the homomorphism $k(H_4)^{\otimes 2}\to k(\G)^{\otimes 2}$
  induced by the inclusion $\G=\langle \ve_1^2,\ve_2^2\rangle \subseteq H_4$ with the cocycle
 $\mu:k(\G)^{\otimes 2}= (k\widehat{\G})^{\otimes 2} \to k$ in \S\ref{ssect.davies.cocycle}. 
 By monoidality, the image of $S$ under this functor is an algebra in $\Vect$, namely $\wtQ$. 

By the general formalism of how cocycles classify such monoidal functors (as covered in \cite{Bic10}, say), our functor $\cM^{k(H_4)}\to \Vect$ factors as  
\begin{equation}\label{eq.Tan}
  \begin{tikzpicture}[auto,baseline=(current  bounding  box.center)]
    \path[anchor=base] (0,0) node (1) {$\cM^{k(H_4)}$} +(4,0) node (2) {$\Vect$} +(2,-1) node (3) {$\cM^H$};
         \draw[->] (1) to  (2);     
         \draw[->] (1) to[bend right=10]  node[left,pos=.55] {$\equiv\phantom{i}$}  (3);
         \draw[->] (3) to[bend right=10]  node[right,pos=.1] {$\phantom{xx}\text{forget}$}  (2);
  \end{tikzpicture}
\end{equation}  
where $H$ is defined below and all arrows are monoidal functors and 
$\cM^{k(H_4)}\to \cM^H$ is a monoidal equivalence, as indicated. 

\subsection{The Hopf algebra $H$}
\label{ssect.Hopf.alg.H}
The ``quantum symmetries'' of $\wtQ$ alluded to in the title of this section will be implemented by a new Hopf algebra 
$H$ which is a ``deformation'' of the Hopf algebra $k(H_4)$ by the cocycle $\s$. The construction
is due to Doi  \cite{Doi93}: 
\begin{enumerate}
  \item 
as a vector space, $H=k(H_4)$;  
  \item 
the coalgebra structure on $H$ is exactly the same as that on $k(H_4)$;  
  \item 
 as an algebra, $H$ is $k(H_4)$ with the new multiplication 
\begin{equation*}
  r*s \; :=\;  \sigma^{-1}(r_1\otimes s_1)\sigma(r_3\otimes s_3)r_2s_2,
\end{equation*}
where $rs$ is the product in $k(H_4)$, $\sigma^{-1}:k(H_4)^{\otimes 2}\to k$ is the convolution-inverse of $\sigma$, and we are using Sweedler's convention
$(\D \otimes \id)\D(r)=r_1 \otimes r_2 \otimes r_3$ with an implied summation. 
\end{enumerate} 

As we will see more explicitly below, 
$H$ is not commutative so is not isomorphic as a Hopf algebra to $k(G)$ for any group $G$. 
Instead, we think of  $H$ as {\it the ring of $k$-valued functions on a finite quantum group}
whose action as ``automorphisms'' of $\wtQ$ is made manifest by an algebra homomorphism $\wtQ\to \wtQ\otimes H$
that  implements the quantum symmetries in the title of this section. 

Since $A$ is obtained from $S$ by deforming via the cocycle $\sigma$, $A$ is the image of $S$ through the horizontal map in \Cref{eq.Tan}, and hence an $H$-comodule algebra (see \cite[pp. 25-26]{Bic10} for more about this).

\subsection{}
We now explain how these quantum symmetries (i.e., the action of the quantum group on $\Gr(\wtQ)$)
produce(s) new line modules. 

Let $M$ be a left $\wtQ$-module and $V$ a left $H$-module. Then $M\otimes V$ is  a left $\wtQ\otimes H$-module and therefore a left $\wtQ$-module via the comodule structure map $\wtQ\to \wtQ\otimes H$.
If $M$ is graded so is $M \otimes V$ with $\deg(M_n \otimes V)=n$. 
Specializing to the graded case,  we get an action 
\begin{equation}\label{eq.cat_act}
  \Gr(\wtQ)\, \boxtimes  \,\Mod(H)\to \Gr(\wtQ)
\end{equation}
of the monoidal category $\Mod(H)$ on the category $\Gr(\wtQ)$ for a categorical tensor product construction $\boxtimes$ that we will not make precise here.

\subsubsection{}
Let $\gamma:H\to k$ be an algebra homomorphism and let $k_\gamma$ be the corresponding 1-dimensional $H$-module.
Since $\gamma$ is a group-like element in the dual Hopf algebra $H^*$ it implements an automorphism, $\varphi_\gamma$ say,
of $\wtQ$ as a graded algebra. The action of  $k_\gamma$  on $\Gr(\wtQ)$ via (\ref{eq.cat_act})
is the same as the auto-equivalence that twists every module and morphism by $\varphi_\gamma$. 
We have already observed that such a twist applied to a line module in $C_0$  produces another line module in $C_0$, 
so that there is no chance of discovering new families of line modules in this fashion. However, since $H$ is not commutative, we might hope that for a line module $M \in \Gr(\wtQ)$ corresponding to a point in the commuting conic $C_0$ and some simple $H$-module $V$ of dimension $\ge 2$ the tensor product $M\otimes V$ (obtained by acting with $V\in \Mod(H)$ on $M\in\Gr(\wtQ)$ via \Cref{eq.cat_act}) might break up (decompose) as a direct sum of several line modules in these other as-yet hypothetical families $C_j$, $1\le j\le 3$. We will see presently that this is indeed what happens. 

\subsection{}
We first need a better understanding of the algebra structure on $H$. It is simpler to do this dually by understanding the co-multiplication 
on the dual Hopf algebra $H^*$ instead; we will then freely switch points of view between the equivalent categories of left $H$-modules and right $H^*$-comodules. In \Cref{eq.cat_act}, for instance, we might substitute $\cM^{H^*}$ for $\Mod(H)$. 

Since $H=k(H_4)$ we can, and will, use the usual isomorphism $kH_4 \to k(H_4)^*$ to identify $H^*$, as a vector space, with the group algebra $kH_4$. We will therefore write $H^*=kH_4$. 
Although the coproducts on $H^*$ and $kH_4$ are different the relation between them is quite simple.

\begin{lemma}
Let  $\xi_i=\varepsilon_i^2$. The coproduct on $H^*$ is the usual group algebra coproduct on $kH_4$ followed by conjugation in $kH_4 \otimes kH_4$ by the involutive element
$$
  J \;  := \; \hbox{$\frac{1}{2}$}
  \big( 1 \otimes 1\, + \, \xi_1\otimes 1\, + \,1\otimes \xi_2 \,- \,\xi_1\otimes \xi_2\big) \; \in \; k\G^{\otimes 2}  \;  \subseteq \; kH_4^{\otimes 2}.
$$
\end{lemma}
\begin{proof}
 This can be seen by using the explicit form of the cocycle on $k(\G) = k\widehat{\G}$ in \S\ref{ssect.davies.cocycle}.
 \end{proof}

Using the fact that conjugation by $\xi_2$ multiplies $\varepsilon_1$ by $\eta=\delta^2$ and $\xi_1=\varepsilon_1^2$ commutes with $\varepsilon_1$,   the comultiplication $\Delta:H^* \longrightarrow H^* \otimes H^*$   is such that
\begin{align*}
\label{eq.C1}
\Delta(  \varepsilon_1) & \;= \;  
  \hbox{$\frac{1}{2}$} 
  \big(
  \varepsilon_1\otimes\varepsilon_1
  \phantom{xxxi} \, \,  \, + \;
   \varepsilon_1\otimes\varepsilon_1\eta 
  \phantom{xxxi} \, +\, 
  \varepsilon_1^{-1}\otimes\varepsilon_1
\phantom{xx}  \; - \;
 \varepsilon_1^{-1}\otimes \varepsilon_1\eta \big)
  \nonumber 
      \\
\Delta(  \varepsilon_1\eta) & \;= \;  
  \hbox{$\frac{1}{2}$} 
  \big(\varepsilon_1\eta\otimes\varepsilon_1\eta
  \phantom{xii} \, \,  +\;
  \varepsilon_1\eta\otimes\varepsilon_1
    \phantom{xxii} \, \,  +\;
    \varepsilon_1^{-1}\eta\otimes\varepsilon_1\eta
    \; -\;
    \varepsilon_1^{-1}\eta\otimes\varepsilon_1\big)
   \\
\Delta(  \varepsilon_1^{-1}) & \;= \;     \hbox{$\frac{1}{2}$} 
  \big(  \varepsilon_1^{-1}\otimes\varepsilon_1^{-1}
    \phantom{xii}  \, +\;
    \varepsilon_1^{-1}\otimes\varepsilon_1^{-1}\eta
 \phantom{ix}     + \;
    \varepsilon_1\otimes \varepsilon_1^{-1}
  \phantom{xi} \, \; - \;  \varepsilon_1\otimes\varepsilon_1^{-1}\eta 
  \big)
  \nonumber
    \\
\Delta(  \varepsilon_1^{-1}\eta) & \;= \;    \hbox{$\frac{1}{2}$} 
  \big(
  \varepsilon_1^{-1}\eta\otimes\varepsilon_1^{-1}\eta
 \,\, + \, \varepsilon_1^{-1}\eta\otimes\varepsilon_1^{-1}
\phantom{xi}  \; + \; 
 \varepsilon_1\eta\otimes \varepsilon_1^{-1}\eta
 \; - \; \varepsilon_1\eta\otimes \varepsilon_1^{-1} 
  \big)
   \nonumber
  \end{align*}
 
\begin{lemma}
\label{lem.comult.on.H*}
The subspaces 
\begin{align*}
D_1 & \; :=\; \rm{span}\{\ve_1^{\pm 1}, \ve_1^{\pm 1}\eta\},
\\
D_2 & \; :=\; \rm{span}\{\ve_2^{\pm 1}, \ve_2^{\pm 1}\eta\},
\\
D_3 & \; :=\; \rm{span}\{\  \varepsilon_1\varepsilon_2,\ \varepsilon_1^{-1}\varepsilon_2^{-1},\ \varepsilon_1\varepsilon_2\eta,\ \varepsilon_1^{-1}\varepsilon_2^{-1}\eta\},
\end{align*}
of $H^*$ are sub-coalgebras and each is isomorphic to $M_2(k)$ as a coalgebra. For example, the elements
\begin{equation}
\label{eq.abcd}
\begin{cases}
\quad
  a  \; := \;  \hbox{$\frac{1}{2}$} 
  \big( 
   \varepsilon_1+\varepsilon_1\eta
   \big)
   \qquad 
   \phantom{xxxxxx}
    b    \; := \;  \hbox{$\frac{1}{2}$} 
  \big( 
 \varepsilon_1-\varepsilon_1\eta
    \big)
  \\
\quad  c      \; := \;   \hbox{$\frac{1}{2}$} 
  \big( 
\varepsilon_1^{-1}-\varepsilon_1^{-1}\eta
\big)
\qquad
   \phantom{xxxx}
 d   \; := \; \hbox{$\frac{1}{2}$} 
  \big( 
\varepsilon_1^{-1}+\varepsilon_1^{-1}\eta
\big)
\end{cases}
\end{equation}
are matrix co-units for $D_1$ in the sense that the comultiplication $\D$ on $M_2(H^*)$ has the property that
\begin{equation}\label{eq.matrix_co}
\D  \begin{pmatrix}
    a&b\\c&d
  \end{pmatrix}
\;=\;
  \begin{pmatrix}
    a&b\\c&d
  \end{pmatrix}
  \otimes
  \begin{pmatrix}
    a&b\\c&d
  \end{pmatrix}.
\end{equation}
\end{lemma}

\begin{lemma}
\label{lem.subgp.G}
The restriction of the comultiplication on $H^*$ to $k \langle \ve_1^2,\ve_2^2,\d\rangle \subseteq  kH_4=H^* $ 
is the usual group algebra comultiplication.
\end{lemma}
\begin{proof}
The subgroup $G:=\langle \ve_1^2=\xi_1,\ve_2^2=\xi_2,\d\rangle$ of $H_4$ is abelian.
Since $J$ belongs to $kG \otimes kG$ it commutes with $\ve_1^2$, $\ve_2^2$, and $\d$. 
Now apply \Cref{lem.comult.on.H*}.
\end{proof}

 Multiplying $D_1$, $D_2$ and $D_{3}$ in $H^*$ by elements of $\langle \ve_1^2,\ve_2^2,\d\rangle$ 
 we get translates that are again $2\times 2$ matrix subcoalgebras of $H$. It is easy to see that each $D$ has four such translates, and the twelve translates are pairwise distinct.

\begin{proposition}
 As a coalgebra, $H^*$ is a direct sum of 16 one-dimensional coalgebras spanned by 
 the elements in $\langle \ve_1^2,\ve_2^2,\d\rangle$ and twelve $2\times 2$ 
 matrix subcoalgebras. In particular, $H^*$ is cosemisimple. 
 \end{proposition}

We could have seen that $H^*$ is cosemisimple from general principles which ensure that cosemisimplicity is preserved 
under the kind of twisting procedure by which $H^*$ was obtained from $kH_4$.

\subsection{The comodules $V_j$ and the endofunctors   $\bullet\otimes V_j$ of $\Gr(A)$}

For $j =1,2,3$, there is a unique two-dimensional right $D_j$-comodule up to isomorphism. 
We will describe one such $D_1$-comodule (equivalently, $D_1^*$-module) which  we will denote by $V_1$.

Let $\{a^*,b^*,c^*,d^*\}$ be  the basis for $D_1^*\subset H$ that is dual to the basis $\{a,b,c,d\}$ for $D_1$ in 
(\ref{eq.abcd}). They are matrix units for the algebra $D_1^*$ in the sense that the linear map 
\begin{alignat*}{2}
  a^*&\mapsto \begin{pmatrix}1&0\\0&0\end{pmatrix},\quad & b^*&\mapsto \begin{pmatrix}0&1\\0&0\end{pmatrix}\\
  c^*&\mapsto \begin{pmatrix}0&0\\1&0\end{pmatrix},\quad & d^*&\mapsto \begin{pmatrix}0&0\\0&1\end{pmatrix} 
\end{alignat*}
is an algebra isomorphism $D_1^* \to M_2(k)$. 
We define $V_1$ to be the left ideal of $D_1^*$, 
$$
V_1:= k a^*+kc^*.
$$ 

We define $V_2$ and $V_3$ in a similar way.

The functor $\bullet \otimes V_j:\Gr(\wtQ) \to \Gr(\wtQ)$ is the functor $M \rightsquigarrow M \otimes V_j$ where 
$M \otimes V_j$ is made into an $\wtQ$-module through the comodule structure map $\wtQ\to \wtQ\otimes H$.

Our goal, as hinted at before, is to show that if $M$ is a commuting line module, then $M\otimes V_j$ is a direct sum of two 
line modules each of which belongs to a family of line modules parametrized by some other conic $C_j$. 
The next result is the first step in this direction.

\begin{proposition}
\label{pr.Omega_transl}
Let $M \in \Gr(\wtQ)$ and $j \in \{1,2,3\}$.  If $\Omega(z)$ annihilates $M$, then  $\Omega(z+\xi_j)$ annihilates $M\otimes V_j$. 
\end{proposition}
\begin{proof}
By \Cref{cor_H4_action_on_center}, the automorphism $\ve_1 \in H_4$ sends $\Omega(z+\xi_1)$ to $\Omega(z)$, as do $\varepsilon_1\eta$, $\varepsilon_1^{-1}$ and $\varepsilon_1^{-1}\eta$. Thinking of $\wtQ$ and $H^*$ as having the same underlying vector spaces as $S$ and $kH_4$ respectively, the previous sentence applies to them as well. In other words, the coaction $\rho:\wtQ\to \wtQ\otimes H$ sends $\Omega(z+\xi_1)\in \wtQ$ to
\begin{equation}\label{eq.one_term}
  \Omega(z)\otimes (\text{some element in the matrix algebra summand }D_1^*\text{ of }H)+\cdots
\end{equation}
where the dots stand for terms whose right hand tensorands lie in other matrix algebra summands of $H$. Since $H$ acts on $V_1$ via the surjection $H\to D_1^*$, the terms after the $+$ sign in (\ref{eq.one_term}) act as zero on  $M\otimes V_1$. But the left-hand tensorand 
$\Omega(z)$ annihilates $M$ by assumption,  so $\Omega(z+\xi_1)$ annihilates $M\otimes V_1$. 

The same argument works for $M\otimes V_2$ and $M\otimes V_{3}$.
\end{proof}

The following decomposition result describes $M\otimes V_j$ in more detail. Before stating it, recall that
$k\langle \ve_1^2,\ve_2^2,\d\rangle \subset H^*$ is a subcoalgebra with its usual group algebra comultiplication, so the subgroup $\langle \ve_1^2=\xi_1,\ve_2^2=\xi_2,\d\rangle< H_4$ acts on $\wtQ$ by graded algebra automorphisms. 
In particular, $\G = \langle \c_1,\c_2\rangle$ does too so we can define the twist autoequivalences 
$\c_i^*$ on $\Mod(\wtQ)$ or $\Gr(\wtQ)$ as in \cite[\S7.2]{CS15}. We will use the identification $\G=E[2]$ with $\c_j=\xi_j$.

\begin{proposition}
\label{pr.sum_2_copies}
If  $M \in \Gr(\wtQ)$, then $  M\otimes V_1 \cong N\oplus \xi_1^* N$ for some $N\in\Gr(\wtQ)$.
\end{proposition}
\begin{proof}
We will write $ \bZ/2$ for the subgroup $\langle \eta=\d^2 \rangle \subseteq H_4$. Since $\d$ acts as multiplication by $i$ 
on $\wtQ_1$, $\eta$ acts on $\wtQ$ as multiplication by $(-1)^n$ on $\wtQ_{n}$. A coaction by the function algebra $k(\bZ/2)$ is the same thing as a $\bZ/2$-action, so the map
\begin{equation*}
  \begin{tikzpicture}[auto,baseline=(current  bounding  box.center)]
    \path[anchor=base] (0,0) node (1) {$H$} +(2,0) node (2) {$H\otimes H$} +(7,0) node (3) {$k(\bZ/2)\otimes H$};
         \draw[->] (1) to node[auto,pos=.5] {$\Delta$} (2);    
         \draw[->] (2) to node[auto,pos=.5] {$\text{projection}\otimes\id$} (3);
  \end{tikzpicture}  
\end{equation*}
gives an action of $\ZZ/2$ as algebra automorphisms of $H$. Thus, $\wtQ\otimes H$ is $\bZ/2$-module algebra.

We now make $M\otimes V_1$ a $\bZ/2$-equivariant $A \otimes H$-module: we make $M$ a $\bZ/2$-equivariant $\wtQ$-module by having $\eta \in \bZ/2$ act on $M_{n}$ as multiplication by $(-1)^n$, and 
make $V_1$ a $\bZ/2$-equivariant $H$-module by having $\eta$ fix $a^*$ and send $c^*$ to $-c^*$. 

The image of  the map $\wtQ \to \wtQ\otimes H$ is contained in the cotensor product
\begin{equation*}
  \wtQ\square_H H \; \subseteq \; \wtQ\square_{k(\bZ/2)}H \; = \; (\wtQ\otimes H)^{\eta}
\end{equation*}
so the action of $\eta$ on $M\otimes V_1$ commutes with the $\wtQ$-action on it. Hence the $(+1)$- and $(-1)$-eigenspaces of $\eta$ are graded $\wtQ$-submodules of $M \otimes V_1$. 
Finally, the involution $b^*+c^* = { 0 \; 1 \choose 1 \; 0}
\in M_2(k)\cong D_1^*\subset H$ interchanges these eigenspaces and commutes with the action of $\xi_1$ on $\wtQ$; i.e.,
\begin{equation*}
  (b^*+c^*)x = \xi_1(x)(b^*+c^*) \text{ on } M\otimes V_1 \text{ for all } x\in \wtQ. 
\end{equation*}
This shows that the two $\eta$-eigenspaces are, as $\wtQ$-modules, twists of each other by the automorphism $\xi_1$. We denote one of these eigenspaces by $N$; the other is then isomorphic to $\xi_1^*N$.  
\end{proof}

There are analogues of \Cref{pr.sum_2_copies} with $V_2$ and $\xi_2$ (or $V_{3}$ and $\xi_1\xi_2$) 
in place of $V_1$ and $\xi_1$.

\subsection{The autoequivalences $a_\ve$ of $\Gr(A)$}
\label{ssect.E4_acts}
Let $\ve_1 \in H_4$ be the automorphism of $S$ defined in \Cref{prop.aut.S}.  
We define the auto-equivalence $a_{\ve_1}$ of $\Gr(\wtQ)$ by declaring that 
\begin{equation}
\label{eq.defn.a.epsilon}
a_{\ve_1}(M):=(M \otimes V_1)^\eta,
\end{equation}
the $\eta$-invariant subspace, or $(+1)$-eigenspace, for the action of $\eta$ defined in  the proof of \Cref{pr.sum_2_copies}; i.e.,  $\eta$ acts diagonally with $\eta$ acting on $M_n$ as multiplication by $(-1)^n$ and on $V_1$ by fixing $a^*$ and
sending $c^*$ to $-c^*$. Thus, $a_{\ve_1}(M)$ is one of the summands in \Cref{pr.sum_2_copies}.

Since $\eta$ fixes $a^*$ and changes the sign of $c^*$,  the underlying graded vector space of $a_{\ve_1}(M)$ is
\begin{equation}\label{eq.alt1}
  \bigoplus_{\text{even }n}M_n\otimes\langle a^*\rangle \; \oplus \;  \bigoplus_{\text{odd }n}M_n\otimes\langle c^*\rangle. 
\end{equation}

The algebra $\wtQ$ acts on this graded vector space through scalar restriction via the comodule structure map $\wtQ\to \wtQ\otimes H$. As a coalgebra, $H$ is isomorphic to the function algebra $k(H_4)$. The restriction $A_1 \to A_1 \otimes H$ is
the comodule structure map corresponding to the action of $H_4$ as linear automorphisms of $A_1$. Thus, if 
$y \in A_1$, then its image through this structure map is
\begin{equation*}
  \ve_1(y)\otimes b^*+\ve_1^{-1}(y)\otimes c^*+\cdots
\end{equation*}
where $\ve_1(y)$ is the result of the action of $H_4$ on $A_1$ dual to the $H\cong k(H_4)$-coaction (equivalently, it is the action of $H_4$ on $S_1$ transported to $A_1$ through our identification $A_1\cong S_1$) and the dots represent summands that vanish on $a_{\ve_1}(M)$ because their right hand tensorands vanish on $V_1$.

To see that $a_{\ve_1}$  is an equivalence  we note that a quasi-inverse to it is the functor 
defined by the formula (\ref{eq.defn.a.epsilon})  with $V_1$ having its ``other'' $\bZ/2$-equivariant structure, i.e.,
the action of $\eta$ that fixes $c^*$ and sends $a^*$ to $-a^*$.

By \Cref{pr.sum_2_copies}, $a_{\ve_1}$ preserves dimensions of homogeneous components.

In \Cref{sect.ann.pt.mods} we saw that each ordinary point module associated to $\fP_j$ is annihilated by 
$\Omega(\tau+\xi_j)$. This allows us to define the following relations on the four-element set of ordinary families.

\begin{definition}\label{def.rel}
Let $\xi\in E[2]$. Two ordinary families $\fP_i$ and $\fP_j$ of point modules are {\sf $\xi$-related} if they are annihilated by $\Omega(z)$ and $\Omega(z+\xi)$ respectively for some $z \in E$. 
\end{definition}

By \Cref{prop.H4.moves.points.for.A}, the action of $H_4$ on $\PP(A_1^*)$ is such that $\ve_i$ sends $\fP_j$ to $\fP_k$
where $k$ is determined by the requirement that $2\ve_1+\xi_j=\xi_k$.

As a consequence of the proof of \Cref{pr.sum_2_copies} we get

\begin{corollary}\label{cor.autoequivs}
For each 2-torsion point $ \xi\in E$ there are four order-4 autoequivalences $a_\ve$ of 
$\Gr(\wtQ)$ that preserve Hilbert series, permute the special point modules, and interchange the point modules in any two $\xi$-related ordinary families. 
\end{corollary}
\begin{proof}
Let $\xi=\xi_1$ and $\ve=\ve_1$. 

Since $a_{\ve}$ sends $\wtQ$ to a module isomorphic to $\wtQ$, $a_\ve$ sends cyclic modules to cyclic modules.  By \Cref{pr.Omega_transl}, $a_\ve$ sends the point modules in an ordinary family to the point modules in the 
$\xi$-related family and permutes the four special point modules since each of those is annihilated by exactly one central element $\Omega(\omega)$, $\omega\in E[2]$.

The four auto-equivalences in $\{\gamma^*a_\ve \; | \; \gamma \in \Gamma\}$ have the same permutation properties because each $\fP_j$ is stable under the action of $\G$. 

The argument for other $\xi$ is analogous, substituting $V_2$ or $V_{3}$ for $V_1$.    
\end{proof}

\begin{corollary}
\label{cor.autoequivs2}
The autoequivalence $a_\ve$ sends cyclic modules to cyclic modules and, since  it preserves Hilbert series,  it
 induces  automorphisms of the point and line schemes for $A$.
\end{corollary}
\begin{proof}
We noticed in the proof of \Cref{cor.autoequivs} that  $a_\ve$ sends cyclic modules to cyclic modules.
Since  the point and line schemes of $A$  classify isomorphism classes of cyclic modules with certain Hilbert series (e.g., \cite{ShV02}), $a_\ve$ induces  an automorphism of these schemes. 
\end{proof}

We can say more.

\begin{proposition}\label{pr.E4_acts}
The autoequivalences $a_\ve$ generate an action of $(\bZ/4)^2\cong E[4]$ on $\Gr(\wtQ)$.
\end{proposition}
\begin{proof}
By construction,  the underlying graded vector space of $a_{\ve_1}(M)$ is
\begin{equation}\label{eq.alt}
  \bigoplus_{\text{even }n}M_n\otimes\langle a^*\rangle \; \oplus \;  \bigoplus_{\text{odd }n}M_n\otimes\langle c^*\rangle
\end{equation}
and $y\in\wtQ_1$ acts on it as $\ve_1(y)\otimes b^*+\ve_1^{-1}(y)\otimes c^*+\cdots$. 
Since
$$
\begin{pmatrix} b^* \\ c^*  \end{pmatrix}    \begin{pmatrix} a^* & c^*  \end{pmatrix} \; = \;  \begin{pmatrix} 0 & \phantom{-}a^* \\ c^* & 0 \end{pmatrix},
$$ 
the action of $y$ on the left hand tensorands of the even components in \Cref{eq.alt} is the initial action of $\varepsilon_1^{-1}\triangleright y$. Similarly, its action on the left hand tensorands of the odd components is the initial action of $\varepsilon_1\triangleright y$. In conclusion, identifying \Cref{eq.alt} with $M$ by 
\begin{equation*}
\langle a^*\rangle\cong k\cong \langle c^*\rangle,\quad a^*=c^*=1,
\end{equation*}
$y \in A_1$ acts as $\ve_1^{-1}(y)$ on the even-degree components of $a_{\ve_1}(M)$ and as $\ve_1(y)$ on the odd-degree components.

The same discussion applies to the other autoequivalences defined above and shows that they do indeed comprise an action of $(\bZ/4)^2$ on $\Gr(\wtQ)$, with $a_{\ve_1}$ and $a_{\ve_2}$ as generators. The fact that $a_{\ve_1}$ and $a_{\ve_2}$ commute amounts to their commutator scaling the homogeneous components of a graded module as in \Cref{old.lemma}(1), and hence acting trivially. 
\end{proof}

Now, as before, the action of $E[4]$ on $\Gr(\wtQ)$ induces one on point and line schemes. In addition, since $\GG(1,3)$ parametrizes the isomorphism classes of cyclic graded $\wtQ$-modules with Hilbert series $1+2t$, the $E[4]$-action on the line scheme extends to one on $\GG(1,3)$.  

\begin{proposition}\label{pr.theyre_conics}
The images of the conic $C_0\subset \PP^5$ under the action of $E[4]$ on $\GG(1,3)$ are conics.    
As subvarieties of $\GG(1,3)$ the four different images, $C_0$, $C_1$, $C_2$, $C_3$, are cut out by the Pl\"ucker relation $ z_{01}z_{23} - z_{02}z_{13}+z_{03}z_{12} =0$ together with three linear relations which are:
 \begin{align*}
 C_0: \phantom{xxxx}   &  z_{23} +\a z_{01} \; = z_{31}+\b z_{02}  \,\; = \; z_{12}+\c z_{03}\; \, = \; 0;       \\
 C_1: \phantom{xxxx}   & z_{23} -\a  z_{01}  \;  = \;  z_{31}-  z_{02}   \phantom{\b}   \; = \; z_{12} +  z_{03}  \phantom{\c} \; =\;  0;       \\
 C_2: \phantom{xxxx}   & z_{23} +  z_{01}   \phantom{\a}  \;  =\; z_{31} - \b z_{02}  \;  = \; z_{12} -  z_{03}  \phantom{\c}  \; = \; 0;       \\
 C_3: \phantom{xxxx}   & z_{23} -  z_{01}   \phantom{\a}     \; = \; z_{31} +  z_{02}   \phantom{\b}  \; =\;  z_{12} - \c z_{03}  \;  = \; 0.
  \end{align*}
These conics are pairwise disjoint. 
\end{proposition}
\begin{proof}
By the remark before the proposition the action of $E[4]$ on $\GG(1,3)$ preserves the line scheme. 

The Pl\"ucker embedding $\GG(1,3) \to \PP^5$ is effected via the top exterior power $\cL$ on $\GG(1,3)$ of the dual tautological rank-two bundle on $\GG(1,3)$. Since the Picard group of $\GG(1,3)$ is $\bZ$ and $[\cL]$ generates it \cite[Theorem 2.10]{EH13}, every automorphism of $\GG(1,3)$ leaves $[\cL]$ invariant. 

In particular, $[\cL]$ is invariant under the action of $E[4]$ on the Picard group of $\GG(1,3)$ so it admits an equivariant structure for the action of the universal central extension $H_4$ of $E[4]\cong (\bZ/4)^2$ on $\GG(1,3)$. The equivariant structure allows us to lift the $H_4$-action on $\GG(1,3)$ to an action on the ambient space $\PP^5$ of the Pl\"ucker embedding. Being a conic is a statement about how $C_0$ sits inside $\PP^5$, so an element of $H_4$ sends $C_0$ to 
another conic. 

The $H_4$ action on $\wtQ_1$ is obtained by transferring the $H_4$ action on $S_1$ to $\wtQ_1$ by using  the linear isomorphism 
$\iota:\wtQ_1 \to S_1$
$$
\iota(y_0)= x_0, \qquad \iota(y_1)= -ix_1, \qquad \iota(y_2)= -ix_2, \qquad \iota(y_3)= x_3
$$
in \Cref{conv.identification}. 
Let $\phi_j$, $j=1,2,3$,  be the automorphisms of $S$ in \Cref{prop.aut.S} and define $\psi_j:=\iota^{-1}\phi_j\iota:\wtQ_1 \to \wtQ_1$.  The action  of $\psi_j$ on $\wtQ_1$ is given by the following table: 
\begin{table}[htp]
\begin{center}
\begin{tabular}{|c|c|c|c|c|c|c|c|}
\hline
 & $y_0$ &  $y_1$ &  $y_2$ &  $y_3$  
\\
\hline
$\psi_1 $  & $ibc y_1$ &   $-y_0$ &  $-by_3$ &  $-ic y_2$ $\phantom{\Big)}$
\\
\hline
$\psi_2$  & $iac y_2$ &   $ia y_3$ &  $-y_0$ &  $cy_1$ $\phantom{\Big)}$
\\
\hline
$\psi_3$  & $ab y_3$ &   $ -ia y_2$ &  $- b y_1$ &  $-iy_0$ $\phantom{\Big)}$
\\
\hline
\end{tabular}
\end{center} 
\end{table}
\newline

A linear automorphism  $\psi:\wtQ_1 \to \wtQ_1$ induces a linear automorphism of $\PP(\wtQ_1^*)$ that
sends the plane $y=0$  to the plane $\psi(y)=0$, and so on.  For example, $\psi_3$ sends the plane 
$\ve y_0+\l y_1+\mu y_2+\nu y_3=0$ to the plane $\ve ab y_3 -\l ia y_2-\mu b y_1 -i \nu y_0=0$. 

Let $L$ be a commuting line defined by $\ve y_0+\l y_1+\mu y_2+\nu y_3=\ve' y_0+\l' y_1+\mu' y_2+\nu' y_3=0$.
The Pl\"ucker coordinates of  $L$  are given by the $2 \times 2$ minors of
$$
\begin{pmatrix}
\ve & \l & \mu & \nu 
   \\
\ve' & \l' & \mu' & \nu'
\end{pmatrix}.
$$
Since $\psi_3(L)$ is given by the equations $\psi_3(\ve y_0+\l y_1+\mu y_2+\nu y_3)=\psi_3(\ve' y_0+\l' y_1+\mu' y_2+\nu' y_3)=0$,  
the Pl\"ucker coordinates of $\psi_3(L)$  are given by the $2 \times 2$ minors of
$$
\begin{pmatrix}
-i\nu  & -b \mu  & -ia \l  & ab \ve 
   \\
-i\nu'  & -b \mu'  & -ia \l ' & ab \ve'
\end{pmatrix}.
$$
Therefore 
\begin{align*}
(z_{23}-z_{01})(\psi_3(L)) & \; = \;  -i\a b(\l\ve'-\l'\ve) -ib(\nu\mu'-\nu'\mu)  \; = \;   ib(\a z_{01}+z_{23})(L)=0,
\\
(z_{13}-z_{02})(\psi_3(L)) & \; = \;  -a\b (\mu\ve'-\mu'\ve) +a(\nu\l'-\nu'\l)  \; = \;   a(\b z_{02} - z_{13})(L)=0,
\\
(z_{12}-\c z_{03})(\psi_3(L)) & \; = \;  ia b(\mu\l'-\mu'\l) +iab\c (\nu\ve'-\nu'\ve)  \; = \;   iab(-z_{12}-\c z_{03})(L)=0. 
\end{align*}
In conclusion, the point on $\GG(1,3)$ corresponding to the line $\psi_3(L)=0$ lies on the plane $z_{23}-z_{01}=z_{13}-z_{02}=z_{12}-\c z_{03}=0$.

Similarly, the Pl\"ucker coordinates of $\psi_2(L)$  are given by the minors of
$$
\begin{pmatrix}
-\mu  & c \nu  & iac\ve  & ia\l 
   \\
-\mu'  & c \nu'  & iac\ve'  & ia\l' 
\end{pmatrix}.
$$
Therefore 
\begin{align*}
(z_{23}+z_{01})(\psi_2(L)) & \; = \;  -\a c(\ve\l'  -\l\ve') -c(\mu\nu'-\mu'\nu)  \; = \;   c(-\a z_{01}- z_{23})(L)=0,
\\
(z_{13}+\b z_{02})(\psi_2(L)) & \; = \;  iac (\nu\l'-\nu'\l) -iac \b (\mu\ve'-\mu'\ve)  \; = \;   iac (-z_{13} + \b z_{02})(L)=0,
\\
(z_{12} - z_{03})(\psi_2(L)) & \; = \;  ia \c(\nu\ve'-\nu'\ve) +ia(\mu\l'-\mu'\l)  \; = \;   ia(-\c z_{03}- z_{12})(L)=0. 
\end{align*}
Thus, the point on $\GG(1,3)$ corresponding to $\psi_1(L)=0$ lies on the plane $z_{23}+z_{01}=z_{13}+\b z_{02}=z_{12}-z_{03}=0$.

The Pl\"ucker coordinates of $\psi_1(L)$  are given by the minors of
$$
\begin{pmatrix}
-\l  & ibc\ve  &-ic\nu  & -b\mu 
   \\
-\l'  & ibc\ve'  &-ic\nu'  & -b\mu'
\end{pmatrix}.
$$
Therefore 
\begin{align*}
(z_{23}-\a z_{01})(\psi_1(L)) & \; = \;  ibc(\nu\mu'  -\nu'\mu) +ibc\a(\l\ve'-\l'\ve)  \; = \;   ibc(-z_{23}- \a z_{01})(L)=0,
\\
(z_{13}+ z_{02})(\psi_1(L)) & \; = \;  -i\b c (\ve\mu' -\ve'\mu) + ic (\l\nu'-\l'\nu)  \; = \;   i c (-\b z_{02} + z_{13})(L)=0,
\\
(z_{12} + z_{03})(\psi_1(L)) & \; = \;  b \c(\ve\nu'-\ve'\nu) +b(\l\mu'-\l'\mu)  \; = \;   b(\c z_{03} + z_{12})(L)=0. 
\end{align*}
Thus, the point on $\GG(1,3)$ corresponding to $\psi_2(L)=0$ lies on the plane $z_{23}-\a z_{01}=z_{13}+ z_{02}=z_{12}+z_{03}=0$.

To see that the conics are pairwise disjoint consider, first, $C_0 \cap C_1$. Since $z_{23}+\a z_{01}$ and $z_{23}-\a z_{01}$ vanish
on $C_0 \cap C_1$, so do $z_{01}$ and $z_{23}$. Since $z_{13}-\b z_{02}$ and $z_{13}+z_{02}$ vanish on $C_0 \cap C_1$ so do 
$z_{13}$ and $z_{02}$. Since $z_{12}+\c z_{03}$ and $z_{12}+z_{03}$ vanish on $C_0 \cap C_1$ so do $z_{12}$ and $z_{03}$. 
Thus all the Pl\"ucker coordinates vanish on $C_0 \cap C_1$. Hence $C_0 \cap C_1 = \varnothing$. 

The other cases are similar.
\end{proof}

In the previous proof, the $\psi_j$'s act on the Pl\"ucker coordinates as 
\begin{table}[htp]
\begin{center}
\begin{tabular}{|c|c|c|c|c|c|c|c|}
\hline
$\phantom{\Big)}$ & $z_{01}$ &  $z_{02}$ &  $z_{03}$ &  $z_{12}$ &  $z_{13}$ &  $z_{23}$
\\
\hline
$\phantom{\Big)} \psi_1$  & $ibcz_{01}$ &  $i cz_{13}$ &  $b  z_{12}$ &  $b\c z_{03}$ &  $-i\b c z_{02}$ &  $-ibcz_{23}$
\\
\hline
$\phantom{\Big)} \psi_2$  & $- c z_{23}$ &  $iac z_{02}$ &  $ia z_{12}$ &  $-ia\c z_{03}$ &  $-iac z_{13}$ &  $-\a c z_{01}$
\\
\hline
$\phantom{\Big)} \psi_3$  & $-i b z_{23}$ &  $ a z_{13}$ &  $iab z_{03}$ &  $-iab z_{12}$ &  $a\b z_{02}$ &  $i\a b z_{01}$
\\
\hline
\end{tabular}
\end{center}
\end{table}

\section{Line modules for $\wtQ$}

\subsection{Classification of line modules}
Let $\LL \subset \GG(1,3) \subseteq \PP(\wedge^2 \wtQ_1^*)\cong \PP^5$ denote the line scheme for $\wtQ$.

\begin{theorem}
\label{thm.main1}
Let  $\xi_1,\xi_2,\xi_3$ be the 2-torsion points of $E$.
The line scheme for $\wtQ$ is a reduced and irreducible curve of degree 20. It is  the union of 3 disjoint quartic 
elliptic curves and 4 disjoint plane conics, 
$$
\LL \; = \; \Big((E/\langle \xi_1 \rangle) \; \sqcup \; (E/\langle \xi_2 \rangle)  \; \sqcup \; (E/\langle \xi_3 \rangle) \Big)  \; \bigcup \; 
\Big(C_0  \; \sqcup \; C_1 \; \sqcup \; C_2  \; \sqcup \;   C_3\Big)
$$
having the property that $\big\vert (E/\langle \xi_i \rangle) \cap C_j \big\vert =2$ for all $(i,j) \in \{1,2,3\} \times \{0,1,2,3\}$.
\end{theorem}
\begin{proof}
We consider $\GG(1,3)$ as a subvariety of $\PP^5$ via the Pl\"ucker embeddding.
 By  \cite[Cor. 2.6]{ShV02}, every component of $\LL$ has dimension $\ge 1$.
By \Cref{thm.D}, $\dim(\LL)\le 1$ so every component of $\LL$ has dimension $=1$.
Hence, by \cite{CSV15}, $\deg(\LL)=20$. 

By \cite[Prop. 11.12]{CS15}, the image of each $E/\langle \xi \rangle$ in $\GG(1,3)$ has degree 4 as a curve in $\PP^5$.
Each of the conics $C_0,C_1,C_2,C_3$ has degree two. Thus, the union of these 7 components of $\LL$ has degree 
$3 \times 4 + 4 \times 2 =20$. It follows that $\LL$ is as claimed. 

The claim about the intersection points of the components of $\LL$ is proved in \Cref{cor.ell_line=conic_line}. It is also
a consequence of the calculations in the Appendix.
\end{proof}

\begin{proposition}
\label{prop.Ej}
For $j=1,2,3$, the image of $E/\langle \xi_j \rangle$ in $\GG(1,3)$ is the curve $E_j$ described in the Appendix. 
Defining equations for $E_j$ are given at the end of the Appendix and the points in $C_i \cap E_j$ are given in Table
\ref{E.cap.C}.
\end{proposition}
\begin{proof}
We just do the case $j=1$. The other cases are similar. By \cite[Lem. 11.7]{CS15}, a secant line of the form $\overline{p,p+\xi_1}$ is cut out by equations of the form $\b_0 y_0+\b_1y_1=\b_2y_2+\b_3y_3=0$. It is easy to see that the point in 
$\GG(1,3)$ corresponding to this line lies on the 3-plane $z_{01}=z_{23}=0$. But the only component of $\LL$ 
contained in that 3-plane  is the curve $E_1$ in the Appendix.
\end{proof}

  \subsection{}
\label{sect.E4.action}

We now give a more detailed description of the action of $E[4]$ on the elliptic and conic families of lines induced by the action on $\Gr(\wtQ)$ discussed above.

\begin{proposition}\label{pr.E4_act_lines}
The group $E[4]\cong (\bZ/4)^2$ acts on the three elliptic curves and the four conics that parametrize the line modules as follows: 
\begin{enumerate}
\renewcommand{\theenumi}{\arabic{enumi}}
  \item $E[4]$ fixes each elliptic family individually and acts on the family $E/\langle \xi\rangle$, $\xi\in E[2]$, as translations by the image of $E[4]$ through $E\to E/\langle \xi\rangle$.  
  \item The subgroup $\G\cong E[2]\subset E[4]$ fixes each of the conics $\{C_i\ |\ 0\le i\le 3\}$ individually and acts on the commuting conic $C_0$ by twisting by the algebra automorphisms 
    \begin{equation}\label{eq.E4_act_lines}
      y_0\mapsto y_0,\ y_i\mapsto y_i,\ y_j\mapsto -y_j,\ y_k\mapsto -y_k
    \end{equation}
    for the three cyclic permutations $(i,j,k)$ of $(1,2,3)$. 
\end{enumerate}
The quotient group $E[4]/E[2]\cong (\bZ/2)^2$ acts on the set $\{C_i \; | \; 0\le i\le 3\}$ as the regular permutation representation. 
\end{proposition}
\begin{proof}
{(1)} 
Fix some $\xi\in E[2]$. 
Let $g:\bP(S_1^*)\to \bP(\wtQ_1^*)$ be the identification made in \Cref{conv.identification}.
By \cite[Proposition 10.10]{CS15},  the closed immersion of 
$E/\langle \xi\rangle$ into the line scheme for $\wtQ$ is given by
\begin{equation}\label{eq.E4_act_elliptic_lines} 
  p+\langle \xi\rangle \; \mapsto \; g(\text{the line in }\bP(S_1^*)\text{ passing through }p\text{ and }p+\xi).
\end{equation}

Translation by the image of $\varepsilon_1\in E[4]$ in $E/\langle \xi\rangle$ moves the line through $p$ and $p+\xi$ to that  through $p+\varepsilon_1$ and $p+\varepsilon_1+\xi$. Let $U\subseteq S_1$ be the two-dimensional subspace annihilating a generator of the line module $M_{p,p+\xi}$, and hence the points $p$ and $p+\xi$.
Applying the automorphism $\varepsilon_1\in H_4 \subseteq \Aut(S)$ to $U$, $\varepsilon_1 U$ annihilates
a generator of the new line module.

The action of $H_4$ on $\wtQ_1$ is defined so that the identification in \Cref{conv.identification},
and hence the map $g$, is $H_4$-equivariant. Hence, transporting the $\varepsilon_1$-translation to the right hand side of \Cref{eq.E4_act_elliptic_lines}, we see that the new line 
module obtained in this fashion is again annihilated by $\varepsilon_1 U$. This, however, matches the description of the action of the autoequivalence $a_{\ve_1}$ on graded $\wtQ$-modules given in \Cref{ssect.E4_acts}: 

namely, for any such module $M$, $a_{\ve_1}M$ coincides with $M$ as a graded vector space and the action of the degree-one elements of $\wtQ$ on even-degree elements of $M$ is twisted by $\varepsilon_1^{-1}$. Hence, if $U\subset \wtQ_1$ annihilates the degree-zero component of $M$, then $\varepsilon_1U$ will annihilate the degree-zero component of $a_{\ve_1}M$.  

{(2)} The fact that two generators of $E[4]\cong (\bZ/4)^2$ ($a_{\ve_1}$ and $a_{\ve_2}$ say) act on the four-element set $\{C_i\}$ as products of two involutions follows from \Cref{cor.autoequivs}. In particular, the $2$-torsion $E[2]$ of $E[4]$ preserves each conic $C_i$. 

On the other hand, by \Cref{lem.subgp.G}, the group algebra of the subgroup
$\langle \ve_1^2,\ve_2^2,\d\rangle \subseteq H_4$ retains its coalgebra structure after the comultiplication of $kH_4$ has been deformed into that of $H^*$. Therefore $\langle \ve_1^2,\ve_2^2,\d\rangle$,
and in particular $\G$, acts on $\wtQ$ by graded algebra automorphisms.
Since the identification of $S_1$ and $A_1$ is of the form 
$x_i\ \leftrightarrow$ a scalar multiple of $y_i$, this $\G$-action is precisely \Cref{eq.E4_act_lines}. 
\end{proof}

\subsection{Elliptic line modules}

The following is an easy consequence of \Cref{prop.line.killers}.

\begin{proposition}
\label{cor.LS}
Let $\xi$, $\xi'$, and $\xi+\xi'$, be the 2-torsion points in $E$. 
Let $p,z \in E$. Define $x=p+E[2]$. The following statements are equivalent:
\begin{enumerate}
  \item 
  \label{p1}
  $\Omega(z)$ annihilates $M_{p,p+\xi}$;
  \item 
  \label{p2}
  $\Omega(z)$ annihilates $M_{p+\xi',p+\xi'+\xi}$;
  \item 
   \label{p3}
  $\Omega(z)$ annihilates $M_{x,\xi}$, the module defined in (\ref{defn.M.x.xi});
  \item{}
   \label{p4}
  $2p+\xi \in \{z,-z-2\tau\}$;
  \item
   \label{p5}
  $z \in \{2p+\xi,-2p-2\tau+\xi\}$.
\end{enumerate}
\end{proposition}

\begin{proposition}
\label{prop.ann'or.elliptic.lines}
Let $z \in E$. Then $\Omega(z)$ annihilates 
\begin{enumerate}
  \item 
  exactly  six modules of the form $M_{x,\xi}$ if $z+\tau \notin E[2]$, and 
  \item 
 exactly three modules of the form $M_{x,\xi}$ if  $z+\tau \in E[2]$. 
\end{enumerate}
\end{proposition}
\begin{proof}
Let $A(z)=\{(x,\xi)\; | \; \Omega(z) M_{x,\xi}=0\}$ and $S(z)=\{(p+E[2] ,\xi) \; | \;  2p+\xi =z\}$. 
By Corollary \ref{cor.LS}, $A(z)=S(z) \cup S(-z-2\tau)$.
 
For each of the three 2-torsion points $\xi$  there is a unique coset $p+E[2]$ such that $2p+\xi=z$.  Hence $|S(z)|=3$. 
If $z+\tau \notin E[2]$, then $z \ne -z-2\tau$ so $S(z) \cap S(-z-2\tau)=\varnothing$, whence $|A(z)|=6$. 
If $z+\tau \in E[2]$, then $z=-z-2\tau$ so $S(z)=S(-z-2\tau)$, whence $|A(z)|=3$. 
\end{proof}

\begin{proposition}
Let $z \in E$. Then $\Omega(z)$ annihilates 
\begin{enumerate}
  \item 
  exactly four elliptic line modules in each elliptic family if $z+\tau \notin E[2]$, and 
  \item 
 exactly two elliptic line modules in each elliptic family if $z+\tau \in E[2]$. 
\end{enumerate}
Let $\xi$ and $\xi'$ be different 2-torsion points.  If $p$ is one of the four points on $E$ such that $2p=z+\xi$, then the line modules 
parametrized by $E/\langle \xi \rangle$ that are annihilated by $\Omega(z)$ correspond to the points on 
$E/\langle \xi \rangle$ that are the images of $p$, $p+\xi'$, $-p-\tau$, and $-p-\tau+\xi'$.  
\end{proposition}

 \subsection{}
 \label{sect.F.tau+xi}
By \Cref{prop.L->F},  
$\Hom_{\Gr(S)}(M_{p,q},\wtV(\tau+\xi))\ne 0$ for all $p+q=\tau+\xi$. 

Let $F_{\tau+\xi}=\wtV(\tau+\xi) \otimes k^2 \in \Gr(S')$.

\begin{lemma}
\label{lem.line.fat}
Let $p \in E$ and write $x=p+E[2]$. The following conditions are equivalent:
\begin{enumerate}
  \item 
  $\Hom_{\Gr(S)}(M_{p,p+\xi},\wtV({\tau+\xi})) \ne 0$;
  \item 
    $\Hom_{\Gr(S)}(M_{p+\xi',p+\xi'+\xi},\wtV({\tau+\xi})) \ne 0$;
  \item 
    $\Hom_{\Gr(S')}(M_{x,\xi},F_{\tau+\xi}) \ne 0$.
\end{enumerate}
\end{lemma}
\begin{proof}
Because $p+p+\xi=(p+\xi')+(p+\xi'+\xi)$ this follows from \Cref{prop.L->F} and the Morita equivalence between $S$ and $S'$. 
\end{proof}

\begin{proposition}
\label{prop.line.fat}
Fix $\xi \in E[2]$. For each 2-torsion point $\omega$, there is a unique $x \in E/E[2]$ such that   
$\Hom_{\Gr(S')}(M_{x,\omega}, F_{\tau+\xi}) \ne 0$, namely $x=  p +E[2]$ where $2p=\tau+\xi+\omega$.
\end{proposition}
\begin{proof} 
The points $p \in E$ such that $2p=\tau+\xi+\omega$ form a single $E[2]$-coset;
hence $x$ does not depend on the choice of $p$.  
Since $2p+\omega=\tau+\xi$, $\Hom_{\Gr(S)}(M_{p,p+\omega},\wtV(\tau+\xi)) \ne 0$ by \Cref{prop.L->F}. 
 Hence $\Hom_{\Gr(S')}(M_{x,\omega}, F_{\tau+\xi}) \ne 0$ by \Cref{lem.line.fat}. 

Let  $y  = q+E[2]$. If $\Hom_{\Gr(S')}(M_{y,\omega}, F_{\tau+\xi}) \ne 0$, then  
$\Hom_{\Gr(S)}(M_{q,q+\omega}, \wtV({\tau+\xi}) )\ne 0$ by \Cref{lem.line.fat} so 
$2q+\omega=\tau+\xi=2p+\omega$. Hence $p-q \in E[2]$ and $x=y$. 
\end{proof}

Let  $p \in \fP_\infty$ and let $M_{p}\in \Gr(A)$ be the  corresponding point module. Since three of the coordinate functions $y_0,y_1,y_2,y_3$ vanish at $p$, $M_{p}$ is annihilated by $\Omega(\xi)$ for some $\xi \in E[2]$. Since $\tau\notin E[2]$, there are, up to isomorphism, exactly twelve line modules  in $\Gr(A)$ that are annihilated by $\Omega(\xi)$. 

\section{Points on conic lines}
 
\subsection{}
This section answers the following question: if $p \in \fP$, i.e., $M_p=A/Ap^\perp$ is a point module in $\Gr(A)$, and $L=A/A\ell^\perp$ is a {\it conic} line module, when is there an epimorphism $\pi^*L \to \pi^*M_p$ in $\QGr(A)$. 
Since $M_p$ is 1-critical, this is  equivalent to the question ``if $p \in \fP$ and $\ell$ is a conic line in $\PP(A_1^*)$ when is there a non-zero homomorphism $A/A\ell^\perp \to A/Ap^\perp$?'' Clearly, there is a non-zero homomorphism $A/A\ell^\perp \to A/Ap^\perp$
if and only if $p \in \ell$ so the original question about morphisms in $\Projnc(A)$ is equivalent to the geometric 
question ``if $p \in \fP$ which conic lines does $p$ lie on?''

We use the explicit descriptions of $\fP$ and the conic line modules to answer the question. 

\begin{theorem}
\label{cor.conics}
The lines in $\Projnc(\wtQ)$, which are also lines in $\PP(\wtQ_1^*)$, parametrized by the conics $C_0,\ldots,C_3,$ are such that
\begin{enumerate}
\item 
every point in $\fP-(\fP_\infty \cup \fP_j)$ lies on exactly one line from $C_j$ and
\item 
no point in $\fP_\infty\cup \fP_j$ lies on any lines from $C_j$. 
\end{enumerate}
If $0 \le i \ne j \le 3$, the lines in $C_i$ that pass through a point in $\fP_j$ form a single $\G$-orbit in the sense of \Cref{pr.E4_act_lines}. 
\end{theorem}
\begin{proof}
Since $C_j$ is the image of $C_0$ under the autoequivalences $a_{\ve_j}$, both (1) and (2) follow from \Cref{thm.pts.on.comm.lines,pr.theyre_conics} together with the fact that the $E[4]$-action on the point scheme permutes the ordinary points, and therefore permutes the special points too. The last sentence follows from the fact that the points in $\fP_j$ form a single $\G$-orbit. 
\end{proof}

The obvious way to get a family of lines in $\PP^3$ parametrized by a smooth conic is to take the lines in a ruling on a smooth quadric. We now determine quadrics $Q_j$ such that the lines parametrized by $C_j$ belong to a ruling on  $Q_j$.
The following result is a standard exercise. 

\begin{proposition}
\label{prop.quadric.ruling}
Let $\l,\mu,\nu \in k^\times$ and $(s,t) \in \PP^1$. Let $Q$ be the quadric $y_0^2-\l^2 y_1^2 +\mu^2 y_2^2 -\nu^2 y_3^2=0$. The line
\begin{equation}
\label{eq.line.on.quadric}
s(y_0-\l y_1) + t(\mu y_2- \nu y_3) \; = \; t(y_0+\l y_1) - s(\mu y_2+\nu y_3) \; = \; 0
\end{equation}
lies on $Q$ and the corresponding point on $\GG(1,3)$ is 
\begin{equation}
\label{eq.P.coords}
(2\l st, \, -\mu(s^2+t^2), \, \nu(t^2-s^2), \, \l\mu(s^2-t^2), \, \l\nu(s^2+t^2), \, -2\mu\nu st)
\end{equation}
which lies on the intersection of $\GG(1,3)$ with the plane
\begin{equation}
\label{eq.quadric.conic.lines}
z_{23}+\l^{-1}\mu\nu z_{01} \; = \; z_{13}+ \mu^{-1}\l\nu z_{02}   \; = \;  z_{12} +\nu^{-1}\l\mu z_{03} \; = \; 0.
\end{equation}
The lines in the other ruling on $Q$ correspond to the points on the  intersection of $\GG(1,3)$ with the plane
\begin{equation}
\label{eq.quadric.other.conic.lines}
z_{23}-\l^{-1}\mu\nu z_{01} \; = \; z_{13}- \mu^{-1}\l\nu z_{02}   \; = \; z_{12} -\nu^{-1}\l\mu z_{03} \; = \; 0.
\end{equation}
\end{proposition}
\begin{proof}
If we write the equation defining $Q$ as $(y_0+\l y_1)(y_0-\l y_1) = (\nu y_3+\mu y_2)(\nu y_3-\mu y_2)$ it is easy to see that the line (\ref{eq.line.on.quadric})
lies on $Q$. The Pl\"ucker coordinates of the corresponding point on $\GG(1,3)$ are given by the minors of the matrix
$$
\begin{pmatrix}
   s & - s \l & t\mu    & - t\nu   \\
    t & t\l & -s\mu & -s\nu
\end{pmatrix}.
$$
Those minors are the entries in (\ref{eq.P.coords}). The point in (\ref{eq.P.coords}) obviously lies on the plane in  (\ref{eq.quadric.conic.lines}). 

Replacing $(\l,\mu,\nu)$ by $(-\l,-\mu,-\nu)$ gives the other ruling on $Q$ and
(\ref{eq.quadric.conic.lines}) changes to (\ref{eq.quadric.other.conic.lines}).
\end{proof}

\begin{corollary}
\label{cor.S123}
The conics $C_0,\ldots,C_3$, are the intersections of $\GG(1,3)$ with the planes in $\PP(\wedge^2\wtQ_1^*)$ 
in the second column of \Cref{table.conics.quadrics}, and the lines in $\PP(\wtQ_1^*)$ that correspond to the points on 
$C_j$ provide a ruling on the quadric $Q_j$ in the third column of the following table:
\begin{table}[htp]
\begin{center}
\begin{tabular}{|c|c|c|}
\hline    
& $C_j=\GG(1,3) \cap \text{\rm (the plane below)}$                 &     $\phantom{\Big\vert}\text{\rm The quadric $Q_j$}$   
\\
\hline
\hline
$\phantom{\Big(} C_0$ & $z_{23}+ \a  z_{01}  =  z_{13} -\b  z_{02}    =  z_{12} + \c  z_{03}  =  0$ & $y_0^2 +\b\c y_1^2+\c \a y_2^2 +\a\b y_3^2=0$  
\\ 
\hline
$\phantom{\Big(} C_1$ & $z_{23}-\a  z_{01}  =  z_{13}+  z_{02}    =  z_{12} +  z_{03}  =  0$ & $y_0^2 -y_1^2-\a y_2^2 +\a y_3^2=0$    
\\ 
\hline
$\phantom{\Big(} C_2$ & $ z_{23} +  z_{01}  = z_{13} + \b z_{02}    =  z_{12} -  z_{03}  =  0$ &  $y_0^2 +\b y_1^2- y_2^2 -\b y_3^2=0$  
\\
\hline
$\phantom{\Big(} C_3$ & $  z_{23} -  z_{01}  = z_{13} -  z_{02}    = z_{12} - \c z_{03}  = 0$ &  $y_0^2 - \c y_1^2 + \c  y_2^2 - y_3^2=0$  
  \\
\hline
\end{tabular}
\end{center}
\vskip .12in
\caption{The quadric ruled by the lines in $C_j$.}
\label{table.conics.quadrics}
\end{table}
\newline
Further, $\fP-(\fP_\infty \cup \fP_j) \subseteq Q_j$ and each point in $ \fP - (\fP_\infty \cup \fP_j)$ lies on a unique line belonging to $C_j$.
\end{corollary}
\begin{proof}
Apply \Cref{prop.quadric.ruling} with $(\l,\mu,\nu)$ equal to
$(-ibc,ac,iab)$, $(1,ia,ia)$, $(ib,i,b)$, and $(c,c,-1)$.

 One sees that $\fP-(\fP_\infty \, \cup \, \fP_j) \subseteq Q_j$ by evaluating the equation for $Q_j$ at the points in $\fP-(\fP_\infty \cup \fP_j)$. The calculation
 is simplified by  noticing that the values of $(y_0^2,y_1^2,y_2^2,y_3^2)$ on $\fP_0$, $\fP_1$, $\fP_2$, and $\fP_3$, are 
 $(1,1,1,1)$, $(\b\c,-1,-\b,\c)$, $(\a\c,\a,-1,-\c)$, and $(\a\b,-\a,\b,-1)$, respectively. Finally, since $Q_j$ is the disjoint union of the lines in the 
 ruling on it parametrized by $C_j$, each point in $ \fP - (\fP_\infty \cup \fP_j)$ lies on a unique line belonging to $C_j$.
\end{proof}

\subsubsection{Remark}
If we identify $\PP(A_1^*)$ with $\PP(S_1^*)$ according to the convention in \S\ref{conv.identification}  
the quadrics $Q_j$ become quadrics in $\PP(S_1^*)$. For example, in the $x_0,x_1,x_2,x_3$ coordinates, 
$Q_0$ is the zero locus of $x_0^2-\b\c x_1^2-\c\a x_2^2 +\a\b x_3^2$. It is reasonable to ask if these quadrics contain $E$.
They do not. For example, $Q_0$ does not contain the point $\tau'+\ve_1=(a,-ia,i,1)$.

\section{Points  on elliptic lines}
\label{se.pts}

In this section we determine which point modules are quotients of which elliptic line modules.

\subsection{}

Let $\xi \in E[2]$. As in \S\ref{sect.F.tau+xi} we write $F_{\tau+\xi}$ for $\wtV(\tau+\xi) \otimes k^2$.
By \Cref{prop.line.fat}, there are at most three 
modules $M_{x,\omega} \in \Gr(S')$ such that  $\Hom_{\Gr(S')}(M_{x,\omega},F_{\tau+\xi}) \ne 0$.

\begin{lemma}\label{pr.few_lines}
Fix a $\G$-equivariant structure on $F_{\tau+\xi} \in \Gr(S')$ and a non-zero morphism 
$f:M_{x,\omega} \to F_{\tau+\xi}$  in $\Gr(S')$.  There is at most one $\G$-equivariant structure on 
$M_{x,\omega}$ for which $f$ is $\G$-equivariant.
\end{lemma}
\begin{proof}
Let $f:M_{x,\omega} \to F_{\tau+\xi}$ be a non-zero morphism in $\Gr(S')$ that is $\G$-equivariant  for some $\G$-equivariant structure on $M_{x,\omega}$. 

Suppose $x=p+E[2]$ and let $\{\omega,\omega',\omega''\}$ be the 2-torsion points on $E$.
By \Cref{prop.L->F} and the equivalence $\Gr(S) \equiv \Gr(S')$, there are, up to scaling, unique non-zero  
morphisms $M_{p,p+\omega}\otimes k^2\to F_{\tau+\xi}$ and $M_{p+\omega',p+\omega''}\otimes k^2\to F_{\tau+\xi}$ in 
$\Gr(S')$. If these homomorphisms are the restrictions of $f$ then they are non-zero because $f$ is $\G$-equivariant.  
Therefore $f$ restricts to non-zero $\langle \omega \rangle$-equivariant maps 
$M_{p,p+\omega}\otimes k^2\to F_{\tau+\xi}$ and $M_{p+\omega',p+\omega''}\otimes k^2\to F_{\tau+\xi}$.
The ideas in the proof of \Cref{prop.Mpq.onto.V}(2) show that $\wtV(\tau+\xi)_0$ is the direct sum of 
the images of the degree zero
components of the non-zero maps $M_{p,p+\omega} \to \wtV(\tau+\xi)$ and $M_{p+\omega',p+\omega''} \to \wtV(\tau+\xi)$. 
 
But $f$  commutes with the action of $\omega$,  so it maps the $(+1)$- and $(-1)$-eigenspaces for the action of 
$\omega$ on the degree-zero component of $M_{x,\omega}\otimes k^2$ isomorphically to  the 
$(+1)$- and $(-1)$-eigenspaces for the action of $\omega$ on $(F_{\tau+\xi})_0$. 
This uniquely determines the $\G$-equivariant structure on $M_{x,\omega}$ because, 
by \cite[Thm. 11.6(3)]{CS15}, interchanging the $(+1)$- and $(-1)$-eigenspaces for $\omega$ on $M_{x,\omega}$ 
switches between the two equivariant structures. 
\end{proof}

\Cref{pr.few_lines} puts an upper bound on the number of  elliptic line modules in $\Gr(\wtQ)$ 
that map onto a given ordinary point module.
We will soon see that this bound is achieved.

Moreover, the situation is different for the four  special point modules.
If  $P \in \Gr(S)$ is a special point module, then $(P^{\oplus 2})\otimes k^2$  is a quotient of  exactly three 
$S'$-modules of the form $M_{x,\xi}$. Indeed, this is a simple count once we recall that $P$ is a quotient of $M_{p,q}$ precisely when $p+q$ is the element of $E[2]$ associated to $P$.

\begin{theorem}\label{th.pts_on_elliptic_lines}
Let $\xi \in E[2]-\{o\}$ and let $p \in \fP$. 
 \begin{enumerate}
      \item{} 
      If $p \in \fP-\fP_\infty$, then $p$ lies on exactly one line in the family parametrized by $E/\langle \xi \rangle$.
      \item{}
If $p \in \fP_\infty$, then $p$ lies on exactly two lines in the family parametrized by $E/\langle \xi \rangle$.
  \end{enumerate}
\end{theorem}
\begin{proof}
Let $P \in \Gr(S)$ be the point module corresponding to $p$.   

{(1)} 
If $p \in \fP_j$, let $F=\wtV(\tau+\xi_j) \otimes k^2$, and give $F$ the unique equivariant $S'$-module structure  such that $P\cong F^\G$.

By \Cref{prop.line.fat}, there is a unique $x \in E[2]$ such that $\Hom_{\Gr(S')}(M_{x,\xi}, F)\ne 0$.
Fix that $x$ and fix $q \in E$ such that $x=q+E[2]$. 
Let $f$ be a non-zero map in $\Hom_{\Gr(S')}(M_{x,\xi}, F )$.  
By \Cref{pr.few_lines}, there is at most one equivariant structure on $M_{x,\xi}$ for which $f$ is $\G$-equivariant. 
We give $M_{x,\xi}$ that equivariant structure.

Given that equivariant structure, every non-zero graded $S'$-module homomorphism from the summand 
$M_{q,q+\xi}\otimes k^2$ of $M_{x,\xi}$ to $F$ is equivariant for the $\xi$-action in degree zero. 
That it is equivariant in all degrees then follows from the fact that $M_{q,q+\xi}\otimes k^2$ is  a cyclic $S'$-module
because every homogeneous element 
in $M_{q,q+\xi}\otimes k^2$ can be obtained by acting on a degree-zero element by an eigenvector of $\xi$. 

A non-zero graded $S'$-module map from the other summand $M_{q+\xi',q+\xi''}$ of $M_{x,\xi}$ to $F $ 
can then be chosen uniquely so that the resulting map 
\begin{equation*}
  M_{x,\xi} = (M_{q,q+\xi} \oplus M_{q+\xi',q+\xi''})\otimes k^2\to F 
\end{equation*}
intertwines the $\G$-action at the degree-zero level. Once more it will then be equivariant in all degrees, since higher-degree homogeneous elements can be obtained by acting on degree-zero elements with eigenvectors of $\G$ in $S'$ and the map is an $S'$-module map.

{(2)} Fix equivariant structures on $M_{x,\xi}$ and $(P^{\oplus 2})\otimes k^2$. 
The degree-zero component of each module is isomorphic as a $\G$-module to the regular representation of $\G$, 
so there is a unique way to match up the respective eigenspaces. Moreover, there is, up to scaling, a unique 
non-zero morphism $M_{x,\xi}\to (P^{\oplus 2})\otimes k^2$
in  $\Gr(S')$ that agrees with this matching. This morphism must then be $\G$-equivariant as in the proof of part (1).  
\end{proof}

\section{Fat points on elliptic lines}
\label{se.fat_elliptic}

This section determines which of the fat point modules $M_p \otimes k^2$ described in \Cref{ssect.B_tilde} lie on which  elliptic lines. Thus we answer the following question:  if $L$ is an elliptic line module, when is $\Hom_{\Gr(A)}(L,M_p \otimes k^2) \ne 0$. If $f:L \to M_p \otimes k^2$ is a non-zero morphism in $\Gr(A)$, then $f$ becomes an epimorphism in $\QGr(A)$.
Recall that  $M_p \otimes k^2$ is a simple object in $\QGr(A)$ and corresponds to the skyscraper sheaf $\cO_{p+E[2]}$ under the equivalence $\QGr(\wtB) \equiv \qcoh(E/E[2])$. 

\subsection{}
Elliptic line modules are obtained by descent from the $\G$-equivariant structures on the $S'$-modules $M_{x,\xi}$.
To understand their relation to the modules $M_p \otimes k^2$ we need an equivariant description of  
$M_p \otimes k^2$ too. 

The $\wtQ$-module $M_p \otimes k^2$ is isomorphic to 
$N^\G$ for some $\G$-equivariant  $S'$-module $N=N'\otimes k^2$. 
Since $\omega^*M_p \cong M_{p+\omega}$ for all $\omega \in \G \equiv E[2]$, the obvious candidate for $N$ is the module  in 
(\ref{eq.module.N}).

\begin{lemma}\label{le.equiv_fat}
Let $p \in E$. There is a unique $\G$-equivariant structure on the $S'$-module 
\begin{equation}
\label{eq.module.N}
N =\bigoplus_{\omega \in E[2]} M_{p+\omega} \otimes k^2
\end{equation}
and $N^\G \cong M_p \otimes k^2$ as $A$-modules.
\end{lemma}
\begin{proof}
If $\omega \in E[2]$, then $\omega^*M_p \cong M_{p+\omega}$. Hence $\omega^*N \cong N$. 
Since the $S$-modules $M_p$ and $M_q$ are isomorphic if and only if $p=q$, $\Aut_{S'}(N)$ is isomorphic to $(k^\times)^4$. We label the elements of  $\Aut_{S'}(N)$ as 4-tuples $\l=(\l_\omega)_{\omega \in E[2]}$ with the convention that $\l$ acts on the summand 
$M_{p+\omega} \otimes k^2$ of $N$ as multiplication by $\l_\omega$.
Arguing as in \cite[Lem. 10.3]{CS15}, $\Aut_{S'}(N)$ is a $\G$-module with $\xi \in \G$ acting as follows:
$$
\xi \triangleright (\l_\omega)_{\omega \in E[2]} \; = \;  (\mu_\omega)_{\omega \in E[2]} 
$$
where $\mu_\omega = \l_{\omega+\xi}$. 

To prove the existence and uniqueness of a $\G$-equivariant structure on $N$ we argue as in \cite[Prop. 10.5]{CS15}.
Thus,  it suffices to show that 
\begin{equation}\label{eq.coh_vanish}
H^1\big(\G,\Aut_{S'}(N)\big) = H^2\big(\G,\Aut_{S'}(N)\big) =0  
\end{equation}
because the vanishing of $H^2$ implies the existence of at least one equivariant structure and, since
$H^1$ acts simply transitively on the set of equivariant structures, the vanishing of $H^1$ implies there is at most one equivariant structure.

To check \Cref{eq.coh_vanish} note that $\Aut_{S'}(N)$ can be identified with the group of $k^\times$-valued functions on $\G$, acted upon by $\G$ via translation on the domain of the functions. In other words, there is an isomorphism of $k\G$-modules
\begin{equation*}
  \Aut_{S'}(N)\cong \Hom_\bZ\big(k\G,(k^\times)^4\big)
\end{equation*}
 where, on the right, the $\bZ$-action is additive on $k\G$ and multiplicative on $(k^\times)^4$.  This is sometimes called a \define{relatively injective} $\G$-module, and its higher cohomology vanishes \cite[Prop. VII.2.1]{Ser79}.
 
The summands $M_{p+\omega}\otimes k^2$, $\omega \in E[2]$, 
are isomorphic to each other so the invariant part $N^\G\in \Gr(\wtQ)$ is isomorphic to $M_p\otimes k^2$. 
\end{proof}

Let $\xi  \in E[2]-\{o\}$ and let $x \in E/\langle \xi\rangle$.  The space $\Hom_{\Gr(S')}(M_{x,\xi},M_p \otimes k^2)$ is non-zero 
if and only if the image of $x$ in $E/E[2]$ equals $p+E[2]$ in which case it is isomorphic to $k^2$. 
Hence $\Hom_{\Gr(S')}(M_{x,\xi},N) \cong k^8$.
We would like some of these maps to be equivariant for at least one of the two $\G$-equivariant structures on $M_{x,\xi}$, in order to have them descend to $\wtQ$. The next result shows this is the case.

\begin{proposition}\label{pr.fat_pts_on_elliptic_lines}
If $p \in E$, then  $M_p\otimes k^2$ is a quotient in $\QGr(\wtQ)$ of exactly two line modules in each of the three elliptic families.   
\end{proposition}
\begin{proof}
Let $x=p+E[2]$ and fix $\xi\in E[2]$.
Let $N$ be the $\Gamma$-equivariant $S'$-module in (\ref{eq.module.N}) equipped with the $\G$-module structure in \Cref{le.equiv_fat}. 
We must show that for each of the two $\G$-equivariant structures on $M_{x,\xi}$ there is a non-zero 
homomorphism  $M_{x,\xi}\to N$ of $\G$-equivariant $S'$-modules. 

The actions of $\G$ on the $\G$-equivariant $S'$-modules $M_{x,\xi}$ and $N$ induce a $\G$-module structure on $\Hom_{\Gr(S')}(M_{x,\xi},N)$. 
We must show that $\Hom_{\Gr(S')}(M_{x,\xi},N)^\G \ne 0$.  

Since $M_{x,\xi}$ is generated in degree zero, $\Hom_{\Gr(S')}(M_{x,\xi},N)$ is isomorphic as a $\G$-module to $\Hom_{M_2(k)}((M_{x,\xi})_0,N_0)$. 
The fact that the trivial $\G$-representation occurs in it follows from the observation that both $(M_{x,\xi})_0$ and $N_0$ are $\G$-equivariant $M_2(k)$-modules and, by descent, the category of $\G$-equivariant $M_2(k)$-modules is equivalent to $\Vect$.
\end{proof}

Using this, we can determine how the group $E[4]$ of autoequivalences of $\Gr(\wtQ)$ acts on the modules 
$M_p\otimes k^2$. Let us first record the following observation.

\begin{proposition}\label{pr.E4_acts_B}
  The action of $E[4]$ on $\Gr(\wtQ)$ introduced in \Cref{ssect.E4_acts}  restricts to an action  of $E[4]$ as auto-equivalences of $\Gr(\wtB)$. 
\end{proposition}
\begin{proof}
By \Cref{pr.Omega_transl}, the auto-equivalence $a_{\ve_j}$ sends modules annihilated by $\Omega(z)$ to modules annihilated  by $\Omega(z+\xi_j)$. The result now follows from the fact that $\Gr(\wtB)$ consists of the $A$-modules that are annihilated by all $\Omega(z)$, $z\in E$.
\end{proof}

Since the modules $M_p\otimes k^2$ are, up to isomorphism, the only 1-critical $\wtB$-modules with Hilbert series
$2(1-t)^{-1}$, it follows from \Cref{pr.E4_acts_B} that $E[4]$ acts on the elliptic curve that parametrizes them.

\begin{proposition}\label{pr.E4_act_fat}
  The action of $E[4]$ on the set of 1-critical $\wtB$-modules with Hilbert series
$2(1-t)^{-1}$ induces an action of $E[4]$ on $E/E[2]$ as translations by  the image of $E[4]$ in $E/E[2]$. 
\end{proposition}
\begin{proof}
Fix $\xi\in E[2]$. By \Cref{pr.fat_pts_on_elliptic_lines}, $M_p\otimes k^2$ (whose isomorphism class only depends on $p+E[2]$) lies on both lines that comprise the preimage of $p+E[2]$ through $E/\langle 
\xi\rangle\to E/E[2]$ and only on those. By part (1) of \Cref{pr.E4_act_lines}, 
$E[4]$ acts as translations on the domain of this covering map, and hence it must act in the same fashion on its codomain $E/E[2]$.
\end{proof}

\section{Fat points on conic lines}
\label{subse.fat_conic}

This section determines the incidence relations between the ``lines'' in $\Projnc(\wtQ)$ parametrized by the conics $C_0,C_1,C_2,C_3$, and the ``points'' on the curve $E/E[2] \subseteq \Projnc(\wtQ)$.

\subsection{}
Since $\wtB$ is a quotient of $\wtQ$, 
\cite[Thm.1.2]{SPS15} shows that  $\Projnc(\wtB)$ is a closed subspace of 
$\Projnc(\wtQ)$ in the sense of \cite[\S3.3]{VdB-blowup} and \cite[Defn. 2.4]{SPS02}. 
Since $\QGr(\wtB) \equiv \qcoh(E/E[2])$ we say that 
$\Projnc(\wtB)$ is isomorphic to   $E/E[2]$ and think of $E/E[2]$ as a closed curve in the ambient non-commutative variety 
$\Projnc(\wtQ)$.

\begin{lemma}
\label{lem.commuting.subspaces}
Fix a point $p=(\d_0,\d_1,\d_2,\d_3) \in E \subseteq \PP(S_1^*)$. Let $f_p: \wtQ_1\to M_2(k)$ be the linear map
\begin{align*}
f_p(\l_0y_0+\l_1y_1+\l_2y_2+\l_3y_3) \; : = & \; 
\begin{pmatrix}
  \l_0\d_0+i\l_1\d_1   & i\l_2\d_2-\l_3 \d_3   \\
  i \l_2\d_2+\l_3 \d_3   &   \l_0 \d_0 -i \l_1\d_1  
\end{pmatrix}
\\
 \; =& \;  \l_0\d_0 q_0+  \l_1\d_1 q_1+  \l_2\d_2 q_2+ \l_3\d_3q_3.
\end{align*}
\begin{enumerate}
  \item 
  Let $y \in \wtQ_1$. 
If $e \otimes v \in (M_p \otimes k^2)_0 - \{0\}$,
then $y(e \otimes v)=0$ if and only if $f_p(y)v=0$.  
  \item 
$f_p$ is a linear isomorphism if and only if  $p \notin E[4]$.
  \item 
Suppose $p \notin E[4]$. As $kv$ varies over the points in $\PP(k^2)$, 
$f_p$ provides a bijection between the 2-dimensional subspaces
$ky+ky'$ of $\wtQ_1$ such that $\Hom_{\Gr(\wtQ)}(\wtQ/\wtQ y+\wtQ y',M_p \otimes k^2) \ne 0$ and the simple left ideals of $M_2(k)$.
\end{enumerate}
\end{lemma}
\begin{proof}
(1)
 There is a basis $e'$ for $(M_p)_1$ such that $x_je=\d_j e'$ for $j=0,1,2,3$.  
Since $\sum_{j=0}^3 \l_j y_j= \sum_{j=0}^3 \l_j  x_j \otimes q_j$,
$$
\Bigg(\sum_{j=0}^3 \l_j y_j\Bigg) (e \otimes v) \; = \; \sum_{j=0}^3 x_j e \otimes \l_jq_j v \; = \; e' \otimes f_p\Bigg(\sum_{j=0}^3 \l_j y_j\Bigg) v.
$$ 
Thus, $y \in \wtQ_1$ annihilates $e \otimes v$ if and only if $f_p(y) $ annihilates $v$. 

(2)
Since $\{1,q_1,q_2,q_3\}$ is linearly independent, $f_p$ is an isomorphism if and only if $\d_0\d_1\d_2\d_3\ne 0$. 
By \Cref{cor.E4}, $\d_0\d_1\d_2\d_3 = 0$ if and only if $p \in E[4]$. 

(3)
The correspondence $kv \longleftrightarrow \Ann_{M_2(k)}(v)$ is a bijection between the points in $\PP(k^2)=\PP^1$ and 
the simple left ideals of $M_2(k)$. The composed bijection $kv \longleftrightarrow \Ann_{M_2(k)}(v)
\stackrel{f_p}{ \longleftrightarrow} \Ann_{\wtQ_1}(e \otimes v)$ is a bijection between the points in $\PP(k^2)=\PP^1$
and 2-dimensional subspaces $ky+ky'$ of $\wtQ_1$ that annihilate a non-zero element in $(M_p \otimes k^2)_0$.  
\end{proof}

We now apply \Cref{lem.commuting.subspaces} with $p$ equal to $\tau':=(abc,a,b,c)$. By \Cref{prop.abc.a.b.c},   $2\tau'=-\tau$. 

\begin{proposition}
\label{prop.commuting.subspaces}
Let $f=f_{\tau'}: \wtQ_1\to M_2(k)$ be the  linear isomorphism defined  in \Cref{lem.commuting.subspaces}.
 \begin{enumerate}
  \item 
$f$ gives a bijection between commuting subspaces of $\wtQ_1$ and  simple left ideals in $M_2(k)$.
  \item 
Each commuting subspace of $\wtQ_1$ annihilates a unique 1-dimensional subspace of $(M_{\tau'} \otimes k^2)_0$
and each 1-dimensional subspace of $(M_{\tau'} \otimes k^2)_0$ is annihilated by a unique commuting subspace.
\item{}
If the fat point $\tau'+E[2]$ lies on a  ``line'' $L$, then $L$ is a commuting line.
\item{}
The only fat point $p+E[2]$ that lies on a commuting line is  $\tau'+E[2]$.
\end{enumerate}
\end{proposition}
\begin{proof}
(1)
Since $E[4]=E \cap\{x_0x_1x_2x_3=0\}$ and $abc \ne 0$,  
$f$ is a linear isomorphism by \Cref{lem.commuting.subspaces}(2). 

Let  $w_0,w_1,w_2,w_3$ be the basis for $\wtQ_1^*$ dual to the basis $y_0,y_1,y_2,y_3$. 
By \Cref{prop.comm.subspace.ruling},  the commuting subspaces of $\wtQ_1$ are the planes
$$
(abcw_0+iaw_1)-t(ibw_2-cw_3) =  t(abcw_0-iaw_1)-(ibw_2+cw_3) =0, \qquad t \in \PP^1.
$$
The union of these subspaces is the quadric  $\a\b\c w_0^2 + \a w_1^2 + \b w_2^2 +\c w_3^2=0$.

If $y=\sum \l_j y_j \in \wtQ_1$, then $\det f(y) = \a\b\c \l_0^2 + \a \l_1^2 + \b \l_2^2 +\c \l_3^2$.
Hence $f$ sends the quadric  $\a\b\c w_0^2 + \a w_1^2 + \b w_2^2 +\c w_3^2=0$ in $\PP(\wtQ_1)$ 
isomorphically to the quadric  $\{\det=0\}$ in $\PP(M_2(k))$ and sends the ruling by commuting subspaces
to one of the rulings on $\{\det=0\}$. One of the rulings on $\{\det=0\}$ is given by the simple left ideals in $M_2(k)$,
the other by  the simple right ideals. 

The commuting subspace spanned by $iy_0+bc y_1$ and $cy_2+iby_3$ is sent by $f$ to the linear span of 
$$
f(iy_0+bc y_1) = 
\begin{pmatrix}
2iabc  & 0   \\
0   &  0 
\end{pmatrix}
\qquad \hbox{and} \qquad 
f(cy_2+iby_3)=
\begin{pmatrix}
0  &0   \\
2ibc   & 0
\end{pmatrix}
$$
which is a {\it left} ideal. Hence $f$ sends commuting subspaces of $\wtQ_1$ to simple left ideals of $M_2(k)$.

(2)
This follows from \Cref{lem.commuting.subspaces}(1) because every simple left ideal in $M_2(k)$ annihilates 
a unique 1-dimensional subspace of $k^2$ and every 1-dimensional subspace of $k^2$ is annihilated by a unique simple left ideal.

(3)
Let $L$ be a line module for $\wtQ$ for which there is a non-zero homomorphism $\phi:L \to M_{\tau'} \otimes k^2$. 
There is a 2-dimensional subspace $ky+ky'\subseteq \wtQ_1$ such that $L=\wtQ/\wtQ y + \wtQ y'$. Hence $ky+ky'$
annihilates a non-zero element  $e \otimes v \in (M_{\tau'} \otimes k^2)_0$. By (2), $e \otimes v$ is also annihilated by 
a commuting subspace of $\wtQ_1$. However, by \cite[Prop. 8.6]{CS15}, $\wtQ_1(e \otimes v) = (M_{\tau'} \otimes k^2)_1$
which is 2-dimensional. It follows that $ky+ky'$ is the commuting subspace of $\wtQ_1$ that annihilates $e \otimes v$.
Thus, $L$ is a commuting line module.  

(4)
To show there are no other points in $E/E[2]$ that lie on a commuting line we must show that if $q \in E-(\tau'+E[2])$, then 
$\Hom_{\Gr(\wtQ)}(L,M_q \otimes k^2)=0$.   This follows from (2) and \Cref{lem.commuting.subspaces}(3). 
\end{proof}

\begin{theorem}\label{th.fat_conic}
For $j \in \{0,1,2,3\}$, $\tau'+\ve_j+E[2]$ is the unique fat point in  $\Projnc(\wtQ)$  that lies on all the lines in $\Projnc(\wtQ)$
that are parametrized by the conic $C_j$. There are no other incidence relations between conic lines and fat points in 
$E/E[2] =\Projnc(\wtB)$.   
\end{theorem}
\begin{proof}
By \Cref{prop.commuting.subspaces}, the result is true for the commuting conic $C_0$.

We use the quantum symmetry technique described in \Cref{se.other_conics} to transfer the result from $C_0$
to the other conics. By \Cref{pr.E4_acts_B}, the action of the autoequivalence group $E[4]$ permutes the conics 
$C_j$ or, more precisely, the line modules they parametrize, and preserves the variety $E/E[2]$ 
parametrizing fat $\wtB$-points. 
By \Cref{pr.E4_act_fat}, the four fat points $\tau'+\ve_j+E[2]$ form an orbit under the action of the two-torsion
subgroup $E[4]/E[2]$ of $E/E[2]$.
 \end{proof}

\begin{corollary}
\label{cor.ell_line=conic_line}
Let $(i,j) \in \{0,1,2,3\} \times \{1,2,3\}$.
The image of $E/\langle \xi_j \rangle$ in $\GG(1,3)$ meets $C_i$ at 2 points.
In total, 6 of the conic lines parametrized by $C_i$ are elliptic lines, two for each $E/\langle \xi_j \rangle$.
\end{corollary}
\begin{proof}
Let $p_i \in E$ be the point such that every line parametrized by $C_i$ passes through the fat 
point module $M_{p_i} \otimes k^2$. By \Cref{pr.fat_pts_on_elliptic_lines}, exactly two elliptic lines in each family 
$E/\langle \xi_j \rangle$ pass through that point. 

If $i=0$,  \Cref{prop.commuting.subspaces}(3) tells us those elliptic lines are, in fact, conic lines. 
Hence $C_0$ meets the image of $E/\langle \xi_j \rangle$ in $\GG(1,3)$ at exactly 2 points.  
By \Cref{pr.E4_act_lines}, the action of $E[4]$ as auto-equivalence of $\Gr(\wtQ)$ permutes the four conics and 
sends the set of line modules parametrized by $E/\langle \xi_j \rangle$ to itself. The result follows.
\end{proof}

\subsubsection{Remark}
The 24 points in $\GG(1,3)$ that belong to 
$$
\bigcup_{0 \le i \le 3 \atop 1 \le j \le 3} C_i \cap (E/\langle \xi_j \rangle)
$$
are given in Table \ref{E.cap.C}. Since $E/\langle \xi_j \rangle$ parametrizes lines in $\PP(S_1^*)$ of the 
form $\overline{p,p+\xi_j}$ we obtain, in this way,  24 distinguished secant lines in $\PP(S_1^*)$.

\section{Exact sequences arising from incidences}
\label{se.seq}

In order to get a more complete picture of the incidence geometry of $\wtQ$ we now identify the kernels of the surjections from elliptic  and conic lines to points and fat points.

Throughout this section we write $\G=\{\c_0,\c_1,\c_2,\c_3\}$ where $\c_0=1$ and $\c_j$ fixes $x_0$ and $x_j$.
As before, we identify $\G$ with $E[2]$.

\subsection{Ordinary points}
\label{subse.ordinary}

Throughout \Cref{subse.ordinary} we fix $j \in \{0,1,2,3\}$ and an ordinary point module $P\in \Gr(\wtQ)$ 
that corresponds to a point in $\fP_j$. Let $F=\wtV(\tau+\xi_j)$. Thus $F$ is a 1-critical $S$-module 
of multiplicity two and, by \Cref{cor.equi_fat}, there is a $\G$-equivariant structure on $F\otimes k^2$  such that 
$P\cong (F \otimes k^2)^\G$.  

If $p+q=\tau+\xi_j$, there is a non-zero homomorphism $M_{p,q} \to F$.  

By \Cref{th.pts_on_elliptic_lines}(1), $P$ is a quotient in $\QGr(\wtQ)$ of exactly three elliptic lines, one in each of the three elliptic families. By \Cref{cor.conics}, $P$ is also a quotient of exactly three conic lines, one in each of three conic families $C_i$, $i\ne j$.

The following result shows that the kernels of the resulting surjections from elliptic $\wtQ$-lines onto $P$ are shifted conic lines and vice versa. 

Recall $E[4]$ acts on the category $\Gr(\wtQ)$ and the restriction of the action to $E[2]=\G$ 
is induced by the action of $\G$ as automorphisms of $\wtQ$. For $\varepsilon\in E[4]$ we write $\varepsilon^*$ for corresponding autoequivalence. 

Throughout \Cref{subse.ordinary} whenever we work with some general element $\xi\in E[2]$ we denote the other two non-zero elements of $E[2]$ by $\xi'$ and $\xi''$.

\begin{proposition}\label{pr.ord_seq}
The kernels of the surjections from line modules onto the ordinary point modules are as follows. 
  \begin{enumerate}
  \item 
  If $P\in \fP_j$ is a quotient of a line module $M$ in the $E/\langle \xi\rangle$-family, there is an exact sequence 
        \begin{equation*}
          0\to L(-1)\to M\to P\to 0
        \end{equation*}
in which $L$ is the line module in the $C_k$-family that maps onto $\xi_k^*P$, where $\xi_k=\xi+\xi_j$.
      \item 
        If $P\in \fP_j$ is a quotient of a line module $L$ in the $C_i$-family,\footnote{Recall that this implies $i \ne j$.} 
        there is an exact sequence 
        \begin{equation*}
          0\to M(-1)\to L\to P\to 0
        \end{equation*}
in which $M$ is a line module in the $E/\langle \xi_i+\xi_j \rangle$-family that corresponds to a $\G$-equivariant 
structure on $(M_{p-2\tau,p+\xi-2\tau}\oplus M_{p+\xi'-2\tau,p+\xi''-2\tau})\otimes k^2$ where $\xi=\xi_i+\xi_j$ and 
$2p=\tau+\xi_k$. 
  \end{enumerate}  
\end{proposition}
\begin{proof}
{(1)} 
The kernel of the surjection $M \to P$ is equal to $L(-1)$ for some line module $L$.
There is a point  $p\in E$ such that $2p+\xi=\tau+\xi_j$ and $M$ corresponds to an equivariant structure on 
$(M_{p,p+\xi}\oplus M_{p+\xi',p+\xi''})\otimes k^2$. By the proof of \Cref{th.pts_on_elliptic_lines}, the 
surjection $M\to P$ comes from a $\G$-equivariant homomorphism of $S'$-modules
\begin{equation*}
  (M_{p,p+\xi}\oplus M_{p+\xi',p+\xi''})\otimes k^2\to F\otimes k^2.
\end{equation*}
Its kernel $K$ must be a $\G$-equivariant $S'$-module and $K^\Gamma \cong L(-1)$. 

{\bf Claim 1: $L$ is a conic line.} 

Suppose to the contrary that $L$ is an elliptic line module. Then $K
\cong (M_{q,q+\omega}\oplus M_{q+\omega',q+\omega''})\otimes k^2$ for some $q \in E$ 
where $\{\omega,\omega',\omega''\}=E[2]-\{o\}$. There would then be a non-zero morphism in $\Gr(S)$ from $M_{q,q+\omega}(-1)$ or $M_{q+\omega',q+\omega''}(-1)$ to either $M_{p,p+\xi}$ or $M_{p+\xi',p+\xi''}$, and its cokernel would be a 
point module. This, however, contradicts the fact that when $u+v\not\in E[2]$, which is the case for $\{u,v\}=\{p,p+\xi\}$ 
and $\{u,v\}=\{p+\xi',p+\xi''\}$,  the only point modules that are quotients of $M_{u,v}$ are those arising from the exact sequences
\begin{equation*}
  0\to M_{u+\tau,v-\tau}(-1)\to M_{u,v}\to M_u\to 0
\end{equation*}
and the analogous one for $M_v$. This proves Claim 1.

{\bf Claim 2: $L$ surjects onto $(\xi+\xi_j)^*P$.} 
By \Cref{pr.E4_act_lines}(1), $E[4]$ acts on $E/\langle\xi\rangle$ as translations by the image of 
$E[4]$ through $E\to E/\langle \xi\rangle$. Hence $\xi^*M \cong M$. Thus, applying $\xi^*$ to $M\to P$ results in a surjection $M\to \xi^*P$. Since $P$ is an ordinary point module, $\xi^*P \not\cong P$. The composition
$L(-1)\to M\to \xi^*P$  is therefore non-zero. Hence there is a surjection $L \to (\xi^*P)(1)_{\ge 0} \cong \c_j^*\xi^*P$
(see the remark after \Cref{eq.theta} ). Thus, $L$ maps onto $(\xi+\xi_j)^*P$ as claimed.

{\bf Claim 3: $L$ belongs to the conic family $C_k$, where $\xi_k=\xi+\xi_j$.} 
Let $C_k$ be the conic family to which $L$ belongs. We will show that $\xi_k=\xi+\xi_j$.

Since $P$ is in $\fP_j$, it is annihilated by $\Omega(\tau+\xi_j)$. 
Since the line module $M$ is annihilated by a unique $\Omega(z)$, it must be annihilated by 
$\Omega(\tau+\xi_j)$. Therefore, if $\Omega=\Omega(z) \ne \Omega(\tau+\xi_j)$,  $M/\Omega M$ is a $\wtB$-module with
Hilbert series $(1-t^2)(1-t)^{-2}$ so is isomorphic in $\QGr(\wtQ)$ to $M_q\otimes k^2$ for some $q\in E$. 
As in the proof of Claim 2,  the composition $L(-1)\to M\to M_q\otimes k^2$ leads to a 
surjection $L \to (M_q \otimes k^2)(1)_{\ge 0} \cong M_{q-\tau}\otimes k^2$. 
Since $L$ is in $C_k$, \Cref{th.fat_conic} tells us that $q-\tau \in \tau'+\ve_k+E[2]$.
By \Cref{prop.abc.a.b.c}, $2\tau'=-\tau$ so $2q=\tau+\xi_k$. 
By  \Cref{pr.fat_pts_on_elliptic_lines} and the remarks before it, since $M_q\otimes k^2$ is a quotient of
$M$, $q \in p+E[2]$. But $2p=\tau+\xi_j+\xi$ so $\tau+\xi_k=2q=\tau+\xi+\xi_j$.
 The claim follows.

{(2)} 
Consider for a moment   the proof of (1). The $\langle\xi\rangle$-equivariant morphisms $M_{p,p+\xi}\otimes k^2\to F\otimes k^2$ and $M_{p+\xi',p+\xi''}\otimes k^2\to F\otimes k^2$ give rise to $\langle\xi\rangle$-equivariant structures on the kernels which are, by \Cref{prop.L->F}, $M_{p-2\tau,p+\xi-2\tau}(-2)\otimes k^2$
and $M_{p+\xi'-2\tau,p+\xi''-2\tau}(-2)\otimes k^2$
respectively. The action of $\xi'$ then interchanges these two kernels, giving rise to a $\G$-equivariant structure on their direct sum and hence to a shifted elliptic line module $M(-2)$, with $M$ as in the statement of (2). 

We have an embedding $M(-2)\subset L(-1)$ and so $L/M(-1)$ is a point module 
which must necessarily coincide with $P$ since a conic line module admits a single surjection onto a 
point module, up to isomorphism (\Cref{cor.conics}).  

To see that {\it every} surjection $L\to P$ as in the statement of (2) fits into this framework simply observe that the group $E[4]$ of autoequivalences of $\Gr(\wtQ)$ acts transitively on the set of isomorphism classes of such surjections. 
\end{proof}

\subsection{Special points}
\label{subse.special}

Let $P$ be a special point module. 
By \Cref{th.pts_on_elliptic_lines}, if $L$ is a conic line, then $\Hom_{\Gr(\wtQ)}(L,P)=0$. 
In \S\ref{subse.special} we describe the kernels of the surjections $M \to P$ when $M$ is an elliptic line module.  

Under the equivalence between $\Gr(\wtQ)$ and $\Gr(S')^\G$,  $\{P \; | \; P \in \fP_\infty\}$  corresponds to 
$$
\big\{(M_e \oplus M_e)\otimes k^2 \in \Gr(S')^\G \; \big\vert \;  M_e \; \hbox{is a special point module for } \; S\big\}.
$$  
Each $M_e$ is annihilated by three of the four $x_j$'s so 
each special point module for $\wtQ$ is annihilated by three of the $y_j$'s.

\begin{proposition}\label{pr.spec_seq}
Let $M_x$ be the line module corresponding to a point $x \in E/\langle \xi_i\rangle$. 
Let $P\in \Gr(\wtQ)$ be the special point module that is annihilated by $\{y_0,y_1,y_2,y_3\}-\{y_j\}$.
If $f:M_x\to P$ is a surjection in $\Gr(\wtQ)$, then
\begin{equation}
\label{ker.elliptic.line.to.fat.pt}
\ker(f:M_x \to P) \;  \cong \;
\begin{cases} 
 M_{x-\tau}(-1) & \text{if $j\not\in\{i,0\}$,}
 \\
 M_{x+\xi_i'-\tau}(-1)& \text{if $j\in\{i,0\}$.}
 \end{cases}
 \end{equation}
\end{proposition}
\begin{proof}
Since $i$ and $j$ are fixed we drop the index $i$ in $\xi_i$ and write  $E[2]=\{o,\xi=\xi_i,\xi',\xi''\}$, and we write 
$M_{e}$ for the point module in $\Gr(S)$ that is annihilated by $\{x_0,x_1,x_2,x_3\}-\{x_j\}$ (see \Cref{se.pts}).

Let $p \in E$ be one of the two preimages of $x$.

The homomorphism $f:M_x \to P$ corresponds to a morphism 
\begin{equation}\label{eq.spec_seq1}
  (M_{p,p+\xi}\oplus M_{p+\xi',p+\xi''})\otimes k^2\to (M_e\oplus M_e)\otimes k^2,
\end{equation}
of $\G$-equivariant $S'$-modules. 
Let $K$ denote the kernel and $M$ the domain of this map. Thus, $\ker(f)=K^\Gamma$ and $M_x =M^\Gamma$. 

By \cite[Thm. 5.7]{LS93}, the kernel of every surjective homomorphism $M_{p,p+\xi} \to M_e$ in $\Gr(S)$ is isomorphic to $M_{p-\tau,p+\xi-\tau}(-1)$.
 Thus, $K$ is isomorphic to the $S'$-module $(M_{p-\tau,p+\xi-\tau}\oplus M_{p+\xi'-\tau,p+\xi''-\tau})\otimes k^2$ equipped with one of its two 
 $\G$-equivariant structures. We will show that the two possibilities correspond to the two cases in 
 (\ref{ker.elliptic.line.to.fat.pt}). We will prove only the first statement, $j\not\in\{0,i\}$; the other case is entirely analogous. 

Since $M_{p,p+\xi}$ surjects onto $M_e$, the subspace of $\wtQ_1$ that annihilates  $(M_{p,p+\xi})_0$  
is contained in the linear 
span of $\{x_0,x_1,x_2,x_3\}-\{x_j\}$. By \cite[Lemma 11.7]{CS15}, this subspace is spanned by an element of the form $\beta_0x_0+\beta_ix_i$ and $x_k$ where $k\not\in\{0,i,j\}$. Moreover, since $(M_{p+\xi',p+\xi''})_0$ is annihilated by $\beta_0x_0-\beta_ix_i$ and $x_k$, $\beta_0$ and $\beta_i$ are both non-zero. 

It follows that $x_0$ annihilates $(M_e)_0$ but not $(M_{p,p+\xi})_0$ or $(M_{p+\xi',p+\xi''})_0$. 
Hence $K_1=x_0M_0$.  

Let $q +\langle \xi \rangle \in E/\langle \xi \rangle$.
As in \cite[$\S$11.4]{CS15}, we define the line module corresponding to $q +\langle \xi \rangle$ as the 
$\G$-invariants for the equivariant structure on $(M_{q,q+\xi}\oplus M_{q+\xi',q+\xi''})\otimes k^2$  
in which $\xi_i$ fixes $(M_{q,q+\xi})_0\otimes {0 \choose 1}$ and changes the sign of $(M_{q+\xi',q+\xi''})_0\otimes {0 \choose 1}$. Since $x_0$ is fixed by $\G$,  $K^\G \cong M_{x-\tau}(-1)$.
\end{proof}

\subsection{Fat point modules}
\label{subse.fat}

As before, $M_p\otimes k^2$ is the fat point module for $\wtB$ corresponding to $p+E[2]$.  If $L$ is a line module and $f:L
\to M_p \otimes k^2$ is non-zero we will show that $\ker(f) \cong L(-2)$.

\begin{definition}
The action of a group $\G$ on a category $\cC$ is \define{weakly trivial} if $\gamma^*c\cong c$ for all $\c \in \G$ and
all $c\in\cC$. 
\end{definition}

\begin{lemma}\label{le.wtriv}
  The action of $\G$ on $\QGr(\wtB)$ is weakly trivial. 
\end{lemma}
\begin{proof}
By \cite[$\S$8]{CS15}, the fat point module $M_p\otimes k^2$ corresponds to the skyscraper sheaf $\cO_{p+E[2]}$
under the equivalence $\QGr(\wtB)\equiv \QCoh(E/E[2])$.

Let $\g\in\G$. By \cite[Theorem 5.4]{Bra13} the autoequivalence $\g^*$ of $\QCoh(E/E[2])$ is isomorphic to 
$\alpha^*\circ (\cL\otimes -)$ for some automorphism $\alpha:E/E[2]\to E/E[2]$ and some line bundle
$\cL$. By \Cref{pr.E4_act_fat}, $\c^*(M_p\otimes k^2) \cong M_p\otimes k^2$; since $\cL \otimes  \cO_{p+E[2]}
\cong \cO_{p+E[2]}$, $\alpha$ must be the identity. 

However, $\gamma^*$ also preserves the isomorphism class of $\wtB\in \QGr(\wtB)\equiv \QCoh(E/E[2])$. Since the latter object is identified with the non-trivial self-extension $\cV$ of $\cO_{E/E[2]}$ (\cite[Prop. 8.4]{CS15}), $\gamma^*\wtB$ is a non-trivial self-extension of $\cL$. This can only be isomorphic to $\cV$ when $\cL\cong \cO_{E/E[2]}$. 
\end{proof}

\begin{proposition}\label{pr.B_no_conic}
  No line module for $A$ is a $\wtB$-module. 
\end{proposition}
\begin{proof}
 By \Cref{pr.E4_act_lines}, no points of the line scheme are fixed by $\G$. 
 Thus, if $L$ is a line module $\c^*L \not\cong L$. Hence, as objects in $\QGr(A)$, $\c^*L \not\cong L$. 
 It therefore follows from \Cref{le.wtriv} that $L$ in not a $\wtB$-module. 
\end{proof}

\begin{theorem}\label{pr.fat_seq}
If $L$ is a line module, then the kernel of every non-zero morphism $L\to M_p\otimes k^2$ 
is isomorphic to $L(-2)$. 
\end{theorem} 
\begin{proof} 
It follows from \Cref{pr.B_no_conic}  that there is a non-zero degree-two central element $\Omega$ that does not annihilate $L$ and hence acts faithfully on it. But $\Omega$ annihilates $M_p\otimes k^2$ so the image of 
$L\to M_p\otimes k^2$ must be isomorphic to $L/\Omega L$ by a Hilbert series comparison. In conclusion, the kernel of $L\to M_p\otimes k^2$ is $\Omega L\cong L(-2)$.
\end{proof}

\vskip .3in

\appendix
\section{The components of the line scheme}
\begin{center}
\textsc{Derek Tomlin}\footnote{derek.tomlin@mavs.uta.edu,
Department of Mathematics, P.O.~Box 19408,
University of Texas at Arlington,
Arlington, TX 76019-0408}
\end{center}

We use the method  in \cite{ShV02_bis} to determine the reduced line scheme, $\mathbb{L}_{\rm red}$, for the algebras 
$A$ defined in \S 3.1. 
For more detail about these methods the reader is referred to \cite{ChV}.

We use the coordinate functions $y_0,y_1,y_2,y_3$ on $\PP(A_1^*)$ used earlier in this paper.   
We associate to the line  in $\PP(A_1^*)$ through different points $(a_0,\ldots,a_3)$ and $(b_0,\ldots,b_3)$ the 
point $(X_{01},X_{02},X_{03},X_{12},$ $X_{13},X_{23}) \in \PP^5$ where $X_{ij}:=a_ib_j-a_jb_i$. Thus $X_{ij}$ is a 
$2 \times 2$ minor  of  
\[
\begin{pmatrix}
a_0 & a_1 & a_2 & a_3 \\[2mm]
b_0 & b_1 & b_2 & b_3 
\end{pmatrix}.
\]
The Pl\"ucker polynomial $X_{01}X_{23}-X_{02}X_{13}+X_{03}X_{12}$ vanishes on this matrix.
The Grassmannian of lines in $\PP(A_1^*)$ is the zero locus in $\PP^5$ of the Pl\"ucker polynomial.


\begin{theorem}\label{thm.D}\cite{DT}
The line variety, $\LL_{\rm red},$ of  $A $ is the union of four plane conics and three quartic elliptic curves:
\begin{enumerate}
\item[$\L_{0}$:] $ \;  X_{01}+\alpha X_{23} \;=\; X_{02}- \beta X_{13} \;=\; X_{03}+ \gamma X_{12}
 \;=\; X_{01}^2 + \alpha \gamma X_{12}^2 +\alpha \beta X_{13}^2\,=\,0.    \phantom{\Big)}$ 

\item[$\L_{1}$:]$ \; X_{01}-\alpha X_{23} \;=\; X_{02}+X_{13}\phantom{x} \;=\;  X_{03}+ X_{12}
\phantom{x} \;=\;  X_{01}^2 - \alpha X_{12}^2 +\alpha X_{13}^2 \phantom{xx}\, =\,0.   \phantom{\Big)}$ 

\item[$\L_{2}$:] $  \;  X_{01}+X_{23} \phantom{x} \;=\; X_{02}+\beta X_{13} \;=\; X_{03}-X_{12}
\phantom{x} \;=\; X_{01}^2 - X_{12}^2 - \beta X_{13}^2 \phantom{xx}\,\, \,\, \,=\,0.   \phantom{\Big)}$ 
	              
\item[$\L_{3}$:] $\;  X_{01}-X_{23} \phantom{x}\,\,=\,X_{02}-X_{13}\phantom{x} \,\,=\,X_{03}-\gamma X_{12}
\;\,= \; X_{01}^2+\c X_{12}^2-X_{13}^2 \; \phantom{xx}\,\,\,\,  = \;0.   \phantom{\Big)} $ 
            
\item[$E_{1}$:] $\; X_{01}\,=\,X_{23}\,=\,X_{02}X_{13}-X_{03}X_{12}   \phantom{\Big)} $ 

$\phantom{xxxxxxxx}
 =\,(1+\gamma) X_{02}^2 - (1-\beta) X_{03}^2 - \gamma (1-\beta) X_{12}^2 - \beta(1+\gamma) X_{13}^2\,=\,0.   \phantom{\Big)}$ 

\item[$E_{2}$:] $ \; X_{02}\,=\,X_{13}\,=\,X_{01}X_{23}+X_{03}X_{12}\,   \phantom{\Big)}$

$\phantom{xxxxxxxx}=\,(1-\gamma)X_{01}^2 - (1+\alpha)X_{03}^2+\gamma(1+\alpha)X_{12}^2 +\alpha(1-\gamma)X_{23}^2\,=\,0.   \phantom{\Big)}$ 
	       
\item[$E_{3}$:] $\; X_{03}\,=\,X_{12}\,=\,X_{01}X_{23}-X_{02}X_{13}\,   \phantom{\Big)}$

$\phantom{xxxxxxxx}
 =\,( 1 + \beta)X_{01}^2 - (1-\alpha)X_{02}^2 - \beta(1-\alpha) X_{13}^2 - \alpha(1+\beta) X_{23}^2\,=\,0. \phantom{\Big)}$
\end{enumerate}
\end{theorem}

There is an isomorphism from $\GG(1,3)$ with coordinate functions $X_{pq}$ to $\GG(1,3)$ with coordinate functions $z_{rs}$, which are used in 
\S\S\ref{se.commuting_conic} and \ref{sect.quantum.symm}, given by $X_{pq} \mapsto z_{rs}$ whenever ${0\, \, 1 \,\, 2 \,\, 3 \choose p\, \, q \,\, r \,\, s }$ is an even permutation.
In the $z$-coordinates, the equations cutting out the four conics are  
\begin{enumerate}
\item[$\L_{0}$:] $ \quad  z_{23}+\alpha z_{01} \;=\; z_{13}- \beta z_{02} \;=\; z_{12}+ \gamma z_{03}
 \;=\;  \a z_{01}^2 + \beta z_{02}^2+  \gamma z_{03}^2 \;=\;0;   \phantom{\Big)}$ 

\item[$\L_{1}$:]$ \quad z_{23}-\alpha z_{01}  \;=\; z_{13}+z_{02} \; \,\, \,\;=\;  z_{12}+ z_{03} \,
\;  \;=\;    \a z_{01}^2 + z_{02}^2 -  z_{03}^2 \phantom{xx}  \;=\;0;   \phantom{\Big)}$ 

\item[$\L_{2}$:] $  \quad  z_{23}+z_{01}\; \; \, \;=\; z_{13}+\beta z_{02} \;=\; z_{12}-z_{03} \;
 \;\;=\;  -z_{01}^2  + \beta z_{02}^2 + z_{03}^2\phantom{i} \;=\;0;   \phantom{\Big)}$ 
	              
\item[$\L_{3}$:] $\quad  z_{23}-z_{01}\; \; \;\,=\,z_{13}-z_{02}\,\;\;\,=\,z_{12}-\gamma z_{03}
 \;\;=\;  z_{01}^2  -  z_{02}^2 +\gamma z_{03}^2 \phantom{xx} \;=\;0   \phantom{\Big)}$. 
\end{enumerate}
These are the same as the equations for the conics in \Cref{pr.theyre_conics},  so $\L_i=C_i$ for $i=0,1,2,3$.

In terms of the $z$-coordinates, the three elliptic curves are 
\begin{enumerate}
\item[$E_{1}$:] $\; z_{23}\,=\,z_{01}\,=\,z_{13}z_{02}-z_{12}z_{03} \phantom{\Big)}$

$\phantom{xxxxixi}
 =\,(1+\gamma) z_{13}^2 - (1-\beta) z_{12}^2 - \gamma (1-\beta) z_{03}^2 - \beta(1+\gamma) z_{02}^2\,=\,0; \phantom{\Big)}$ 

\item[$E_{2}$:] $ \; z_{13}\,=\,z_{02}\,=\,z_{23}z_{01}+z_{12}z_{03}\, \phantom{\Big)}$

$\phantom{xxixxxi}=\,(1-\gamma)z_{23}^2 - (1+\alpha)z_{12}^2+\gamma(1+\alpha)z_{03}^2 +\alpha(1-\gamma)z_{01}^2\,=\,0; \phantom{\Big)}$ 
	       
\item[$E_{3}$:] $\; z_{12}\,=\,z_{03}\,=\,z_{23}z_{01}-z_{13}z_{02}\, \phantom{\Big)}$

$\phantom{xxxixxi}
 =\,( 1 + \beta)z_{23}^2 - (1-\alpha)z_{13}^2 - \beta(1-\alpha) z_{02}^2 - \alpha(1+\beta) z_{01}^2\,=\,0. \phantom{\Big)}$
\end{enumerate}
By \Cref{prop.Ej}, $E_j \cong E/\langle \xi_j \rangle$.

\begin{corollary}\cite{DT}
Each $C_i$, $i=0,\,1,\,2,\,3,$ intersects each $E_j$,  $j=1,2,3$, at two points.
With respect to the ordered coordinate functions  $(z_{01},z_{02},z_{03},z_{12},z_{13},z_{23})$, 
the two points in $C_i \cap E_j$ are those in row $C_i$ and column $E_j$ of Table   \ref{E.cap.C}.

\begin{table}[htp]
\begin{center}
\begin{tabular}{|l||l|l|l|l|l|}
\hline
  & $\phantom{xxxxxxxx} E_1$ & $\phantom{xxxxxxxx} E_2$ & $\phantom{xxxxxxx\Big)}E_3 $
\\
\hline
\hline
$C_0$ &   $ (\, 0, \,  c ,\, \pm i b ,\,   \, \mp i b\c  , \, \beta c ,0  \,) $ & $(\,c,\,0,\,\pm ia, \,\mp i a\c,\, 0,\, -\alpha  c\,) $ &  
$(\, b , \, \pm i a ,\, 0 ,\,0 \, \pm i a\b , \,  -\alpha b \,) $     $\phantom{\Big\vert}$ 
 \\
 \hline
 $C_1$ &   $ (\, 0 , \, 1 ,\, \pm 1, \, \mp 1 , \,  -1 , \, 0  \,) $ & $(\,1 ,\,0,\,\pm a,\,\mp a,\, 0,\,\alpha\,) $ &  
 $(\,  1, \, \pm i a ,\, 0 ,\, 0 , \, \mp i a , \, \alpha  \,)  $     $\phantom{\Big\vert}$ 
  \\
\hline
$ C_2$ &  $(\, 0 , \,   1 ,\, \pm i b , \, \pm ib , \, -   \b , \, 0  \,) $    $\phantom{\Big\vert}$   &  $ (\,1,\,0,\,\pm 1,\,\pm 1,\,0,\,-1\,) $   &  
 $(\, b , \, \pm 1 ,\, 0 ,\, 0 , \, \mp \b , \, -b   \,)$  \\
\hline
$C_3$ &  $(\, 0 , \,  c , \, \pm 1 , \, \pm \c, \, c ,\, 0  \,) $   & $(\,c,\,0,\, \pm i, \, \pm i \c,\,0,\, c\,)$&  
$(\, 1 , \, \pm 1 , \, 0 , \, 0 , \, \pm 1 ,\, 1  \,)$    $\phantom{\Big\vert}$ 
\\
\hline
\end{tabular}
\end{center}
\vskip .12in
\caption{Intersection points $C_i \cap E_j$.}
\label{E.cap.C}
\end{table}
\end{corollary}

\bibliography{biblio}{}
\bibliographystyle{plain}

\end{document}